\documentclass[a4paper,10pt,reqno]{amsart}
\usepackage{amsmath,amssymb,amsfonts,amsthm,ascmac,array,bm,multirow}
\usepackage[dvipdfmx]{graphicx}
\usepackage[dvipsnames]{xcolor}
\usepackage{lscape}

\usepackage[
colorlinks=true,bookmarks=true,citecolor=blue,
bookmarksnumbered=true,bookmarkstype=toc,linktocpage=true
]{hyperref}

\allowdisplaybreaks
\numberwithin{equation}{section}

\usepackage{mathrsfs}

\usepackage{mathtools}
\mathtoolsset{showonlyrefs=true}

\newtheorem{thm}{Theorem}[section]

\newtheorem{lem}[thm]{Lemma}

\newtheorem{cor}[thm]{Corollary}

\newtheorem{rem}[thm]{Remark}
\newtheorem{ass}[thm]{Assumption}

\newcommand{\mcc}{\mathcal{C}}

\newcommand{\mcf}{\mathcal{F}}
\newcommand{\mci}{\mathcal{I}}

\newcommand{\mcl}{\mathcal{L}}

\newcommand{\mbbh}{\mathbb{H}}
\newcommand{\mbbi}{\mathbb{I}}

\newcommand{\mbbn}{\mathbb{N}}
\newcommand{\mbbr}{\mathbb{R}}

\newcommand{\mbby}{\mathbb{Y}}

\newcommand{\al}{\alpha}
\newcommand{\del}{\delta} 
\newcommand{\vp}{\varphi}
\newcommand{\ve}{\varepsilon}

\newcommand{\sig}{\sigma}
\newcommand{\ep}{\epsilon}
\newcommand{\D}{\Delta}
\newcommand{\Sig}{\Sigma}
\newcommand{\lam}{\lambda}

\newcommand{\Ups}{\Upsilon}
\newcommand{\gam}{\gamma}
\newcommand{\Gam}{\Gamma}
\newcommand{\p}{\partial}

\newcommand{\cil}{\xrightarrow{\mcl}} 
\newcommand{\scl}{\xrightarrow{\mcl_{s}}} 
\newcommand{\cip}{\xrightarrow{p}} 

\newcommand{\argmax}{\mathop{\rm argmax}}
 
\newcommand{\trace}{\mathop{\rm trace}} 

\def\ds#1{\displaystyle{#1}} 

\def\nn{\nonumber}
\def\cadlag{c\`adl\`ag}
\def\wp{Wiener process}
\def\cpp{compound-Poisson process}
\def\lp{L\'{e}vy process}
\def\lm{L\'{e}vy measure}

\def\sumj{\sum_{j=1}^{n}}

\def\pr{P}
\def\E{E}

\def\hmrev#1{\textcolor{black}{#1}} 
\def\serev#1{\textcolor{black}{#1}} 
\def\frev#1{\textcolor{black}{#1}} 

\def\tz{\theta_{0}}

\def\tes{\hat{\theta}_{n}}

\newcommand{\ly}{Y^\star} 
\newcommand{\oy}{Y} 
\newcommand{\lx}{X^\star} 
\newcommand{\ox}{X} 
\newcommand{\nY}{\overline{\mathsf{y}}} 
\newcommand{\cx}{\check{X}} 

\title[Robustified Gaussian quasi-likelihood inference for volatility]
{Robustified Gaussian quasi-likelihood inference for volatility}

\author[S. Eguchi]{Shoichi Eguchi}
\address{Faculty of Information Science and Technology, Osaka Institute of Technology, 1-79-1 Kitayama, Hirakata City, Osaka, 573-0196, Japan.}
\email{shoichi.eguchi@oit.ac.jp}

\author[H. Masuda]{Hiroki Masuda}
\address{Graduate School of Mathematical Sciences, University of Tokyo, 3-8-1 Komaba Meguro-ku Tokyo 153-8914, Japan.}
\email{hmasuda@ms.u-tokyo.ac.jp}

\date{\today}

\begin{document}
\setlength{\baselineskip}{4.5mm}

\begin{abstract}
We consider statistical inference for a class of continuous semimartingale regression models based on high-frequency observations subject to contamination by finite-activity jumps and spike noise. By employing density-power weighting and H\"{o}lder-inequality-based normalization, we propose easy-to-implement, robustified versions of the conventional Gaussian quasi-maximum-likelihood estimator that require only a single tuning parameter. We prove their asymptotic mixed normality at the standard rate of $\sqrt{n}$. It is theoretically shown that these estimators are simultaneously robust against contamination in both the covariate and response processes. Additionally, under suitable conditions on the selection of the tuning parameter, the proposed estimators achieve the same asymptotic distribution as the conventional estimator in the contamination-free case. Illustrative simulation results highlight the estimators' insensitivity to the choice of the tuning parameter.
\end{abstract}

\maketitle

\section{Introduction}

Suppose that we want to estimate the parametric diffusion coefficient of the $d$-dimensional continuous It\^{o} semimartingale $(X,Y)$ where
\begin{align}
Y_t &= Y_0 + \int_0^t \mu_s ds + \int_0^t \sig(X_{s},\theta)dw_s,
\label{hm:csm.reg} \\
X_t &= X_0 + \int_0^t \mu'_{s} ds + \int_0^t \sig'_{s}dw'_s
\label{hm:csm.reg.X}
\end{align}
based on a discrete-time sample $\{(X_{t_j},Y_{t_j})\}_{j=0}^{n}$ with $t_j=t_j^n:=jT/n$ for a fixed $T>0$, when finite-activity jumps and spike noises contaminate the process, 
\hmrev{where $w$ and $w'=(w,w^\dagger)$ are standard 
{\wp es} in $\mbbr^r$ and $\mbbr^{r'}$ respectively, with $w^\dagger$ denoting a standard {\wp} in $\mbbr^{r^\dagger}$ independent of $w$ ($r'=r+r^\dagger$);} 
we will model spike noises, which appear in several application fields, such as echo-planar imaging-based diffusion tensor imaging \cite{ChaStoGra09}, using \eqref{hm:Y.spike} and \eqref{hm:X.spike} below.
In the absence of discontinuous components, the conventional Gaussian quasi-(log-)likelihood function based on the Euler scheme (GQLF; see \eqref{hm:GQLF} below) is efficient within the non-ergodic framework, treating the drift function $\mu=(\mu_t)$ as a nuisance parameter. 
We refer to \cite{GenJac93}, \cite{Gob01} and \cite{UchYos13} for related details for the efficient results in the case of clean continuous processes.
However, it is well-known that the GQLF is sensitive to contamination from discontinuous variations.

The most popular estimation strategy for removing jumps is the threshold estimation. Related previous studies about parametric inference include the following.
\begin{itemize}
\item \textit{Local-threshold estimation \cite{ShiYos06}, \cite{Shi08}, and \cite{OgiYos11} for ergodic diffusion with finite-activity jumps.}
They classified the jumps by looking at only the increment sizes of $Y$ and entirely or partly involving a jump-detection filter of the form
\begin{equation}
|Y_{t_j}-Y_{t_{j-1}}| > (t_j-t_{j-1})^{\rho} C
\label{hm:local.filter}
\end{equation}
for some user-input constants $C,\rho>0$. We also refer to \cite{AmoGlo21} for a contrast function based on a smoothed version of the indicator function of the event \eqref{hm:local.filter}.
\item \textit{Global-threshold estimation for non-ergodic regression with jumps.}
The previous work \cite{InaYos21} proposed the global jump detection filter, which removes a portion of large increments from the Gaussian quasi-likelihood in an appropriate manner. 
They proved the asymptotic mixed normality and the uniform tail-probability estimate for the associated statistical random fields. 
Their estimation procedure requires several user-specified elements, including a preliminary estimator for the spot volatility and a threshold sequence to filter out large increments.
\end{itemize}
The threshold estimators should perform efficiently if the threshold is appropriately chosen, although choosing the threshold is a delicate problem. Indeed, calibration of the best filter is a non-asymptotic problem with considerable degrees of freedom; the meaning of ``best'' itself encompasses several aspects, including unbiasedness and/or minimum variance in a finite sample, computational efficiency (stability), and so on.
Relatedly, we also refer to \cite{MasUeh21} for estimation of ergodic diffusions with compound-Poisson jumps based on testing the presence of jumps through the self-normalized residuals; the strategy does not involve any thresholding, but requires repeated computations after removing the largest increments until the rejection of the presence of jumps is terminated.
While the estimation procedure of \cite{MasUeh21} is easy to implement and hence practical, the theoretical consideration was made only for a limited class of coefficients.

The primary objective of this paper is to present robust estimation methods that do not involve a precise fine-tuning filter of the form \eqref{hm:local.filter}. Specifically, we propose modifications of the GQLF through the density-power divergence and the H\"{o}lder inequality, \hmrev{and demonstrate their robustness properties both theoretically and numerically.} 
To the best of our knowledge, the only previous studies on the robust divergence-based inference for SDE models are \cite{LeeSon13}, \cite{Son17}, and \cite{Son20}, all of which are concerned with ergodic Markovian diffusions without theoretical consideration in the presence of the contaminations.
On the other hand, we will theoretically show that the simple density-power tapering and H\"{o}lder-based normalization automatically remove discontinuous contaminations under the high-frequency sampling scheme, thus providing us with a handy practical alternative to threshold estimation. 

Unlike the threshold estimation mentioned above, our asymptotic results ensure asymptotic mixed normality for each fixed tuning parameter, which we will denote by $\lam>0$. We also prove that, with a suitable control of the tuning parameter, we can derive the ``clean'' asymptotic mixed normality in which the asymptotic random covariance matrix is formally the same as in the well-known form (see \cite{GenJac93} and \cite{UchYos13}) except that its randomness implicitly depends on possible jumps, while not on the spikes. See Theorem \ref{hm:thm_AMN} below.

\medskip

The paper is organized as follows. 
In Section \ref{hm:sec_model.intro}, we introduce the model setup and mention some more background. Then, in Section \ref{hm:sec_robust.GQLFs}, we describe the two robustified versions of the conventional Gaussian quasi-likelihood function in terms of the density-power tapering and the normalized Gaussian quasi-score via the H\"{o}lder inequality. 
Section \ref{hm:sec_main.results} presents the main results of this paper, namely, the asymptotic mixed normality of the two proposed estimators. Its proof is given in Appendix \ref{hm:sec_main.results.proof} after presenting a series of auxiliary asymptotic results in Appendix \ref{hm:sec_aux.asymp}.
Illustrative simulation experiments are given in Section \ref{sec_simulations}.
Some additional numerical results are given in Appendix \ref{sec:add.sim}.


\subsection*{Basic notation and convention}

Throughout this paper, we will denote by $C$ and $C'$ positive universal constants, which do not depend on $n$ and the user-input tuning parameter $\lambda=\lam_n>0$ 
introduced later, 
but possibly depend on its supremum $\overline{\lambda}$ (see \eqref{hm:lam.condition-1} and \eqref{hm:lam.condition-2} for the required conditions on $(\lam_n)$); the constants may vary from line to line.
For positive real sequences $(a_n)$ and $(b_n)$, $a_n\lesssim b_n$ 
means that $\limsup_n (a_n/b_n)<\infty$; 
the notation will also be used for random variables when they hold a.s.
For any matrix $A$, $A^{\otimes 2}:=AA^\top$ with $\top$ denoting the transposition.
When $A$ is a square matrix, $\lam_{\min}(A)$ denotes the minimum eigenvalue of $A$.
\serev{We denote by $|x|$ the Euclidean norm of $x\in\mbbr^{d}$.}
For $k\ge 1$, we denote by $\phi_k(\cdot;\mu,\Sig)$ the $k$-dimensional normal $N_k(\mu,\Sig)$-density and $\phi_k(\cdot):=\phi_k(\cdot;0,I_k)$ with $I_k$ denoting the $k$-dimensional identity matrix; we will often simply write $\phi(\cdot)=\phi_d(\cdot)$ for the dimension $d$ of $Y$.
We write $\p_a^k$ for the $k$th (partial) derivative operator with respect to the variable $a$.
Both $I(A)$ and $I_A$ denote the indicator function of the event $A$.
For a vector $u$, we write $M[u]=\sum_k M_k u_k$ (the inner product) for a linear form $M=\{M_{k}\}$ and also $M[u^{\otimes 2}]=\sum_{k,l} M_{kl}u_k u_l$ (the quadratic form) for a bilinear form $M=\{M_{kl}\}$; depending on the context, $M[u]$ and $M[u^{\otimes 2}]$ themselves may be a (multi)linear form.
Finally, all non-random and stochastic order symbols are used for $n\to\infty$ unless otherwise mentioned.

\section{Volatility regression model with contaminations}
\label{hm:sec_model.intro}

\subsection{Setup and assumptions}
\label{hm:sec_setup}

Suppose that we are given a filtered probability space $(\Omega,\mcf,(\mcf_t)_{t\in[0,T]},\pr)$ satisfying the usual hypothesis, where $T>0$ is a fixed constant. We consider the {\cadlag} ($(\mcf_t)$-)adapted processes $\ly$ and $\lx$ in $\mbbr^d$ and $\mbbr^{d'}$, respectively, described by
\begin{align}
\ly_t &= \ly_0 + \int_0^t \mu_{s} ds + \int_0^t \sig(\lx_{s-},\theta)dw_s + J_t,
\label{hm:model.Y.star}\\
\lx_t &= \lx_0 + \int_0^t \mu'_{s} ds + \int_0^t \sig'_{s-}dw'_s + J'_t,
\label{hm:model.X.star}
\end{align}
where $\zeta_{s-}:=\lim_{u\uparrow s}\zeta_{u}$ for a process $\zeta$ and the ingredients are given as follows:
\begin{itemize}
\item The diffusion coefficient $\sig:\,\mbbr^{d'} \times \overline{\Theta} \to \mbbr^d\otimes \mbbr^r$ is known except for the finite-dimensional parameter 
\begin{equation}
    \theta=(\theta_1,\dots,\theta_p)\in\Theta\subset\mbbr^p,
\end{equation}
where the parameter space $\Theta$ is a bounded convex domain, so that its closure $\overline{\Theta}$ is compact;
\item $\mu$, $\mu'$, and $\sig'$ are {\cadlag} adapted processes in $\mbbr^d$, $\mbbr^{d'}$, and $\mbbr^{d'}\otimes\mbbr^{r'}$, respectively;
\item 
$J$ and $J'$ are {\cadlag} adapted finite-activity pure-jump processes in $\mbbr^d$ and $\mbbr^{d'}$, respectively, that is, both $J$ and $J'$ vary only by a.s. finitely many jumps on $[0,T]$.
\end{itemize}
The process $(\lx,\ly)$ is the underlying dynamics of our model. 
Let $t_j=t^n_j:=jh$ with $h=h_n:=T/n$; here and in what follows, we will omit the dependence on $n$ from the notation for brevity. For each $n\in\mbbn$, we define the process $\oy=\oy^n$ and $\ox=\ox^n$ observed at at high frequency as follows:
\begin{align}
\oy_t &= \ly_t + \sumj \Ups_j I(t=t_j),
\label{hm:Y.spike}\\
\ox_t &= \lx_t + \sumj \Ups'_j I(t=t_j),
\label{hm:X.spike}
\end{align}
where $\Ups_{j}=\Ups_{n,j}\in\mbbr^d$ and $\Ups'_{j}=\Ups'_{n,j}\in\mbbr^{d'}$, all being $\mcf_{t_j}$-measurable triangular array of random variables. 
The components of the covariate process $\lx$ may contain those of $\ly$. Hence, our scope includes, in particular, situations where $Y^\star$ is a stochastic differential equation model and the covariate process $X^\star$ is also contaminated.

Our objective is to estimate the true value $\tz\in\Theta$, assumed to exist, based on a discrete-time sample $\{(\ox_{t_{j}}, \oy_{t_j})\}_{j=0}^{n}$.
Throughout this paper, we regard $\{(\Ups_j,\Ups'_j)\}_{n,j}$ and $(J,J')$ as contaminating elements, implying that the ideal situation is the case where $|\Ups_j|\vee|\Ups'_j|\equiv 0$ and $J,J'\equiv 0$ so that the model $(X,Y)$ equals the continuous-semimartingale regression model \eqref{hm:csm.reg} and \eqref{hm:csm.reg.X}.
\hmrev{
One may regard the spikes $\{(\Ups_j,\Ups'_j)\}_{n,j}$ as exogenous outliers—such as human or instrumental measurement errors—while viewing the jumps from $(J,J')$ as intrinsic components of the underlying dynamics.
Nevertheless, our primary goal is to estimate the diffusion coefficient under discontinuous irregularities caused by both jumps and spikes; thus, we treat both variations as outliers in this specific context.
Our method effectively distinguishes the continuous characteristics of the underlying model from its discontinuous features. See Remark \ref{hm:rem_G0j} for related comments.
}

Denote by $\pr_\theta$ the distribution of the random elements
\begin{equation}
\left(\ly,\lx,\mu,\mu',\sig',w,w',J,J',\{\Ups_{j}\}_{n,j},\{\Ups'_{j}\}_{n,j}\right)
\nonumber
\end{equation}
associated with $\theta\in\overline{\Theta}$ and the corresponding expectation by $\E_\theta$.
\footnote{
\hmrev{
As one of the reviewers indicated, it is natural to treat the spike contamination $(\Upsilon,\Upsilon')$ as an external mechanism independent of the underlying dynamics. 
Although $(\Upsilon,\Upsilon')$ is under $\pr_\theta$'s rule, by splitting the underlying probability space without loss of generality, we may disintegrate the image measure $\pr_\theta$ as the product $\pr_\theta(du,dv)=\pr'_\theta(du) \otimes \pr^{(\Upsilon,\Upsilon')}(dv)$ where $\pr'_\theta$ denotes the joint distribution of $(\ly,\lx,\mu,\mu',\sig',w,w',J,J')$ and $\pr^{(\Upsilon,\Upsilon')}$ does that of $(\Upsilon,\Upsilon')$; here, $\pr^{(\Upsilon,\Upsilon')}$ is completely free from $\theta$ and independent of any other randomness driving the dynamics of $(\lx,\ly)$.
}
}
We will use the shorthands $\pr=\pr_{\tz}$ and $\E=\E_{\tz}$ with slight abuse of notation.
Moreover, let $\pr^{j-1}_\theta[\cdot]=\pr_\theta[\cdot|\mcf_{t_{j-1}}]$ and $\E^{j-1}_\theta[\cdot]=\E_\theta[\cdot|\mcf_{t_{j-1}}]$, the regular conditional $\pr_\theta$-probability given $\mcf_{t_{j-1}}$ and the associated conditional expectation. Further, we will abbreviate as $\pr^{j-1}[\cdot]=\pr^{j-1}_{\tz}[\cdot]$ and $\E^{j-1}[\cdot]=\E^{j-1}_{\tz}[\cdot]$.
As usual, we will mostly omit the qualifier ``a.s.'' when mentioning these conditional quantities.

\medskip

We now state \serev{our assumptions}.
Let $S(x,\theta) := \sig(x,\theta)^{\otimes 2}$.
We use the shorthands $\sup_\theta$ and $\inf_\theta$ for $\sup_{\theta\in\overline{\Theta}}$ and $\inf_{\theta\in\overline{\Theta}}$, respectively.
\hmrev{We begin with some regularity conditions on the diffusion coefficient.}

\begin{ass}[Diffusion coefficient]
\label{hm:A_diff.coeff}~
\begin{enumerate}
\item The function $(x,\theta) \mapsto S(x,\theta)$ belongs to the class $\mcc^{2,4}(\mbbr^{d'} \times \Theta)$, with all the partial derivatives continuous in $\overline{\Theta}$ for each $x$, and moreover, $\theta \mapsto \p_x^k \p_\theta^l S(x,\theta)$ is continuous for each $x\in\mbbr^{d'}$ and admissible $(k,l)$.

\item 
$\ds{\sup_\theta |\p_x^k \p_\theta^l S(x,\theta)| \lesssim (1+|x|)^C}$ and there exists a constant $c_S'\ge 0$ such that
\begin{equation}
\inf_{\theta}\lam_{\min}(S(x,\theta)) \gtrsim (1+|x|)^{-c_S'}
\nn
\end{equation}
for $x\in\mbbr^{d'}$.
\end{enumerate}
\end{ass}


For a process $\xi$, we will write $\D \xi_s = \xi_s - \xi_{s-}$ for the jump size of $\xi$ at time $s$, and $\D_j \xi = \xi_{t_j} - \xi_{t_{j-1}}$ for the $j$th increment of $\xi$.
Let
\begin{align}
   N_t := \sum_{0<s\le t} I(\D Y^\star_s \ne 0),
   \qquad 
   N'_t := \sum_{0<s\le t} I(\D X^\star_s \ne 0).
\end{align}


\begin{ass}[Jump structure]
\label{hm:A_J}~
\begin{enumerate}
\item The numbers of jumps of $\ly$ and $\lx$ are a.s. finite in $[0,T]$:
\begin{equation}
    \pr\left[\max\{N_T,\, N'_T\} <\infty\right] =1.
\end{equation}
\item There exist constants $\kappa>1/2$ and $c_1 \ge 0$ for which
\begin{equation}\label{hm:A_J-1}
\pr^{j-1}\left[ \D_j N + \D_j N' \ge 1\right] \lesssim (1+|X_{t_{j-1}}|^{c_1}) \, h^{\kappa}
\end{equation}
for $j=1,\dots,n$.
\item $\ds{\sup_{t\le T}\E[|J_t'|^K]<\infty}$ for any $K>0$, and
\begin{equation}\label{hm:A_J-2}
    \sup_{\genfrac{}{}{0pt}{}{t,s\in[0,T];}{|t-s|< h}} \E\left[|J'_t - J'_s|^2\right] \lesssim h^{c'}
\end{equation}
for some $c'>0$.
\end{enumerate}
\end{ass}

\hmrev{Then, since the jump process $(J,J')$ is finite-activity,} we can write
\begin{equation}
    J_t = \sum_{0<s\le t} \D Y^\star_s,
    \qquad J'_t = \sum_{0<s\le t} \D X^\star_s.
\end{equation}
It may happen that $J_\cdot(\omega)\equiv 0$ and/or $J'_\cdot(\omega)\equiv 0$ on $[0,T]$ with positive probability. No structural assumptions are made about the jump-occurrence mechanism and the jump-size distributions; for example, the jump component can be driven by the Hawkes-driven type studied in \cite{DioLemLoc20}. The simplest example of $(J,J')$ is a compound Poisson process, for which \eqref{hm:A_J-1} holds with $c_1=0$ and $\kappa=1$. The condition \eqref{hm:A_J-2} is met with $c'=1$ for a large class of jump processes.

Next, we impose structural assumptions on the spike-type contaminations $\{\Ups_{j}\}$ and $\{\Ups'_{j}\}$.

\begin{ass}[Spike-noise structure]
\label{hm:A_spike}~
\begin{enumerate}
\item For each $n\in\mbbn$ and $j\le n$, the random variable $(\Ups_j,\Ups'_j)\in\mbbr^d \times \mbbr^{d'}$ is $\mcf_{t_j}$-measurable.

\item 
We have $|\Ups_j|\vee|\Ups'_j|>0$ a.s. only for finitely many $j\le n$ uniformly in $n\in\mbbn$:
\begin{equation}
\pr\left[
\sup_{n\in\mbbn} 
\#
\{j\le n :\, 
|\Ups_j-\Ups_{j-1}|\vee |\Ups'_{j-1}|
>0 \} < \infty \right] = 1.
\nonumber
\end{equation}

\item Given $\kappa>1/2$ in \eqref{hm:A_J-1},
\begin{equation}\label{hm:A_spike-1}
\pr^{j-1}\left[ 
|\Ups_j-\Ups_{j-1}|\vee |\Ups'_{j-1}|
>0 \right] \lesssim (1+|X_{t_{j-1}}|^{C}) h^{\kappa}
\end{equation}
for $j=1,\dots,n$.

\item $\ds{\sup_{n\ge 1}\max_{j\le n} \E[|\Ups'_{n,j}|^K]<\infty}$ for any $K>0$.


\end{enumerate}
\end{ass}

\medskip

Denote the $j$th short-time interval by
\begin{equation}
    \mbbi_j := (t_{j-1},t_j],
\end{equation}
and write $\int_j=\int_{\mbbi_j}$. 
By \eqref{hm:Y.spike}, we have
\begin{align}
\D_j Y 
&= \D_j Y^\star + \Ups_j - \Ups_{j-1} \nn\\
&= \int_j \mu_s ds + \int_j \sig(X^\star_{s-},\tz)dw_s + \D_j J + \Ups_j - \Ups_{j-1}.
\label{hm:DYj_form}
\end{align}

In deriving the key limit theorems \serev{(Appendix \ref{hm:sec_aux.asymp})}, we will need to manage the effect of contaminations caused by $(J,\{\Ups_j\},J',\{\Ups'_j\})$. To control the term ``$\D_j J + \Ups_j - \Ups_{j-1}$'' on the right-hand side of \eqref{hm:DYj_form}, we introduce the (Good) event
\begin{equation}\label{hm:def_Gj}
    G_j = G_{n,j} := G_{1,j} \cap G_{2,j}\in\mcf_{t_j},
\end{equation}
where
\begin{align}
G_{1,j} &:= \left\{ \D_j N = 0,~ \D_j N' = 0\right\}
\nn\\
&=\{ \text{There are no jumps on $\mbbi_j$ for both $J$ and $J'$} \}, \nn\\
G_{2,j} &:= \left\{ |\Ups_j-\Ups_{j-1}|=0, ~ |\Ups'_{j-1}| = 0\right\}.
\nonumber
\end{align}
On $G_j$, we have
\begin{align}
    \D_j Y
    &= \int_{j} \mu_s ds + \int_{j} \sig(X^\star_{s-},\tz)dw_s,
    \label{hm:DY_Gj}\\
\ox_t &= \lx_{t_{j-1}} + \int_{t_{j-1}}^t \mu'_s ds + \int_{t_{j-1}}^t \sig'_{s-}dw'_s
\label{hm:X.obs_Gj}
\end{align}
for $t\in\mbbi_j^\circ =(t_{j-1},t_j)$. By splitting $\Omega$ into $G_j$ and $G_j^c$ for each $j\le n$, we will utilize:
\begin{itemize}
    \item On $G_j$, the non-contaminated nature of \eqref{hm:DY_Gj} and \eqref{hm:X.obs_Gj}, which enables us to proceed as in the continuous model \eqref{hm:csm.reg} and \eqref{hm:csm.reg.X};
    \item On $G_j^c$, the essential boundedness of the integrands of the proposed quasi-likelihood through the tapering factor $\phi(\cdot)^\lam$ for a tuning parameter $\lam>0$, making the corresponding terms asymptotically negligible.
\end{itemize}
The above operations through $G_j$ will appear only in the proof and are implicit in computing our estimator in practice; the specific two forms of the proposed quasi-likelihoods will be given in \eqref{hm:def_dp-H-0} and \eqref{hm:def_ldp-H-0} in Section \ref{hm:sec_robust.GQLFs}.


\medskip


Let us denote by $\cx$ the continuous (no-jump and no-spike) version of $X$:
\begin{equation}
    \cx := X^\star - J'.
\end{equation}
\serev{Moreover, we denote $E[\mathsf{X};A]=E[\mathsf{X} \,I_{A}]$ for any random variable $\mathsf{X}$ and any event $A\in\mcf$.}
We impose several integrability conditions.

\begin{ass}[Drift coefficient and covariate process]
\label{hm:A_drif.coeff&X}~
$\mu=(\mu_t)_{t\le T}$, $\mu'=(\mu'_t)_{t\le T}$ and $\sig'=(\sig'_t)_{t\le T}$ are $(\mcf_t)$-adapted {\cadlag} processes in $\mbbr^d$, $\mbbr^{d'}$, $\mbbr^{d'}\otimes\mbbr^{r'}$, respectively, such that 
\begin{align}
& \sup_{t\in[0,T]} \E\left[|\lx_0|^K + |\mu_t|^K + |\mu'_t|^K + |\sig'_t|^K\right] <\infty,
\nn\\
& 
\max_{1\le j\le n}\sup_{t\in\mbbi_j^\circ} 
\left(
\E\left[|\mu_t - \mu_{t_{j-1}}|^K; G_j\right] 
+ 
\E\left[ |\sig'_t - \sig'_{t_{j-1}}|^2 \right] 
\right)
=o(1)
%
\end{align}
for any $K\ge 2$.
\end{ass}

From Assumptions \ref{hm:A_J}, \ref{hm:A_spike}, and \ref{hm:A_drif.coeff&X}, it is easily seen that for any $K>0$,
\begin{equation}\label{hm:X_m.bound}
  \sup_{t\le T} \left(\E\left[|\lx_t|^K\right] + \E\left[|\ox_t|^K\right] \right) < \infty,
\end{equation}
\begin{equation}\label{hm:nY_moment}
    \max_{j\le n} \E\left[\left|\frac{\D_j Y}{\sqrt{h}}\right|^K;\,G_j\right] = O(1),
\end{equation}
and
\begin{equation}\label{hm:X-moment.estimates}
    \sup_{s,t\le T;\,|t-s|\le h} \left( \E\left[\left|\frac{1}{\sqrt{h}}(\cx_t - \cx_s)\right|^K\right] + \E\left[\left|\frac{1}{h^{(1 \wedge c')/2}}(\lx_t - \lx_s)\right|^2\right] \right) < \infty,
\end{equation}
where $c'>0$ is the one given in \eqref{hm:A_J-2}.
Moreover,
\begin{equation}
    \sup_{s,t\le T;\,|t-s|< h} \E\left[\left|\frac{1}{h^{c''}}(\ox_t - \ox_s)\right|^2\right] < \infty
\end{equation}
for some constant $c''>0$.

We also need the identifiability condition.

\begin{ass}[Identifiability]
\label{hm:A_iden}
We have
\begin{equation}
\pr\big[\forall t\in[0,T],~S(X^\star_t,\theta) = S(X^\star_t,\tz)\big]=1
\label{hm:assump_iden}
\end{equation}
if and only if $\theta=\tz$.
\end{ass}

\medskip

Finally, our estimation procedure will depend on a tuning parameter $\lam>0$, which is assumed to satisfy
\begin{equation}\label{hm:lam.condition-1}
    \lam \in (0,\overline{\lambda}]
\end{equation}
for some $\overline{\lambda}\in(0,\infty)$.
It is often enough to only consider $\lam\in(0,1]$; we do not need to specify the value $\overline{\lam}$ in practice.
We will allow the parameter $\lam>0$ satisfying \eqref{hm:lam.condition-1} to depend on the sample size $n$, say $\lam_n$.

\begin{ass}[Tapering parameter]
\label{hm:A_lam}
We have either
\begin{enumerate}
    \item $\lam_n\equiv \lam>0$ (a fixed constant), or
    \item $\lam_n\to 0$ in such a way that for $\kappa>1/2$ in \eqref{hm:A_J-1},
    \begin{equation}
    \label{hm:lam.condition-2}
    \frac{\sqrt{n}\,h^{\kappa}}{\lam_n} \to 0.
    \end{equation}
\end{enumerate}
\end{ass}

The condition \eqref{hm:lam.condition-2} is equivalent to $n^{\kappa-1/2}\lam_n \to \infty$, implying that the rate at which $\lam_n \to 0$ must not be too fast, though it can be arbitrarily slow.
This seems natural, since our asymptotic results fail to hold for $\lam=0$ in the presence of the contaminations. We will see that the asymptotic covariance matrices of our estimators under Assumption \ref{hm:A_lam}(1) are continuous at $\lam=0$.

\subsection{Preliminary observation}
\label{hm:sec_ex.levy}

As a toy model, let us briefly look at the $d$-dimensional {\lp}
\begin{equation}
Y_t = Y_0 + \sig(\theta)\, w_t + J_t,
\label{hm:model.Levy}
\end{equation}
where 
$J_t = \sum_{j=1}^{N_t}\xi_j$ is a compound-Poisson process with a Poisson process $N$ with intensity $\rho>0$ and i.i.d. (jump-size) random variables $\xi_j$ in $\mbbr^d$ such that $\pr[|\xi_1|=0]=0$.
The constant diffusion matrix $S(\theta)=\sig(\theta)^{\otimes 2}$ fulfills that $S(\cdot) \in \mcc^3(\overline{\Theta})$ and $\inf_{\theta\in\overline{\Theta}}\lam_{\min}(S(\theta))>0$.
With these settings, we want to estimate the true value $\tz$ from a sample $(Y_{t_j})_{j=0}^{n}$ without knowing $\rho$ and $\mcl(\xi_1)$. 
Although the model \eqref{hm:model.Levy} is a rather special case of the model \eqref{hm:model.Y.star}, a closer look at this setting will clarify some essence of the present study.

The {\lm} of $J$ is given by $\nu(dz)=\rho F(dz)$, where $F(dz)$ is the non-trivial distribution of $\xi_1$ with $F(\{0\})=0$. Let $F^{\ast k}$ denote the $k$-fold convolution of $F$ (with $F^{\ast 0}$ being the Dirac measure at $0$).
Since $J$ is a {\cpp}, $\D_j J$ has the same distribution as $J_h$, say $\pr^{J_h}$, and
\begin{equation}
\pr^{J_h}(dz)=e^{-\rho h}\sum_{l=0}^\infty \frac{(\rho h)^{l}}{l!} F^{\ast l}(dz).
\nonumber
\end{equation}
We have
\begin{equation}
\mcl(Y_{t_j}|Y_{t_{j-1}}=x) = N_d\left(x, h S(\theta)\right) \ast P^{J_h}
\nonumber
\end{equation}
under $\pr_\theta$. This conditional distribution has a non-degenerate Gaussian part; hence, it admits a bounded, positive, and smooth density (see \cite[Proposition 28.1]{Sat99} and \cite{Sha69}); we denote by $p_h(x,y;\theta)$ the corresponding transition density.
Then,
\begin{align}
p_h(x,y;\theta)
&= \int \phi_d(y-x-z;0,h S(\theta))P^{J_h}(dz)
\nn\\
&= e^{-\rho h}\phi_d(y-x;0,h S(\theta)) 
\nn\\
&{}\qquad 
+ e^{-\rho h}\sum_{l=1}^\infty \frac{(\rho h)^{l}}{l!} \int \phi_d(y-x-z;0,h S(\theta))F^{\ast l}(dz).
\label{hm:wp.cpp_trans.dens-1}
\end{align}
Note that the second identity corresponds to the mixture-distribution representation
\begin{align}
p_h(x,y;\theta)
&= e^{-\rho h}\phi_d(y-x;0,h S(\theta)) + (1-e^{-\rho h}) \bar{r}_h(x,y;\theta)
\label{hm:wp.cpp_trans.dens-2}
\end{align}
for the two probability densities $y\mapsto\phi_d(y-x;0,h S(\theta))$ and
\begin{align}
\bar{r}_h(x,y;\theta) := \frac{e^{-\rho h}}{1-e^{-\rho h}}\sum_{l=1}^\infty \frac{(\rho h)^{l}}{l!} \int \phi_d(y-x-z;0,h S(\theta))F^{\ast l}(dz).
\nonumber
\end{align}
It is expected that we may effectively ignore the terms corresponding to the event where $\D_j N \ge 1$ as $O(h)$-quantities in computing the conditional expectation of the form
\begin{align}
\int K(x,y) p_h(x,y;\theta)dy &= e^{-\rho h} \int K(x,y) \phi_d(y-x;0,h S(\theta)) dy 
\nn\\
&{}\qquad + (1-e^{-\rho h}) \int K(x,y) \bar{r}_h(x,y;\theta)dy,
\nonumber
\end{align}
where the integrand $K$ is essentially bounded so that the second term on the right-hand side becomes negligible compared with the first one.
These observations are significant when dealing with a bounded quasi-likelihood when $Y$ admits a sufficiently smooth transition density.

The previous arguments are informative. We will implicitly take an analogous route to handle the general model $(X,Y)$ given by \eqref{hm:Y.spike} and \eqref{hm:X.spike} through the sequence of good events $(G_j)$.

\subsection{Gaussian quasi-likelihood}

We now return to the original model setup described in Section \ref{hm:sec_setup}. For later reference, we recall the conventional Gaussian quasi-likelihood (GQLF) for $Y^\star$ without jumps.
We will write
\begin{equation}
f_{j-1}(\theta) = f(\ox_{t_{j-1}},\theta)
\nonumber
\end{equation}
for any measurable function $f$ defined on $\mbbr^{d'}\times\mbbr^d\times\overline{\Theta}$; we will just write $f_{j-1}$ if the argument $\theta$ is missing. 
The naive Euler approximation, which ignores the drift, the jump component, and the spike-noise structure \eqref{hm:Y.spike}, is given by (under $\pr_\theta$)
\begin{equation}
Y_{t_j} \overset{\pr_\theta}{\approx} Y_{t_{j-1}} + h \sig_{j-1}(\theta)\D_j w.
\label{hm:Euler}
\end{equation}
Let us write
\begin{equation}
    \mathsf{d}_{j-1}(\theta) = \det(S_{j-1}(\theta)).
\end{equation}
The GQLF associated with \eqref{hm:Euler} is defined by
\begin{align}
\mbbh_{n}(\theta) &:= \sumj \log \phi\left(Y_{t_j};\, Y_{t_{j-1}},\, h S_{j-1}(\theta)\right)
\nn\\
&= \text{(Const.)} - \frac12 \sumj \left( \log\mathsf{d}_{j-1}(\theta) + \frac1h S_{j-1}(\theta)^{-1}[(\D_j Y)^{\otimes 2}] \right).
\label{hm:GQLF}
\end{align}
We refer to \cite{GenJac93} and \cite{UchYos13} for asymptotics of this GQLF in case where $J,J'\equiv 0$ and $\Ups_{j},\Ups'_{j}\equiv 0$.

We denote by
\begin{equation}
\nY_j = \nY_{n,j} := h^{-1/2}\D_j\oy
\nonumber
\end{equation}
the $j$th increments of $Y$ scaled by $h^{-1/2}$ and then let
\begin{align}
\phi_j(\theta)=\phi_{n,j}(\theta) 
&:= \phi\left(\oy_{t_j};\, \oy_{t_{j-1}},\, h S_{j-1}(\theta)\right)
\nn\\
&=\frac{1}{h^{d/2} \mathsf{d}_{j-1}(\theta)^{1/2}}\phi\left( S_{j-1}(\theta)^{-1/2}\nY_j\right).
\label{hm:deg_phi.j}
\end{align}
The Gaussian quasi-score function associated with \eqref{hm:GQLF} is given by $\p_\theta\mbbh_n(\theta)=\sumj \psi_j(\theta)$ with
\begin{align}
\psi_j(\theta)
&=\{\psi_{j,k}(\theta)\}_{k=1}^{p}
:=\p_\theta\log \phi_j(\theta) \nn\\
&= -\frac12 \left( \p_\theta \log\mathsf{d}_{j-1}(\theta) + \p_\theta(S_{j-1}^{-1})(\theta)[\nY_j^{\otimes 2}] \right).
\label{hm:deg_psi.j}
\end{align}
The unbounded Gaussian quasi-score function $\p_\theta\mbbh_n(\theta)$ is fragile against outliers. Roughly speaking, the variable $h^{-1/2}(\D_j Y)$ is no longer stochastically bounded, rendering the asymptotics (of \cite{GenJac93} and \cite{UchYos13}) theoretically untenable.
As we mentioned in the Introduction, a popular strategy to remedy this point is to employ a jump-detection filter. In the subsequent sections, we will proceed with an entirely different route.

\section{Robustified Gaussian quasi-likelihood functions}
\label{hm:sec_robust.GQLFs}

In this section, we will consider the two robustified variants of GQLFs through the density-power weighting and the H\"{o}lder inequality, and then present our main result. 

\subsection{Density-power divergence}

In general, the density-power divergence (also known as $\beta$- or BHHJ divergence) from the true distribution $g d\mu$ to the statistical model $f_{\theta}d\mu$, where $\mu$ denotes some dominating $\sig$-finite measure, is defined by
\begin{align}
(f_{\theta};g) 
&\mapsto \frac{1}{1+\lambda}\int \left( f_{\theta}^{1+\lambda} - \left(1+\frac{1}{\lambda}\right)f_{\theta}^\lambda g + \frac{1}{\lambda}g^{1+\lambda}\right) d\mu
\nn\\
&= \frac{1}{\lambda+1}\int f_{\theta}^{1+\lambda}d\mu - \frac{1}{\lambda}\int f_{\theta}^\lambda g d\mu 
+ \frac{1}{\lambda(\lambda+1)}\int g^{1+\lambda} d\mu.
\nn
\end{align}
This is a nonnegative quantity, which becomes zero if and only if $\mu(g=f_{\theta})=1$. 
Through one tuning parameter $\lambda\ge 0$, the density-power divergence smoothly connects the outlier-sensitive Kullback-Leibler divergence $(f_{\theta};g) \mapsto \int \log(g/f_{\theta})g d\mu$ for $\lambda \to 0$ and the outlier-resistant $L^2$-distance $(f_{\theta},g) \mapsto \int (f_{\theta}-g)^2 d\mu$ for $\lambda = 1$.
This means that the minimum contrast estimator associated with the density-power divergence bridges the maximum-likelihood estimator and the $L^2$-distance one, which are respectively defined to be minimizers of the empirical counterparts of $\theta\mapsto -\int (\log f_\theta) g d\mu$ and $\theta\mapsto \int f_\theta^2 d\mu - 2 \int f_\theta g d\mu$.
The density-power divergence enables us to balance between them, providing a practical and transparent estimation procedure which is robust against outliers without requiring any nonparametric-type smoothing.
See \cite{BasHarHjoJon98}, \cite{JonHjoHarBas01}, \cite[Chapter 9]{BasShiPar11}, and the references therein for more details.

Turning back to our framework, we consider the density-power weighting of the score function associated with the GQLF $\mbbh_n(\theta)$ of \eqref{hm:GQLF} by multiplying the (non-predictable) weight $\phi_j(\theta)^\lambda$ to each summand, namely $\sumj \phi_j(\theta)^\lambda \psi_j(\theta)$ with $\psi_j(\theta)$ given by \eqref{hm:deg_psi.j}. In each $\mbbi_j$, this weight mitigates discontinuous variations, which are, in our setting, caused by jumps and/or spikes and are much larger compared with the continuous variation due to the drift and diffusion coefficients.

Since the weighting entails a bias in the estimating equation, we need the compensation:
\begin{equation}
\theta\mapsto \sumj \left( \phi_j(\theta)^\lambda\psi_j(\theta) - \E_\theta^{j-1}[\phi_j(\theta)^\lambda\psi_j(\theta)]\right)
\label{hm:tGQ.score}
\end{equation}
to obtain the associated genuine martingale estimating function for estimating $\theta$. 
The conditional distribution $\mcl(Y_{t_j}|\mcf_{t_{j-1}})$ can rarely be explicitly given, so that \eqref{hm:tGQ.score} cannot be of direct use in practice.

Because of the Euler scheme \eqref{hm:Euler}, it is natural to approximate $\nY_j$ by $\sig_{j-1}Z_j$, where 
\begin{equation}\label{hm:Z_def}
    Z_j=Z_{n,j}:=h^{-1/2}\D_j w,\qquad j=1,\dots,n,
\end{equation}
forms an $N_r(0,I_r)$-i.i.d. sequence for each $n$.

As was mentioned in Section \ref{hm:sec_ex.levy}, with the density-power weighting under the high-frequency sampling scheme, we may expect that jumps and/or spikes are automatically ignored. This wishful thinking will be justified in Section \ref{hm:sec_robustification}, which in particular enables us to identify the ``leading'' term of $\E_\theta^{j-1}[\phi_j(\theta)^\lambda\psi_j(\theta)]$ in an explicit way and will serve as a basic tool in our asymptotic analyses.

By \eqref{hm:deg_psi.j},
\begin{align}
h^{d\lambda/2} \phi_j(\theta)^\lambda \psi_j(\theta) &= -\frac12 \mathsf{d}_{j-1}(\theta)^{-\lambda/2}
\phi\big(S_{j-1}(\theta)^{-1/2}\nY_j\big)^\lambda 
\nn\\
&{}\qquad \times \left\{ \p_\theta \log \mathsf{d}_{j-1}(\theta)
+ \p_\theta (S_{j-1}^{-1})(\theta)[\nY_j^{\otimes 2}]\right\},
\label{hm:dp-const-1}
\end{align}
which is, roughly speaking, an $O_p(1)$-quantity when there is no comtamination; without the multiplicative factor $h^{d\lambda/2}$, it is stochastically divergent.
Further, let
\begin{align}
\phi_{j-1}(y;\theta) &:= \phi\left(y; \oy_{t_{j-1}},\, h S_{j-1}(\theta) \right), \nn\\
\psi_{j-1}(y;\theta) &:= \p_\theta\log \phi_{j-1}(y;\theta).
\nonumber
\end{align}
\hmrev{Based on the contents of Appendix \ref{hm:sec_robustification}, where we will theoretically show that the discontinuous variations can be effectively ignored, it will turn out that we can proceed as if}
\begin{align}
\E_\theta^{j-1}[\phi_j(\theta)^\lambda\psi_j(\theta)]
&= 
\int \phi_{j-1}(y;\theta)^{\lam+1} \psi_{j-1}(y;\theta) dy
+ \frac{1}{h^{d\lambda/2}}\mathsf{R}_{j-1}(\theta;\lam)
\label{hm:intro.dp-1}
\end{align}
for each $j\le n$ under $\pr_\theta$, where the ``remainder'' term $\mathsf{R}_{j-1}(\theta;\lam)$ will turn out to be negligible in a certain sense. 
The specification of the ``leading'' term is of theoretical importance since it is crucial in the general $M$-estimation framework to construct an approximate martingale estimating function.

\hmrev{By integrating \eqref{hm:tGQ.score} with respect to $\theta$ \frev{
and approximating as 
\begin{equation}
    \E_\theta^{j-1}[\phi_j(\theta)^\lambda\psi_j(\theta)]
\approx \int \phi_{j-1}(y;\theta)^{\lam+1} \psi_{j-1}(y;\theta) dy
\nn
\end{equation}
}as suggested by \eqref{hm:intro.dp-1}, we obtain} the following random function to be maximized:
\begin{equation}
\theta\mapsto
\sumj \left( 
\frac{1}{\lam}\phi_j(\theta)^\lam - \frac{1}{\lam+1} \int \phi_{j-1}(y;\theta)^{\lam +1} dy
\right).
\label{hm:intro.dp-2}
\end{equation}
By the identity 
\begin{equation}\label{hm:phi.power-identity}
\int_{\mbbr^d}\phi(z;\mu,\Sig)^a dz = a^{-d/2} \det(2\pi \Sig)^{(1-a)/2}
\end{equation}
for $a>0$, we have
\begin{equation}
\frac{1}{\lambda+1}\int \serev{\phi_j(y;\theta)^{\lambda+1}} dy 
= h^{-d\lambda/2} K_{\lambda,d}\, \mathsf{d}_{j-1}(\theta)^{-\lambda/2},
\label{hm:intro.ldp-4}
\end{equation}
where
\begin{align}
K_{\lambda,d} &:= \frac{(2\pi)^{-d\lambda/2}}{(\lambda+1)^{1+d/2}}.
\nn
\end{align}
By multiplying \eqref{hm:intro.dp-2} by $h^{d\lambda/2}$, we introduce the fully explicit \textit{density-power GQLF}:
\begin{align}
\mbbh_{n}(\theta;\lambda) 
&=
\sumj \left( 
\frac{h^{d\lambda/2}}{\lam}\phi_j(\theta)^\lam - K_{\lambda,d}\, \mathsf{d}_{j-1}(\theta)^{-\lambda/2}
\right)
\nn\\
&=\sumj \mathsf{d}_{j-1}(\theta)^{-\lambda/2} 
\left( \frac{1}{\lambda}\phi\left(S_{j-1}(\theta)^{-1/2}\nY_j\right)^\lambda 
- K_{\lambda,d} \right).
\label{hm:def_dp-H-0}
\end{align}
Given a value $\lambda>0$, we define the \textit{density-power GQMLE} by any element
\begin{align}
\tes(\lambda) \in \argmax_{\theta \in\overline{\Theta}}\mbbh_{n}(\theta;\lambda).
\label{hm:def_tes.beta}
\end{align}
The continuity of $\mbbh_{n}(\cdot;\lambda)$ and the measurable selection theorem ensure that there always exists a measurable $\tes(\lambda)$.


\begin{rem}\normalfont
\label{hm:rem_dpgqlf.form}
An application of l'H\^opital's rule shows that for $\lambda\to 0$ with $n$ fixed,
\begin{align}
& \frac{1}{h^{d\lambda/2}}\mbbh_{n}(\theta;\lambda) - \frac{n}{\lambda} + \frac{n}{h^{d\lambda/2}}
\nn\\
&=\sumj
\left( \frac{1}{\lambda}\left(\phi_j(\theta)^\lambda -1\right) 
- \frac{1}{h^{d\lambda/2}}\left( K_{\lambda,d}\, \mathsf{d}_{j-1}(\theta)^{-\lambda/2} -1\right) \right)
\nonumber
\end{align}
a.s. tends to the conventional GQLF $\mbbh_n(\theta)$ of \eqref{hm:GQLF}; 
the first summand on the right-hand side equals the Box-Cox transform of $\phi_j(\theta)$.
The above-mentioned facts are worth noting, although we do not use them at all.
\end{rem}

\subsection{H\"{o}lder-based divergence}

The normalized-score-based divergence, also known as $\gam$- or JHHB divergence, can be seen as the ``logarithmic'' version of the density-power divergence (see \eqref{hm:def_log-dp.div} below). Its origin goes back to \cite{Win95} and was then studied by \cite{JonHjoHarBas01}, \cite{FujEgu08}, and \cite{Fuj13} in detail. In this section, we will describe how this divergence can apply to our model setup. 

It is well-known that the above-mentioned divergence is closely related to the H\"{o}lder inequality. In our framework, we will not emphasize the ``normalized-score'' nature, but more simply, introduce the divergence from the viewpoint of the H\"{o}lder inequality.
In general, given two densities $f$ and $g$ with respect to a reference measure $\mu$ and a constant $\lam>0$, the H\"{o}lder inequality gives
\begin{equation}\label{hm:holder.ineq}
    \int f^\lam g d\mu \le \bigg(\int f^{\lam+1} d\mu\bigg)^{\lam/(\lam+1)} \bigg(\int g^{\lam+1} d\mu\bigg)^{1/(\lam+1)},
\end{equation}
from which we have
\begin{equation}
\left(\int g^{\lam+1} d\mu\right)^{1/(\lam+1)} 
- \int \frac{f^\lam }{(\int f^{\lam+1}d\mu)^{\lam/(\lam+1)}} \,g d\mu \ge 0,
\label{hm:intro.ldp-1}
\end{equation}
where the equality holds if and only if $g=f$ a.e.; this defines a divergence from the true (unknown) $g$ to the model $f$.

As in \eqref{hm:intro.dp-1}, the technical tool given in \serev{Appendix \ref{hm:sec_aux.asymp}} ensures that 
\begin{align}
\E_\theta^{j-1}[\phi_j(\theta)^\lam]
&= 
\int \phi_{j-1}(y;\theta)^{\lam+1} dy + \frac{1}{h^{d\lambda/2}}\mathsf{R}_{j-1}(\theta),
\label{hm:intro.ldp-2}
\end{align}
here again, the term $\mathsf{R}_{j-1}(\theta)$ being negligible in a certain sense. 
In view of \eqref{hm:intro.ldp-2} with \eqref{hm:intro.ldp-4}, by ignoring the term $\mathsf{R}_{j-1}(\theta)$ in the former, it is natural to estimate $\tz$ by a maximizer of the following empirical counterpart of \eqref{hm:intro.ldp-1}:
\begin{align}
\mbbh^{\flat}_n(\theta;\lam) 
&:= \sumj \frac{\phi_j(\theta)^\lam}{\left( \int \phi_{j-1}(y;\theta)^{\lam+1}dy \right)^{\lam/(\lam+1)}}
\label{hm:intro.ldp-5}\\
&= 
(h^{d\lambda/2})^{-1/(\lam+1)} \lam \left\{ (\lam+1) K_{\lam,d} \right\}^{-\lam/(\lam+1)}
\nn\\
&{}\qquad \times \sumj \frac{1}{\lam} \mathsf{d}_{j-1}(\theta)^{-\lam/(2(\lam+1))} \phi\big(S_{j-1}(\theta)^{-1/2}\nY_j\big)^{\lambda}.
\nonumber
\end{align}
By multiplying $\mbbh^{\flat}_n(\theta;\lam)$ by $(h^{d\lambda/2})^{1/(\lam+1)} \lam^{-1} \left\{ (\lam+1) K_{\lam,d} \right\}^{\lam/(\lam+1)}$, we introduce the \textit{H\"{o}lder-based GQLF}:
\begin{equation}
\mbbh_n(\theta;\lam) := \sumj \frac{1}{\lam} \mathsf{d}_{j-1}(\theta)^{-\lam/(2(\lam+1))} \phi\big(S_{j-1}(\theta)^{-1/2}\nY_j\big)^{\lambda},
\label{hm:def_ldp-H-0}
\end{equation}
where $\lam>0$ satisfies \eqref{hm:lam.condition-1}; this is an abuse of notation in conjunction with \eqref{hm:def_dp-H-0}, but there would be no confusion in the subsequent context.
As with \eqref{hm:def_tes.beta}, we define the \textit{H\"{o}lder-based GQMLE} 
by any element
\begin{align}\label{hm:def_tes.gam}
\tes(\lambda) \in \argmax_{\theta \in\overline{\Theta}}\mbbh_{n}(\theta;\lambda).
\end{align}
Although the density-power GQLF \eqref{hm:def_dp-H-0} and the H\"{o}lder-based GQLF \eqref{hm:def_ldp-H-0} look similar, they were constructed from different viewpoints.

\medskip

It will be seen that both $\mbbh_n(\theta;\lam)$ of \eqref{hm:def_dp-H-0} and \eqref{hm:def_ldp-H-0} are defined in such a way that $n^{-1}(\mbbh_n(\theta;\lam)-\mbbh_n(\tz;\lam))$ admits a non-trivial limit in probability for both cases in Assumption \ref{hm:A_lam}.

\begin{rem}\normalfont
\label{hm:fv.rem_add1}
Our quasi-likelihood functions \eqref{hm:def_dp-H-0} and \eqref{hm:def_ldp-H-0} consist of sums of functionals of the variables $\D_j Y$ and $X_{t_{j-1}}$ (see \eqref{hm:def_dp-H-0} and \eqref{hm:def_ldp-H-0} below). By \eqref{hm:X.spike} and \eqref{hm:DYj_form}, the effects of the spikes contained in $(X_{t_{j-1}},\D_j Y)$ are expected to disappear under 
Assumption \ref{hm:A_spike}, 
and indeed, our proofs will proceed as such. This is a rough rationale for why the increments $\Ups_j-\Ups_{j-1}$ are included therein, whereas they are not for $\Ups’$.    
\end{rem}

\begin{rem}\normalfont
\label{hm:rem_ldpgqlf.form}
Here is an analogue to Remark \ref{hm:rem_dpgqlf.form}:
for $\lambda\to 0$ with $n$ fixed, we have
\begin{equation}
 \left( \int \phi_{j-1}(y;\theta)^{\lam+1}dy \right)^{\lam/(\lam+1)}=1+O(\lam^2)\qquad \text{a.s.,}
\end{equation}
from which we can deduce that 
\begin{align}
\frac{1}{\lam}\left(\mbbh^{\flat}_n(\theta;\lam) - n \right)
&=
\frac{1}{\lam}\left( (h^{d\lambda/2})^{-1/(\lam+1)} \left\{ (\lam+1) K_{\lam,d} \right\}^{\lam/(\lam+1)}
\mbbh_{n}(\theta;\lambda) - n\right)
\nn\\
&=\sumj \frac{1}{\lam} \left(
\frac{\phi_j(\theta)^\lam}{\left( \int \phi_{j-1}(y;\theta)^{\lam+1}dy \right)^{\lam/(\lam+1)}} - 1\right)
\nn\\
&= \sumj \left\{\frac{1}{\lam} \left(\phi_j(\theta)^\lam - 1\right) + O(\lam) \right\}
\nonumber
\end{align}
a.s. tends to the GQLF $\mbbh_n(\theta)$.
\end{rem}

\medskip

We end this section with the following brief yet theoretically important remark by pointing out a relation to the (approximate) martingale property of the associated estimating functions.
\serev{As we will observe below, in the dynamic-structure (or more broadly, some inhomogeneous conditional-distribution) model, it is not clear if the partial derivative with respect to $\theta$ admits a normalized-score structure.}

\begin{rem}\normalfont
\label{hm:rem_ns.eq}
The estimating equation $\p_\theta\mbbh^{\flat}_n(\theta;\lam)=0$ is given by
\begin{equation}
\sumj \frac{\phi_j(\theta)^\lam \psi_j(\theta)}{\left( \int \phi_{j-1}(y;\theta)^{\lam+1}dy \right)^{\lam/(\lam+1)}}
=
\sumj \frac{\phi_j(\theta)^\lam \int \phi_{j-1}(y;\theta)^{\lam+1} \psi_{j-1}(y;\theta) dy}{\left( \int \phi_{j-1}(y;\theta)^{\lam+1}dy \right)^{(2\lam+1)/(\lam+1)}}.
\nonumber
\end{equation}
This reduces to the ``normalized estimating equation'' if and only if the integral $\int \phi_{j-1}(y;\theta)^{\lam+1}dy$ does not depend on $j$; see \cite[Eq.(1.2)]{Fuj13}. 
That is, the heterogeneity of data makes the implication of normalizing the (quasi-)score function vague.

The inequality \eqref{hm:intro.ldp-1} is equivalent to the logarithmic variant:
\begin{equation}
\frac{1}{\lam(\lam+1)}\log\bigg(\int g^{\lam+1}\bigg) 
- \left\{ \frac{1}{\lam}\log\bigg(\int f^{\lam}g\bigg) - \frac{1}{\lam+1}\log\bigg(\int f^{\lam+1}\bigg) \right\}
\ge 0.
\label{hm:intro.ldp-3}
\end{equation}
This form is seemingly different from \eqref{hm:intro.ldp-5} and would suggest estimating $\tz$ by maximizing 
\begin{equation}\label{hm:def_log-dp.div}
\mbbh_{\log,n}(\theta;\lam) :=
\frac{1}{\lam}\log\bigg( \sumj \phi_j(\theta)^\lam \bigg)
- \frac{1}{\lam+1}\log\bigg(\sumj \int \phi_{j-1}(y;\theta)^{\lam+1}dy\bigg).
\nn
\end{equation}
In the regression context, this type of divergence was considered in \cite{KawFuj17}, while the case of \eqref{hm:intro.ldp-5} was considered earlier by \cite{FujEgu08}.

The partial derivatives of the random functions $\mbbh^{\flat}_n(\theta;\lam)$ and $\mbbh_{\log,n}(\theta;\lam)$ can be respectively written as
\begin{align}
\p_\theta \mbbh^\flat_{n}(\theta;\lam) &= \sumj m_{n,j}(\theta),
\nn\\
\p_\theta \mbbh_{\log,n}(\theta;\lam) &= 
\left\{\sumj \left( \int \phi_{j-1}(y;\theta)^{\lam+1}dy \right)^{\lam/(\lam+1)} \right\}^{-2}
\sumj m_{\log,n,j}(\theta)
\nonumber
\end{align}
for suitable $m_{n,j}(\theta)$ and $m_{\log,n,j}(\theta)$.
Assume temporarily that \eqref{hm:intro.ldp-2} holds with $\mathsf{R}_{j-1}(\theta)= 0$ and that in the identity we can pass the partial differentiation $\p_{\theta}$ under the integral sign so that \eqref{hm:intro.dp-1} holds with $\mathsf{R}_{j-1}(\theta)= 0$; in particular, this can be the case when the statistical model contains the true data-generating distribution.
Then, by direct computations, we can observe the following.
\begin{itemize}
\item On the one hand, we have $\E_\theta^{j-1}[m_{n,j}(\theta)]=0$ if and only if
\begin{align}
\E_\theta^{j-1}[\phi_j(\theta)^\lambda\psi_j(\theta)] = \frac{1}{\lam+1}
\p_\theta \E_\theta^{j-1}[\phi_j(\theta)^\lambda].
\nonumber
\end{align}
This is the case as we have just temporarily assumed; in our model, this holds approximately in the high-frequency regime.

\item On the other hand, however, we have $\E_\theta^{j-1}[m_{\log,n,j}(\theta)]=0$ if and only if the following identity holds:
\begin{align}
& \left(\sumj \p_\theta \int \phi_{j-1}(y;\theta)^{\lam+1}dy \right)
\left(\sumj \left( \int \phi_{j-1}(y;\theta)^{\lam+1}dy \right)^{\lam/(\lam+1)} \right)
\nn\\
&= \left(\sumj \int \phi_{j-1}(y;\theta)^{\lam+1}dy \right)
\nn\\
&{}\qquad \times
\Bigg\{\sumj \left( \int \phi_{j-1}(y;\theta)^{\lam+1}dy \right)^{\lam/(\lam+1)} 
\p_\theta \log\left(\int \phi_{j-1}(y;\theta)^{\lam+1}dy\right)
\Bigg\}.
\nonumber
\end{align}
This does not hold in general, while it does if the model is homogeneous in the sense that the integral $\int \phi_{j-1}(y;\theta)^{\lam+1}dy$ does not depend on $j$.
\end{itemize}
It follows from the above observation that $\mbbh^{\flat}_n(\theta;\lam)$ gives rise to an approximate martingale estimation function, hence so does $\mbbh_n(\theta;\lam)$, while $\mbbh_{\log,n}(\theta;\lam)$ does not. 
Therefore, from a theoretical point of view, we should use $\mbbh_n(\theta;\lam)$ rather than $\mbbh_{\log,n}(\theta;\lam)$.

The previous study \cite{KawFuj23} compared the two versions of the H\"{o}lder-based GQLFs and proved the robustness under heterogeneous heavy contamination in the context of $\gam$-divergence-based regression introduced in \cite{FujEgu08}.
Our observation in the last paragraph demonstrates that the martingale property of the estimating functions (the quasi-score equation) can also explain this point.
\end{rem}

\subsection{Main result}
\label{hm:sec_main.results}

We will state the asymptotic mixed normality of $\tes(\lam)$ defined by \eqref{hm:def_tes.beta} or \eqref{hm:def_tes.gam}:
\begin{equation}
\hat{u}_n(\lambda) :=
\sqrt{n} (\tes(\lambda) - \tz) 
\cil MN_p\left(0,V_0\right),
\label{hm:AMN_pre}
\end{equation}
where the symbol $MN_p(0,V_0)$ denotes the random covariance mixture of the $N_p(0,I_p)$ distribution for an $\mcf$-measurable non-negative definite random covariance matrix $V_0=V_0(\omega)$, possibly depending on $\lam$ when $\lam>0$ is fixed. Specifically, the random variable $\Phi\sim MN_p(0,V_0(\omega))$ is defined on an extended probability space, say $(\overline{\Omega},\overline{\mcf},\overline{\pr})$,  and is characterized by
\begin{equation}
\overline{\E}\left[\exp(i \Phi[u])\right] = \E\left[\exp\left(-\frac12 V_0 [u^{\otimes 2}]\right)\right],\qquad u\in\mbbr^p.
\nonumber
\end{equation}
The specific form of $V_0$ will be given shortly; it takes different forms for \eqref{hm:def_tes.beta} and \eqref{hm:def_tes.gam}.

We introduce further notation. 
For both density-power and H\"{o}lder-based GQLFs, written as $\mbbh_n(\theta;\lam)$, we introduce the following notation:
\begin{align}
\D_n(\lambda) &:= \frac{1}{\sqrt{n}} \p_\theta \mbbh_n(\tz;\lambda),
\label{hm:def_Del.n}\\
\Gam_n(\lambda) &:= - \frac1n \p_{\theta}^{2}\mbbh_{n}(\tz;\lambda).
\label{hm:def_Gam.n}
\end{align}
These correspond respectively to the quasi-score and quasi-observed information matrix associated with the robustified GQLFs.

With a slight abuse of notation, we will write
\begin{equation}\label{hm:gen.notation_f}
    \mathsf{f}_{j-1}^\star(\theta) = \mathsf{f}(\lx_{t_{j-1}},\theta),
    \qquad \mathsf{f}_{t}^\star(\theta) = \mathsf{f}(\lx_{t},\theta),
    \qquad \mathsf{f}_{t}^\star = \mathsf{f}(\lx_{t},\tz)
\end{equation}
for any measurable function $\mathsf{f}$ defined on $\mbbr^{d'}\times\overline{\Theta}$, which may depend on $\lam$.
Let 
\begin{equation}
    \Xi^\star_{t,k}:={S^\star_t}^{-1} (\p_{\theta_k}S^\star_t),\qquad k=1,\dots,p,
\end{equation}
and let $\Gam_0(\lambda)=(\Gam_{0,kl}(\lambda))_{k,l=1}^{p}$ and $\Sig_0(\lambda)=(\Sig_{0,kl}(\lambda))_{k,l=1}^{p}$ denote the random matrices in $\mbbr^p\otimes \mbbr^p$ given by:
\begin{align}
    \Gam_{0,kl}(\lam) 
    &= \left(\frac{K_{\lam,d}}{\lam+1}\right) \frac{1}{2T}\int_0^T {\mathsf{d}^\star_t}^{-\lam/2} \bigg\{ \trace\left(\Xi^\star_{t,k}\,\Xi^\star_{t,l}\right) 
    \nn\\
    &{}\qquad + \frac{\lam^2}{2} 
    \trace\left(\Xi^\star_{t,k} \right)
    \trace\left(\Xi^\star_{t,l} \right)
    \bigg\} dt,
    \label{hm:Gam0_dp}\\
    \Sig_{0,kl}(\lam) &= \frac{K_{2\lam,d}}{2\lam+1} \frac{1}{2T} \int_0^T 
    {\mathsf{d}^\star_t}^{-\lam} 
    \trace\left(\Xi^\star_{t,k}\,\Xi^\star_{t,l}\right) dt 
    \nn\\
    &{}\quad + \frac{\ep'(\lam)}{4}
    \frac1T \int_0^T {\mathsf{d}^\star_t}^{-\lam} 
    \trace\left(\Xi^\star_{t,k} \right)
    \trace\left(\Xi^\star_{t,l} \right)
    dt
    \label{hm:Sig0_dp}
\end{align}
for the density-power case, and
\begin{align}
    \Gam_{0,kl}(\lam) 
    &= 
    \left(\frac{K_{\lam,d}}{\lam+1}\right) 
    \frac{1}{2T}\int_0^T {\mathsf{d}^\star_t}^{-\frac{\lam}{2(\lam+1)}} \trace\left(\Xi^\star_{t,k}\,\Xi^\star_{t,l}\right) dt,
    \label{hm:Gam0_ldp}\\
    \Sig_{0,kl}(\lam) &= \frac{K_{2\lam,d}}{2\lam+1} \frac{1}{2T} \int_0^T 
    {\mathsf{d}^\star_t}^{-\frac{\lam}{\lam+1}} 
    \trace\left(\Xi^\star_{t,k}\,\Xi^\star_{t,l}\right) dt 
    \nn\\
    &{}\quad + \ep''(\lam) 
    \frac1T \int_0^T {\mathsf{d}^\star_t}^{-\frac{\lam}{\lam+1}} 
    \trace\left(\Xi^\star_{t,k} \right)
    \trace\left(\Xi^\star_{t,l} \right)
    dt
    \label{hm:Sig0_ldp}
\end{align}
for the H\"{o}lder-based case, 
where
\begin{align}
    \ve'(\lam) &:= \left(\frac{1}{2(\lam+1)} + 2\lam -1\right)K_{2\lam,d} - \lam^2 K_{\lam,d}^2,
    \nn\\
    \ve''(\lam) &:= \frac14 \left(\frac{1}{2\lam+1} - \frac{1}{(\lam+1)^2} \right) K_{2\lam,d}^2.
\end{align}
Obviously, $\lim_{\lam\downarrow 0}\ve'(\lam) = \lim_{\lam\downarrow 0}\ve''(\lam) =0$.
Finally, we define $\mci(\tz)=(\mci_{kl}(\tz))_{k,l=1}^{p}$ by
\begin{equation}\label{hm:def_FI.lim}
    \mci_{kl}(\tz) = \frac{1}{2\,T} \int_0^T \text{trace}\left( (S^{-1} (\p_{\theta_k}S) S^{-1} (\p_{\theta_l} S))_t \right) dt,
\end{equation}
which corresponds to the Fisher-information matrix (see \cite{UchYos13} and the references therein).

The following theorem is the main result of this paper.
The proof is given in Appendices \ref{hm:sec_aux.asymp} and \ref{hm:sec_main.results.proof}.

\begin{thm}
\label{hm:thm_AMN}
Suppose that Assumptions 
\ref{hm:A_diff.coeff} to \ref{hm:A_lam} 
hold, and let $\eta \sim N_p(0,I_p)$ being independent of $\mcf$.
Then, we have the following.
\begin{enumerate}
    \item For the density-power case:
    \begin{enumerate}
        \item Under Assumption \ref{hm:A_lam}(1), if $\Gam_0(\lam)$ defined by \eqref{hm:Gam0_dp} is a.s. positive definite for $\lam\in(0,\overline{\lam}]$, then
        \begin{align}
        \hat{u}_n(\lambda) 
        &= \Gam_{0}(\lambda)^{-1} \D_{n}(\lambda) + o_p(1) \nn\\
        &\cil \Gam_{0}(\lambda)^{-1} \Sigma_{0}(\lambda)^{1/2}\eta 
        \nn\\
        &\sim MN_{p}\left(0,\Gam_0(\lambda)^{-1}\Sig_0(\lambda)\Gam_0(\lambda)^{-1}\right)
        \label{hm:AMN(1)}
        \end{align}
        with $\Sig_0(\lam)$ given by \eqref{hm:Sig0_dp};
        
        \item Under Assumption \ref{hm:A_lam}(2),
        \begin{align}
        \hat{u}_n(\lambda) 
        &= \Gam_{0}(\lambda)^{-1} \D_{n}(\lambda) + o_p(1) \nn\\
        &\cil \mci(\tz)^{-1/2}\eta \sim MN_{p}\left(0,\mci(\tz)^{-1}\right).
        \label{hm:AMN(2)}
        \end{align}
    \end{enumerate}
    
    \item For the H\"{o}lder-based case:
    \begin{enumerate}
        \item Under Assumption \ref{hm:A_lam}(1), if $\Gam_0(\lam)$ defined by \eqref{hm:Gam0_ldp} is a.s. positive definite for $\lam\in(0,\overline{\lam}]$, then \eqref{hm:AMN(1)} holds with $\Sig_0(\lam)$ given by \eqref{hm:Sig0_ldp};
        \item Under Assumption \ref{hm:A_lam}(2), \eqref{hm:AMN(2)} holds.        
    \end{enumerate}
\end{enumerate}
% ...
    
\end{thm}

It is straightforward to construct consistent estimators of $\Gam_0(\lam)$, $\Sig_0(\lam)$, and $\mci(\tz)$ through the following simple fact:
\begin{equation}
    \sup_{\lam\in(0,\overline{\lam}]} \left|\frac1n \sumj \mathsf{g}(X_s,\tes(\lam);\lam) - \frac1T \int_0^T \mathsf{g}(X_s,\tz;\lam)dt\right| \cip 0 \label{se:conti.map}
\end{equation}
for any measurable function $\mathsf{g}$ smooth enough. This readily provides us with a practical recipe for constructing an approximate confidence set.

In Theorem \ref{hm:thm_AMN}(1), we did not guarantee positive definiteness of $\Sig_0(\lambda)$. Still, it could be the case for $\lam>0$ small enough as soon as $\mci(\tz)$ is a.s. positive definite, thanks to the (right-)continuity of $\lam \mapsto (\Gam_0(\lam), \Sig_0(\lam))$.

We end with the following remark.

\begin{rem}\normalfont
\label{hm:rem_G0j}
The Gaussian tapering through $\vp_j(\theta)^\lam$ is simple enough and leads to the intuitively interpretable ``divergence-based'' framework, providing a single-parameter-tuning estimation procedure robust against ``non-continuous'' transition.
Roughly speaking, the essence of the proofs is that for each $\mbbi_j$ we can construct a good event $G_j$,
\begin{itemize}
\item on which the process $Y$ obeys the (ideal) continuous semimartingale, and
\item whose probability is close to $1$.
\end{itemize}
Since our proofs do not utilize any specific structure of the jump and spike components, the proposed estimation strategy should be robust against other types of non-continuous transitions and can be applied in analogous ways to many different types of random dynamical systems, such as diffusions with small noise and ergodic diffusions contaminated by jumps and spike noises.
\end{rem}

\section{Numerical experiments}
\label{sec_simulations}

In this section, we present simulation results to observe the finite-sample performance of the density-power and the H\"{o}lder-based GQMLEs.
We use the {\tt yuima} package in R (see \cite{YUIMA14}) to generate data and compute the GQMLE.
All Monte Carlo trials are based on 1000 independent sample paths, and the simulations are done for $n=1000$ and $5000$ with $T=1$.
\serev{For $\lambda=0.2$, we have $\sqrt{n}h/\lambda\fallingdotseq0.16$ for $n=1000$ and $\sqrt{n}h/\lambda\fallingdotseq0.07$ for $n=5000$, so that we may regard \eqref{hm:lam.condition-2} of Assumption \ref{hm:A_lam} with $\kappa=1$ approximately holds.}

\subsection{Time-inhomogeneous Wiener process}
\label{se:simu1}

We suppose that 
\begin{align*}
\lx_{t_{j}}&=(\lx_{1,t_{j}},\lx_{2,t_{j}},\lx_{3,t_{j}})^{\top}=\left(\cos\left(\frac{2j\pi}{n}\right),\sin\left(\frac{2j\pi}{n}\right),\cos\left(\frac{4j\pi}{n}\right)\right)^{\top}.
\end{align*}
and that $(\ly_t)_{t\le 1}$ is a solution to the stochastic regression model
\begin{align*}
d\ly_{t}&=\exp\left\{\frac{1}{2}(-2\lx_{1,t}+3\lx_{2,t})\right\}dw_{t}+s\;dJ_{t}, \quad \ly_{0}=0, \quad t\in[0,1],
\end{align*}
where $s$ takes $0$ or $1$, and $J$ is a compound Poisson process with intensity $q$.
We consider the following model for the estimation:
\begin{align*}
d\ly_{t}=\exp\left\{\frac{1}{2}(\theta_{1}\lx_{1,t}+\theta_{2} \lx_{2,t}+\theta_{3}\lx_{3,t})\right\}dw_{t}, \qquad t\in[0,1].
\end{align*}
Then, the true parameter $\tz=(\theta_{1,0},\theta_{2,0},\theta_{3,0})^{\top}=(-2,3,0)^{\top}$.
We set the initial value, lower bound, and upper bound in numerical optimization $0$, $-10$, and $10$, respectively.

\subsubsection{With some spikes} \label{se:simu12}

We set $s=0$ and $X_{t_{j}}=\lx_{t_{j}}$, and consider that the sample $(\ly_{t_{j}})_{j=0}^{n}$ is the original data.
Moreover, we deal with the situation, which is similar to the settings of \cite{LeeSon13,Son20}, that the observed random variable $(Y_{t_{j}})_{j=0}^{n}$ includes the outliers $(Y_{c,t_{j}})\sim^{\text{i.i.d.}}\serev{\mathcal{N}}(0,\sigma^{2})$ and that $Y_{t_{j}}$ is given by the scheme $Y_{t_{j}}=\ly_{t_{j}}+p_{j}Y_{c,t_{j}}$, where $p_{0},p_{1},\ldots,p_{n}$ are random variables that independently follow the Bernoulli distribution with a success probability of $\mathfrak{p}$.
We assume that $(p_{j})$, $(\ly_{t_{j}})$, and $(Y_{c,t_{j}})$ are independent, and the simulations are done for $\sigma^{2}=1,3$, and $\mathfrak{p}=0.01,0.05$.
Figure \ref{pathplot1} shows one of 1000 sample paths for $\sigma^{2}=1, \mathfrak{p}=0.01$, and $n=5000$.

Tables \ref{esti13} and \ref{esti14} summarize the estimation results using GQMLE, density-power GQMLE, and H\"{o}lder-based GQMLE \serev{for $\mathfrak{p}= 0.01$ and $\mathfrak{p}= 0.05$ respectively}.
In the original data, all estimators are close to the true parameter values.
In the two types of arranged data(spike data), the performance of GQMLE is different from the case of original data, while the performances of density-power and H\"{o}lder-based GQMLEs are similar to those in the original data.
Comparing the calculation times of density-power and H\"{o}lder-based GQMLEs, we observed that the H\"{o}lder-based GQMLE could be computed faster than the density-power GQMLE.

Figures \ref{estplot13} and \ref{holestplot13} show the means and standard deviations of the density-power and H\"{o}lder-based GQMLEs in each $\lambda$, and both density-power and H\"{o}lder-based GQMLEs have similar tendencies.
We can observe that estimators $\hat{\theta}_{1,n}(\lambda)$, $\hat{\theta}_{2,n}(\lambda)$, and $\hat{\theta}_{3,n}(\lambda)$ seem to perform best when $\lambda=0.2$.
We can also observe that there is a difference between the true value and the estimators of $\theta_{3}$ if $\lambda$ becomes too large.
Moreover, the standard deviations of the estimators increase as $\lambda$ increases.

For the density-power and H\"{o}lder-based GQMLEs, we define
\begin{align}
u_{i,n}(\lambda)
&=\left\{\left(\Gamma_{n}(\hat{\theta}_{n}(\lambda);\lambda)^{-1}\Sigma_{n}(\hat{\theta}_{n}(\lambda);\lambda)\Gamma_{n}(\hat{\theta}_{n}(\lambda);\lambda)^{-1}\right)_{ii}\right\}^{-1/2} \\
&\qquad \times \sqrt{n}(\hat{\theta}_{i,n}(\lambda)-\theta_{i,0}).
\label{se: asym.norm}
\end{align}
Taking \eqref{se:conti.map} into account, $\Gamma_{n}(\hat{\theta}_{n}(\lambda);\lambda)$ and $\Sigma_{n}(\hat{\theta}_{n}(\lambda);\lambda)$ are constructed so as to converge in probability to $\Gamma_{0}(\lambda)$ and $\Sigma_{0}(\lambda)$ given in \eqref{hm:Gam0_dp}--\eqref{hm:Sig0_ldp}, respectively. 
Figures \ref{plot12} and \ref{holplot12} give the histograms of $u_{1,n}(\lambda)$, $u_{2,n}(\lambda)$, and $u_{3,n}(\lambda)$ in the case of $\lambda=0.2$.
Figure \ref{plot12} is based on the 1000th sample data and density-power GQMLE, and Figure \ref{holplot12} is based on the 1000th sample data and H\"{o}lder-based GQMLE.
From these figures, we can see the appropriate performance of the theoretical claims in Theorem \ref{hm:thm_AMN}.

Tables \ref{ciratio12} and \ref{holciratio12} show the empirical frequencies with which the true values of $\theta_{1}$, $\theta_{2}$, and $\theta_{3}$ are included in the 95\% confidence intervals in each $\lambda$ for density-power and H\"{o}lder-based GQMLEs, respectively.
The estimated frequencies were generally accurate for $\lam\ge 0.2$, while they fell below $0.95$ when $\lam=0.1$.


\begin{figure}[t]
\begin{tabular}{c}
\includegraphics[scale=0.35]{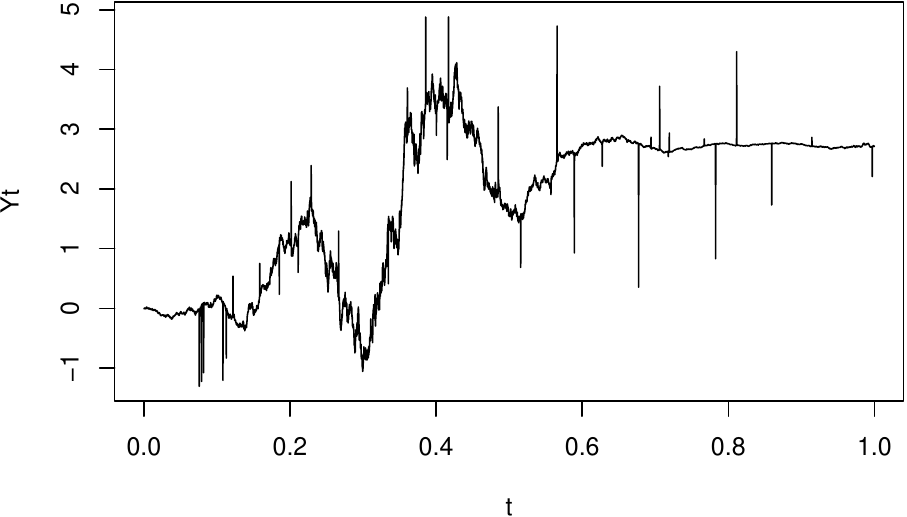}
\end{tabular}
\caption{One of 1000 sample paths in Section \ref{se:simu12} ($\mathfrak{p}=0.01$, $n=5000$).}
\label{pathplot1}
\end{figure}


\begin{table}[t]
\begin{center}
\caption{GQMLE, density-power GQMLE, and H\"{o}lder-based GQMLE in Section \ref{se:simu12} ($\theta_{0}=(-2,3,0)^{\top}$, $\mathfrak{p}=0.01$). ``time'' shows the mean calculation time in the cases where $n=5000$.}
\scalebox{0.7}[0.7]{
\serev{
\begin{tabular}{r r | r r r | r r r | r r r} \hline
& & \multicolumn{3}{c|}{} & \multicolumn{3}{c|}{} & \multicolumn{3}{c}{} \\[-3mm]
& & \multicolumn{3}{l|}{original} & \multicolumn{3}{l|}{spike: $\sigma^{2}=1$} & \multicolumn{3}{l}{spike: $\sigma^{2}=3$} \\[1mm] \hline
& & \multicolumn{3}{c|}{} & \multicolumn{3}{c|}{} & \multicolumn{3}{c}{} \\[-3mm]
\multicolumn{1}{l}{GQMLE} & & $\hat{\theta}_{1,n}$ & $\hat{\theta}_{2,n}$ & $\hat{\theta}_{3,n}$ & $\hat{\theta}_{1,n}$ & $\hat{\theta}_{2,n}$ & $\hat{\theta}_{3,n}$ & $\hat{\theta}_{1,n}$ & $\hat{\theta}_{2,n}$ & $\hat{\theta}_{3,n}$ \\ \hline
& & \multicolumn{3}{c|}{} & \multicolumn{3}{c|}{} & \multicolumn{3}{c}{} \\[-3mm]
$n=1000$ & mean & -2.0101 & 2.9944 & -0.0005 & -0.6385 & 0.9202 & -0.0824 & -0.4788 & 0.6723 & -0.0452 \\[1mm]
& s.d. & 0.0646 & 0.0627 & 0.0602 & 0.9904 & 1.0060 & 0.9714 & 1.1704 & 1.1684 & 1.1595 \\[1mm]
$n=5000$ & mean & -2.0013 & 2.9981 & 0.0015 & -0.1243 & 0.1429 & -0.0353 & -0.0712 & 0.0598 & -0.0124 \\[1mm]
& s.d. & 0.0283 & 0.0281 & 0.0282 & 0.3682 & 0.3591 & 0.3768 & 0.3852 & 0.3820 & 0.3979 \\[1mm] 
& & \multicolumn{3}{r|}{(time: 0.2651)} & \multicolumn{3}{r|}{(time: 0.2065)} & \multicolumn{3}{r}{(time: 0.2074)} \\[1mm] \hline
& & \multicolumn{3}{c|}{} & \multicolumn{3}{c|}{} & \multicolumn{3}{c}{} \\[-3mm]
\multicolumn{1}{l}{Density-power:} & $\lambda=0.1$ & $\hat{\theta}_{1,n}(\lambda)$ & $\hat{\theta}_{2,n}(\lambda)$ & $\hat{\theta}_{3,n}(\lambda)$ & $\hat{\theta}_{1,n}(\lambda)$ & $\hat{\theta}_{2,n}(\lambda)$ & $\hat{\theta}_{3,n}(\lambda)$ & $\hat{\theta}_{1,n}(\lambda)$ & $\hat{\theta}_{2,n}(\lambda)$ & $\hat{\theta}_{3,n}(\lambda)$ \\ \hline
& & \multicolumn{3}{c|}{} & \multicolumn{3}{c|}{} & \multicolumn{3}{c}{} \\[-3mm]
$n=1000$ & mean & -1.9917 & 3.0077 & -0.0005 & -2.0122 & 3.0379 & -0.0041 & -2.0067 & 3.0297 & -0.0041 \\[1mm]
& s.d. & 0.0656 & 0.0638 & 0.0613 & 0.0723 & 0.0715 & 0.0689 & 0.0709 & 0.0700 & 0.0672 \\[1mm]
$n=5000$ & mean & -1.9974 & 3.0006 & 0.0016 & -2.0089 & 3.0181 & -0.0021 & -2.0042 & 3.0115 & -0.0013 \\[1mm]
& s.d. & 0.0287 & 0.0284 & 0.0289 & 0.0305 & 0.0306 & 0.0306 & 0.0301 & 0.0300 & 0.0300 \\[1mm] 
& & \multicolumn{3}{r|}{(time: 0.5397)} & \multicolumn{3}{r|}{(time: 0.5334)} & \multicolumn{3}{r}{(time: 0.5327)} \\[1mm] \hline
& & \multicolumn{3}{c|}{} & \multicolumn{3}{c|}{} & \multicolumn{3}{c}{} \\[-3mm]
& $\lambda=0.5$ & $\hat{\theta}_{1,n}(\lambda)$ & $\hat{\theta}_{2,n}(\lambda)$ & $\hat{\theta}_{3,n}(\lambda)$ & $\hat{\theta}_{1,n}(\lambda)$ & $\hat{\theta}_{2,n}(\lambda)$ & $\hat{\theta}_{3,n}(\lambda)$ & $\hat{\theta}_{1,n}(\lambda)$ & $\hat{\theta}_{2,n}(\lambda)$ & $\hat{\theta}_{3,n}(\lambda)$ \\ \hline
& & \multicolumn{3}{c|}{} & \multicolumn{3}{c|}{} & \multicolumn{3}{c}{} \\[-3mm]
$n=1000$ & mean & -1.9920 & 3.0091 & -0.0007 & -1.9937 & 3.0115 & -0.0003 & -1.9873 & 3.0010 & 0.0007 \\[1mm]
& s.d. & 0.0821 & 0.0825 & 0.0795 & 0.0809 & 0.0813 & 0.0772 & 0.0833 & 0.0841 & 0.0803 \\[1mm]
$n=5000$ & mean & -1.9968 & 2.9999 & 0.0019 & -1.9916 & 2.9920 & 0.0022 & -1.9914 & 2.9917 & 0.0023 \\[1mm]
& s.d. & 0.0360 & 0.0354 & 0.0361 & 0.0361 & 0.0356 & 0.0366 & 0.0360 & 0.0355 & 0.0366 \\[1mm] 
& & \multicolumn{3}{r|}{(time: 0.6450)} & \multicolumn{3}{r|}{(time: 0.6347)} & \multicolumn{3}{r}{(time: 0.6346)} \\[1mm] \hline
& & \multicolumn{3}{c|}{} & \multicolumn{3}{c|}{} & \multicolumn{3}{c}{} \\[-3mm]
& $\lambda=0.9$ & $\hat{\theta}_{1,n}(\lambda)$ & $\hat{\theta}_{2,n}(\lambda)$ & $\hat{\theta}_{3,n}(\lambda)$ & $\hat{\theta}_{1,n}(\lambda)$ & $\hat{\theta}_{2,n}(\lambda)$ & $\hat{\theta}_{3,n}(\lambda)$ & $\hat{\theta}_{1,n}(\lambda)$ & $\hat{\theta}_{2,n}(\lambda)$ & $\hat{\theta}_{3,n}(\lambda)$ \\ \hline
& & \multicolumn{3}{c|}{} & \multicolumn{3}{c|}{} & \multicolumn{3}{c}{} \\[-3mm]
$n=1000$ & mean & -1.9912 & 3.0110 & -0.0016 & -1.9917 & 3.0107 & 0.0005 & -1.9797 & 2.9881 & 0.0033 \\[1mm]
& s.d. & 0.1035 & 0.1058 & 0.1112 & 0.1013 & 0.1043 & 0.1001 & 0.1053 & 0.1076 & 0.1123 \\[1mm]
$n=5000$ & mean & -1.9966 & 2.9993 & 0.0027 & -1.9845 & 2.9765 & 0.0062 & -1.9845 & 2.9765 & 0.0062 \\[1mm]
& s.d. & 0.0457 & 0.0449 & 0.0489 & 0.0456 & 0.0448 & 0.0497 & 0.0456 & 0.0448 & 0.0496 \\[1mm] 
& & \multicolumn{3}{r|}{(time: 0.6856)} & \multicolumn{3}{r|}{(time: 0.6801)} & \multicolumn{3}{r}{(time: 0.6752)} \\[1mm]  \hline
& & \multicolumn{3}{c|}{} & \multicolumn{3}{c|}{} & \multicolumn{3}{c}{} \\[-3mm]
\multicolumn{1}{l}{H\"{o}lder-based:} & $\lambda=0.1$ & $\hat{\theta}_{1,n}(\lambda)$ & $\hat{\theta}_{2,n}(\lambda)$ & $\hat{\theta}_{3,n}(\lambda)$ & $\hat{\theta}_{1,n}(\lambda)$ & $\hat{\theta}_{2,n}(\lambda)$ & $\hat{\theta}_{3,n}(\lambda)$ & $\hat{\theta}_{1,n}(\lambda)$ & $\hat{\theta}_{2,n}(\lambda)$ & $\hat{\theta}_{3,n}(\lambda)$ \\ \hline
& & \multicolumn{3}{c|}{} & \multicolumn{3}{c|}{} & \multicolumn{3}{c}{} \\[-3mm]
$n=1000$ & mean & -1.9917 & 3.0076 & -0.0004 & -2.0122 & 3.0379 & -0.0041 & -2.0074 & 3.0307 & -0.0042 \\[1mm]
& s.d. & 0.0656 & 0.0637 & 0.0612 & 0.0723 & 0.0715 & 0.0689 & 0.0708 & 0.0700 & 0.0671 \\[1mm]
$n=5000$ & mean & -1.9974 & 3.0006 & 0.0015 & -2.0095 & 3.0190 & -0.0022 & -2.0046 & 3.0121 & -0.0014 \\[1mm]
& s.d. & 0.0286 & 0.0283 & 0.0288 & 0.0305 & 0.0306 & 0.0305 & 0.0301 & 0.0300 & 0.0299 \\[1mm] 
& & \multicolumn{3}{r|}{(time: 0.3214)} & \multicolumn{3}{r|}{(time: 0.3143)} & \multicolumn{3}{r}{(time: 0.3141)} \\[1mm] \hline
& & \multicolumn{3}{c|}{} & \multicolumn{3}{c|}{} & \multicolumn{3}{c}{} \\[-3mm]
& $\lambda=0.5$ & $\hat{\theta}_{1,n}(\lambda)$ & $\hat{\theta}_{2,n}(\lambda)$ & $\hat{\theta}_{3,n}(\lambda)$ & $\hat{\theta}_{1,n}(\lambda)$ & $\hat{\theta}_{2,n}(\lambda)$ & $\hat{\theta}_{3,n}(\lambda)$ & $\hat{\theta}_{1,n}(\lambda)$ & $\hat{\theta}_{2,n}(\lambda)$ & $\hat{\theta}_{3,n}(\lambda)$ \\ \hline
& & \multicolumn{3}{c|}{} & \multicolumn{3}{c|}{} & \multicolumn{3}{c}{} \\[-3mm]
$n=1000$ & mean & -1.9916 & 3.0082 & 0.0001 & -1.9937 & 3.0115 & -0.0003 & -1.9927 & 3.0100 & 0.0001 \\[1mm]
& s.d. & 0.0797 & 0.0794 & 0.0763 & 0.0809 & 0.0813 & 0.0772 & 0.0806 & 0.0810 & 0.0770 \\[1mm]
$n=5000$ & mean & -1.9967 & 2.9999 & 0.0015 & -1.9974 & 3.0018 & 0.0007 & -1.9970 & 3.0011 & 0.0009 \\[1mm]
& s.d. & 0.0351 & 0.0340 & 0.0346 & 0.0353 & 0.0342 & 0.0351 & 0.0352 & 0.0341 & 0.0351 \\[1mm] 
& & \multicolumn{3}{r|}{(time: 0.3556)} & \multicolumn{3}{r|}{(time: 0.3495)} & \multicolumn{3}{r}{(time: 0.3529)} \\[1mm] \hline
& & \multicolumn{3}{c|}{} & \multicolumn{3}{c|}{} & \multicolumn{3}{c}{} \\[-3mm]
& $\lambda=0.9$ & $\hat{\theta}_{1,n}(\lambda)$ & $\hat{\theta}_{2,n}(\lambda)$ & $\hat{\theta}_{3,n}(\lambda)$ & $\hat{\theta}_{1,n}(\lambda)$ & $\hat{\theta}_{2,n}(\lambda)$ & $\hat{\theta}_{3,n}(\lambda)$ & $\hat{\theta}_{1,n}(\lambda)$ & $\hat{\theta}_{2,n}(\lambda)$ & $\hat{\theta}_{3,n}(\lambda)$ \\ \hline
& & \multicolumn{3}{c|}{} & \multicolumn{3}{c|}{} & \multicolumn{3}{c}{} \\[-3mm]
$n=1000$ & mean & -1.9908 & 3.0093 & 0.0003 & -1.9917 & 3.0107 & 0.0005 & -1.9913 & 3.0100 & 0.0007 \\[1mm]
& s.d. & 0.0997 & 0.1021 & 0.0989 & 0.1013 & 0.1043 & 0.1001 & 0.1011 & 0.1040 & 0.1000 \\[1mm]
$n=5000$ & mean & -1.9963 & 2.9994 & 0.0019 & -1.9963 & 3.0002 & 0.0014 & -1.9962 & 2.9999 & 0.0015 \\[1mm]
& s.d. & 0.0446 & 0.0430 & 0.0433 & 0.0447 & 0.0430 & 0.0440 & 0.0447 & 0.0430 & 0.0440 \\[1mm] 
& & \multicolumn{3}{r|}{(time: 0.3599)} & \multicolumn{3}{r|}{(time: 0.3529)} & \multicolumn{3}{r}{(time: 0.3530)} \\[1mm] \hline
\end{tabular}
}
}
\label{esti13}
\end{center}
\end{table}


\begin{table}[t]
\begin{center}
\caption{GQMLE, density-power GQMLE, and H\"{o}lder-based GQMLE in Section \ref{se:simu12} ($\theta_{0}=(-2,3,0)^{\top}$, $\mathfrak{p}=0.05$). ``time'' shows the mean calculation time in the cases where $n=5000$.}
\scalebox{0.7}[0.7]{
\serev{
\begin{tabular}{r r | r r r | r r r | r r r} \hline
& & \multicolumn{3}{c|}{} & \multicolumn{3}{c|}{} & \multicolumn{3}{c}{} \\[-3mm]
& & \multicolumn{3}{l|}{original} & \multicolumn{3}{l|}{spike: $\sigma^{2}=1$} & \multicolumn{3}{l}{spike: $\sigma^{2}=3$} \\[1mm] \hline
& & \multicolumn{3}{c|}{} & \multicolumn{3}{c|}{} & \multicolumn{3}{c}{} \\[-3mm]
\multicolumn{1}{l}{GQMLE} & & $\hat{\theta}_{1,n}$ & $\hat{\theta}_{2,n}$ & $\hat{\theta}_{3,n}$ & $\hat{\theta}_{1,n}$ & $\hat{\theta}_{2,n}$ & $\hat{\theta}_{3,n}$ & $\hat{\theta}_{1,n}$ & $\hat{\theta}_{2,n}$ & $\hat{\theta}_{3,n}$ \\ \hline
& & \multicolumn{3}{c|}{} & \multicolumn{3}{c|}{} & \multicolumn{3}{c}{} \\[-3mm]
$n=1000$ & mean & -2.0101 & 2.9944 & -0.0005 & -0.0950 & 0.1366 & -0.0319 & -0.0407 & 0.0533 & -0.0080 \\[1mm]
& s.d. & 0.0646 & 0.0627 & 0.0602 & 0.3711 & 0.3450 & 0.3695 & 0.3882 & 0.3663 & 0.3918 \\[1mm]
$n=5000$ & mean & -2.0013 & 2.9981 & 0.0015 & -0.0290 & 0.0310 & -0.0048 & -0.0189 & 0.0158 & -0.0004 \\[1mm]
& s.d. & 0.0283 & 0.0281 & 0.0282 & 0.1547 & 0.1543 & 0.1562 & 0.1560 & 0.1561 & 0.1578 \\[1mm] 
& & \multicolumn{3}{r|}{(time: 0.2638)} & \multicolumn{3}{r|}{(time: 0.1857)} & \multicolumn{3}{r}{(time: 0.1861)} \\[1mm] \hline
& & \multicolumn{3}{c|}{} & \multicolumn{3}{c|}{} & \multicolumn{3}{c}{} \\[-3mm]
\multicolumn{1}{l}{Density-power:} & $\lambda=0.1$ & $\hat{\theta}_{1,n}(\lambda)$ & $\hat{\theta}_{2,n}(\lambda)$ & $\hat{\theta}_{3,n}(\lambda)$ & $\hat{\theta}_{1,n}(\lambda)$ & $\hat{\theta}_{2,n}(\lambda)$ & $\hat{\theta}_{3,n}(\lambda)$ & $\hat{\theta}_{1,n}(\lambda)$ & $\hat{\theta}_{2,n}(\lambda)$ & $\hat{\theta}_{3,n}(\lambda)$ \\ \hline
& & \multicolumn{3}{c|}{} & \multicolumn{3}{c|}{} & \multicolumn{3}{c}{} \\[-3mm]
$n=1000$ & mean & -1.9917 & 3.0077 & -0.0005 & -2.1008 & 3.1735 & -0.0244 & -2.0698 & 3.1293 & -0.0262 \\[1mm]
& s.d. & 0.0656 & 0.0638 & 0.0613 & 0.1007 & 0.1092 & 0.1044 & 0.0910 & 0.0984 & 0.0947 \\[1mm]
$n=5000$ & mean & -1.9974 & 3.0006 & 0.0016 & -2.0581 & 3.0965 & -0.0188 & -2.0335 & 3.0582 & -0.0116 \\[1mm]
& s.d. & 0.0287 & 0.0284 & 0.0289 & 0.0394 & 0.0404 & 0.0402 & 0.0360 & 0.0361 & 0.0362 \\[1mm] 
& & \multicolumn{3}{r|}{(time: 0.5342)} & \multicolumn{3}{r|}{(time: 0.5286)} & \multicolumn{3}{r}{(time: 0.5283)} \\[1mm] \hline
& & \multicolumn{3}{c|}{} & \multicolumn{3}{c|}{} & \multicolumn{3}{c}{} \\[-3mm]
& $\lambda=0.5$ & $\hat{\theta}_{1,n}(\lambda)$ & $\hat{\theta}_{2,n}(\lambda)$ & $\hat{\theta}_{3,n}(\lambda)$ & $\hat{\theta}_{1,n}(\lambda)$ & $\hat{\theta}_{2,n}(\lambda)$ & $\hat{\theta}_{3,n}(\lambda)$ & $\hat{\theta}_{1,n}(\lambda)$ & $\hat{\theta}_{2,n}(\lambda)$ & $\hat{\theta}_{3,n}(\lambda)$ \\ \hline
& & \multicolumn{3}{c|}{} & \multicolumn{3}{c|}{} & \multicolumn{3}{c}{} \\[-3mm]
$n=1000$ & mean & -1.9920 & 3.0091 & -0.0007 & -2.0025 & 3.0254 & -0.0031 & -1.9661 & 2.9663 & 0.0043 \\[1mm]
& s.d. & 0.0821 & 0.0825 & 0.0795 & 0.0848 & 0.0876 & 0.0828 & 0.0860 & 0.0893 & 0.0845 \\[1mm]
$n=5000$ & mean & -1.9968 & 2.9999 & 0.0019 & -1.9699 & 2.9577 & 0.0049 & -1.9692 & 2.9563 & 0.0053 \\[1mm]
& s.d. & 0.0360 & 0.0354 & 0.0361 & 0.0381 & 0.0370 & 0.0381 & 0.0378 & 0.0370 & 0.0379 \\[1mm] 
& & \multicolumn{3}{r|}{(time: 0.6399)} & \multicolumn{3}{r|}{(time: 0.6275)} & \multicolumn{3}{r}{(time: 0.6273)} \\[1mm] \hline
& & \multicolumn{3}{c|}{} & \multicolumn{3}{c|}{} & \multicolumn{3}{c}{} \\[-3mm]
& $\lambda=0.9$ & $\hat{\theta}_{1,n}(\lambda)$ & $\hat{\theta}_{2,n}(\lambda)$ & $\hat{\theta}_{3,n}(\lambda)$ & $\hat{\theta}_{1,n}(\lambda)$ & $\hat{\theta}_{2,n}(\lambda)$ & $\hat{\theta}_{3,n}(\lambda)$ & $\hat{\theta}_{1,n}(\lambda)$ & $\hat{\theta}_{2,n}(\lambda)$ & $\hat{\theta}_{3,n}(\lambda)$ \\ \hline
& & \multicolumn{3}{c|}{} & \multicolumn{3}{c|}{} & \multicolumn{3}{c}{} \\[-3mm]
$n=1000$ & mean & -1.9912 & 3.0110 & -0.0016 & -1.9950 & 3.0151 & 0.0002 & -1.9274 & 2.8900 & 0.0203 \\[1mm]
& s.d. & 0.1035 & 0.1058 & 0.1112 & 0.1059 & 0.1109 & 0.1054 & 0.1077 & 0.1119 & 0.1138 \\[1mm]
$n=5000$ & mean & -1.9966 & 2.9993 & 0.0027 & -1.9324 & 2.8810 & 0.0215 & -1.9325 & 2.8809 & 0.0216 \\[1mm]
& s.d. & 0.0457 & 0.0449 & 0.0489 & 0.0477 & 0.0461 & 0.0508 & 0.0476 & 0.0461 & 0.0506 \\[1mm] 
& & \multicolumn{3}{r|}{(time: 0.6807)} & \multicolumn{3}{r|}{(time: 0.6553)} & \multicolumn{3}{r}{(time: 0.6661)} \\[1mm] \hline
& & \multicolumn{3}{c|}{} & \multicolumn{3}{c|}{} & \multicolumn{3}{c}{} \\[-3mm]
\multicolumn{1}{l}{H\"{o}lder-based:} & $\lambda=0.1$ & $\hat{\theta}_{1,n}(\lambda)$ & $\hat{\theta}_{2,n}(\lambda)$ & $\hat{\theta}_{3,n}(\lambda)$ & $\hat{\theta}_{1,n}(\lambda)$ & $\hat{\theta}_{2,n}(\lambda)$ & $\hat{\theta}_{3,n}(\lambda)$ & $\hat{\theta}_{1,n}(\lambda)$ & $\hat{\theta}_{2,n}(\lambda)$ & $\hat{\theta}_{3,n}(\lambda)$ \\ \hline
& & \multicolumn{3}{c|}{} & \multicolumn{3}{c|}{} & \multicolumn{3}{c}{} \\[-3mm]
$n=1000$ & mean & -1.9917 & 3.0076 & -0.0004 & -2.1008 & 3.1735 & -0.0244 & -2.0739 & 3.1357 & -0.0272 \\[1mm]
& s.d. & 0.0656 & 0.0637 & 0.0612 & 0.1007 & 0.1092 & 0.1044 & 0.0913 & 0.0989 & 0.0951 \\[1mm]
$n=5000$ & mean & -1.9974 & 3.0006 & 0.0015 & -2.0617 & 3.1018 & -0.0195 & -2.0360 & 3.0620 & -0.0121 \\[1mm]
& s.d. & 0.0286 & 0.0283 & 0.0288 & 0.0396 & 0.0406 & 0.0403 & 0.0361 & 0.0362 & 0.0362 \\[1mm] 
& & \multicolumn{3}{r|}{(time: 0.3208)} & \multicolumn{3}{r|}{(time: 0.3128)} & \multicolumn{3}{r}{(time: 0.3123)} \\[1mm] \hline
& & \multicolumn{3}{c|}{} & \multicolumn{3}{c|}{} & \multicolumn{3}{c}{} \\[-3mm]
& $\lambda=0.5$ & $\hat{\theta}_{1,n}(\lambda)$ & $\hat{\theta}_{2,n}(\lambda)$ & $\hat{\theta}_{3,n}(\lambda)$ & $\hat{\theta}_{1,n}(\lambda)$ & $\hat{\theta}_{2,n}(\lambda)$ & $\hat{\theta}_{3,n}(\lambda)$ & $\hat{\theta}_{1,n}(\lambda)$ & $\hat{\theta}_{2,n}(\lambda)$ & $\hat{\theta}_{3,n}(\lambda)$ \\ \hline
& & \multicolumn{3}{c|}{} & \multicolumn{3}{c|}{} & \multicolumn{3}{c}{} \\[-3mm]
$n=1000$ & mean & -1.9916 & 3.0082 & 0.0001 & -2.0025 & 3.0254 & -0.0031 & -1.9978 & 3.0183 & -0.0017 \\[1mm]
& s.d. & 0.0797 & 0.0794 & 0.0763 & 0.0848 & 0.0876 & 0.0828 & 0.0838 & 0.0868 & 0.0823 \\[1mm]
$n=5000$ & mean & -1.9967 & 2.9999 & 0.0015 & -2.0014 & 3.0091 & -0.0015 & -1.9994 & 3.0055 & -0.0005 \\[1mm]
& s.d. & 0.0351 & 0.0340 & 0.0346 & 0.0374 & 0.0363 & 0.0369 & 0.0370 & 0.0362 & 0.0366 \\[1mm] 
& & \multicolumn{3}{r|}{(time: 0.3564)} & \multicolumn{3}{r|}{(time: 0.3477)} & \multicolumn{3}{r}{(time: 0.3516)} \\[1mm] \hline
& & \multicolumn{3}{c|}{} & \multicolumn{3}{c|}{} & \multicolumn{3}{c}{} \\[-3mm]
& $\lambda=0.9$ & $\hat{\theta}_{1,n}(\lambda)$ & $\hat{\theta}_{2,n}(\lambda)$ & $\hat{\theta}_{3,n}(\lambda)$ & $\hat{\theta}_{1,n}(\lambda)$ & $\hat{\theta}_{2,n}(\lambda)$ & $\hat{\theta}_{3,n}(\lambda)$ & $\hat{\theta}_{1,n}(\lambda)$ & $\hat{\theta}_{2,n}(\lambda)$ & $\hat{\theta}_{3,n}(\lambda)$ \\ \hline
& & \multicolumn{3}{c|}{} & \multicolumn{3}{c|}{} & \multicolumn{3}{c}{} \\[-3mm]
$n=1000$ & mean & -1.9908 & 3.0093 & 0.0003 & -1.9950 & 3.0151 & 0.0002 & -1.9928 & 3.0121 & 0.0009 \\[1mm]
& s.d. & 0.0997 & 0.1021 & 0.0989 & 0.1059 & 0.1109 & 0.1054 & 0.1051 & 0.1109 & 0.1053 \\[1mm]
$n=5000$ & mean & -1.9963 & 2.9994 & 0.0019 & -1.9977 & 3.0034 & 0.0001 & -1.9971 & 3.0018 & 0.0007 \\[1mm]
& s.d. & 0.0446 & 0.0430 & 0.0433 & 0.0470 & 0.0455 & 0.0459 & 0.0468 & 0.0455 & 0.0457 \\[1mm] 
& & \multicolumn{3}{r|}{(time: 0.3634)} & \multicolumn{3}{r|}{(time: 0.3521)} & \multicolumn{3}{r}{(time: 0.3516)} \\[1mm] \hline
\end{tabular}
}
}
\label{esti14}
\end{center}
\end{table}


\begin{figure}[t]
\begin{tabular}{c}
\includegraphics[scale=0.09]{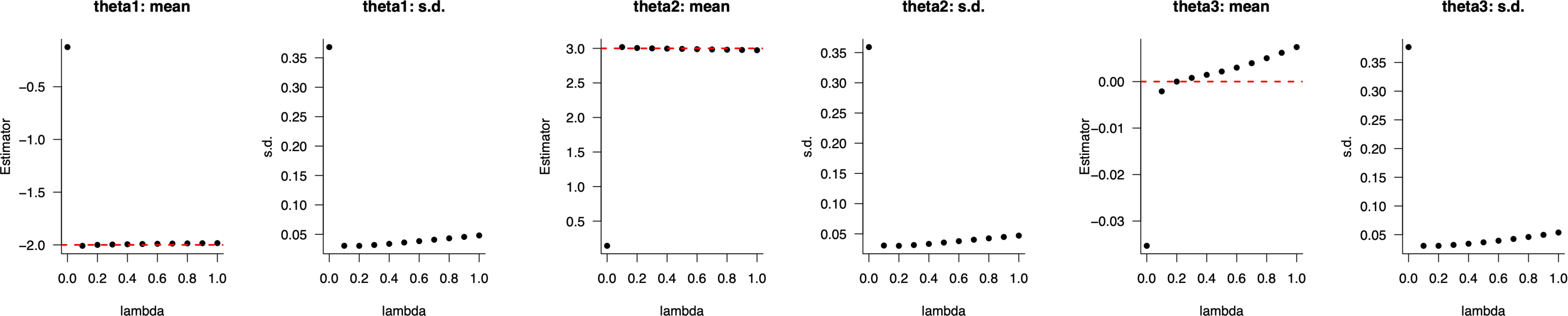}
\end{tabular}
\caption{Mean and standard deviation of the density-power estimator in each $\lambda$ in Section \ref{se:simu12} ($\sigma^{2}=1,\mathfrak{p}=0.01,n=5000$).}
\label{estplot13}
\end{figure}

\begin{figure}[t]
\begin{tabular}{c}
\includegraphics[scale=0.09]{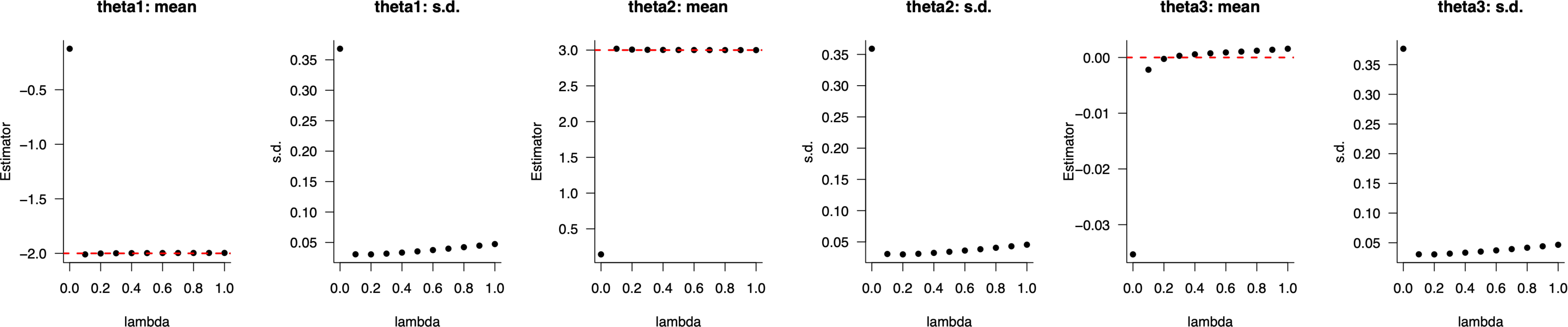}
\end{tabular}
\caption{Mean and standard deviation of the H\"older-based estimator in each $\lambda$ in Section \ref{se:simu12} ($\sigma^{2}=1,\mathfrak{p}=0.01,n=5000$).}
\label{holestplot13}
\end{figure}

\begin{figure}[t]
\begin{tabular}{c}

\begin{minipage}{0.32 \hsize}
\begin{center}
\includegraphics[scale=0.25]{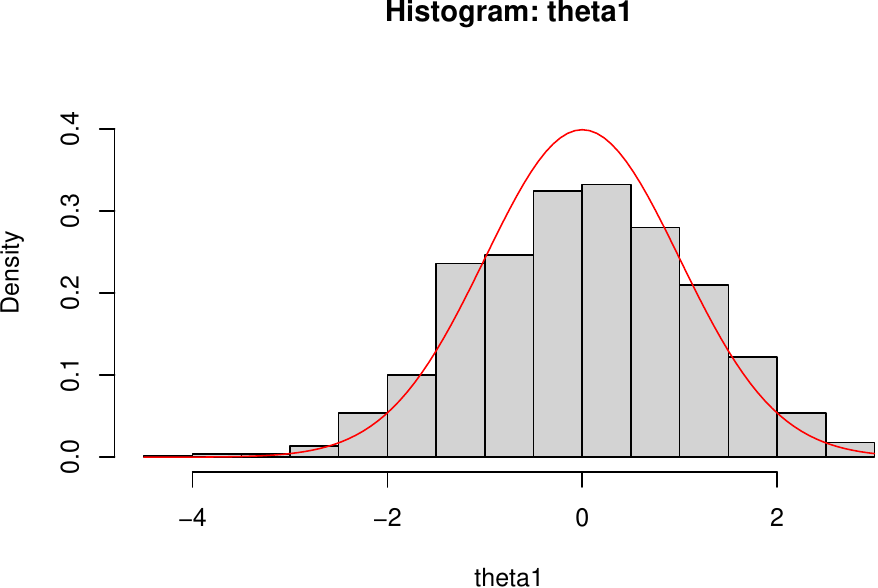}
\end{center}
\end{minipage}

\begin{minipage}{0.32 \hsize}
\begin{center}
\includegraphics[scale=0.25]{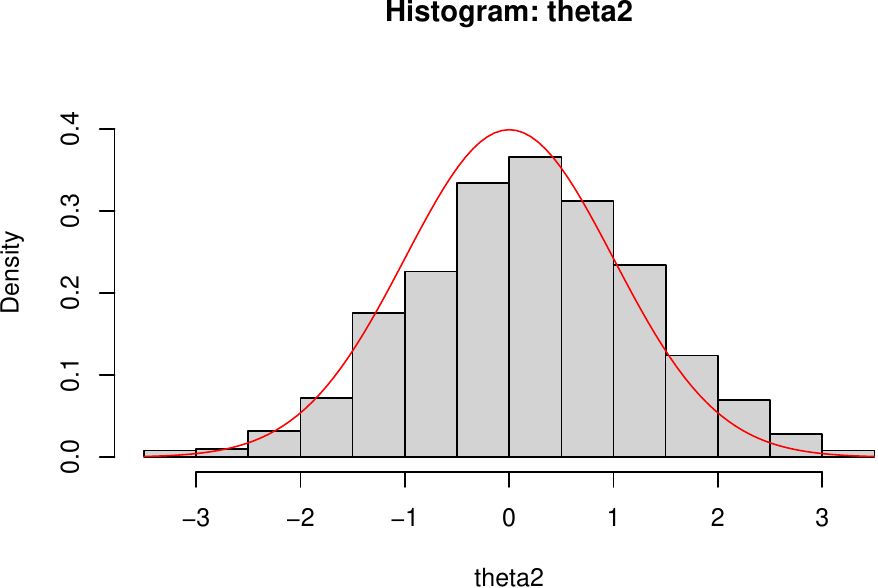}
\end{center}
\end{minipage}

\begin{minipage}{0.32 \hsize}
\begin{center}
\includegraphics[scale=0.25]{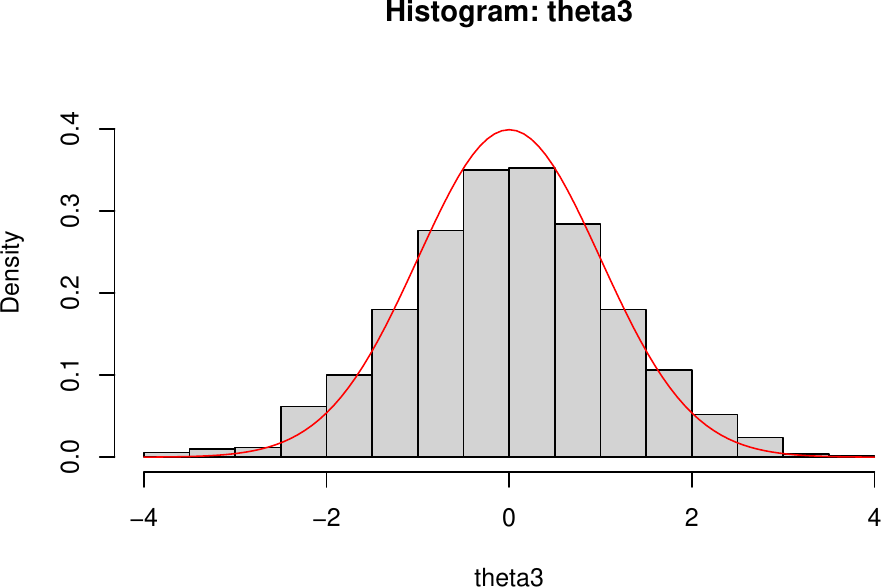}
\end{center}
\end{minipage}

\end{tabular}
\caption{Histograms of $u_{1,n}(\lambda)$, $u_{2,n}(\lambda)$, and $u_{3,n}(\lambda)$ corresponding to the density-power estimator in Section \ref{se:simu12} ($\sigma^{2}=1,\mathfrak{p}=0.01,n=5000, \lambda=0.2$).}
\label{plot12}
\end{figure}


\begin{figure}[t]
\begin{tabular}{c}

\begin{minipage}{0.32 \hsize}
\begin{center}
\includegraphics[scale=0.25]{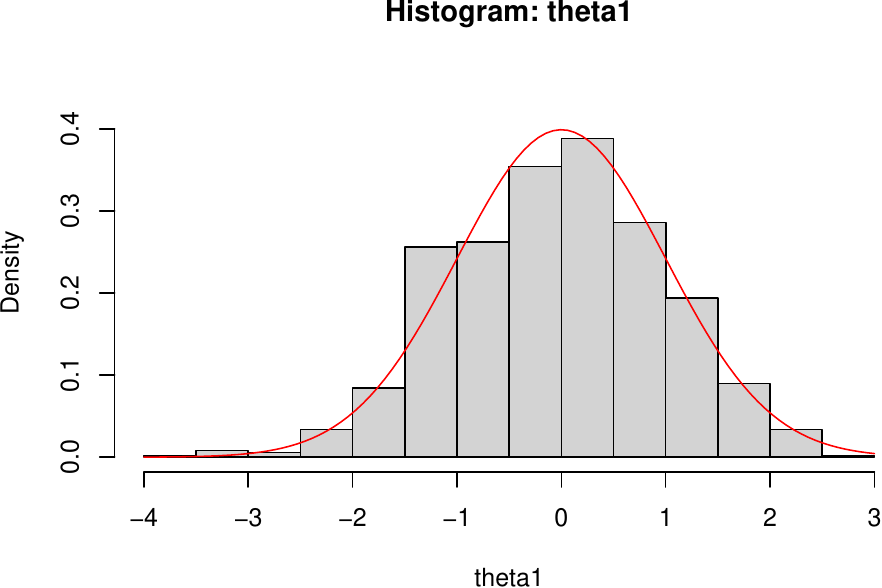}
\end{center}
\end{minipage}

\begin{minipage}{0.32 \hsize}
\begin{center}
\includegraphics[scale=0.25]{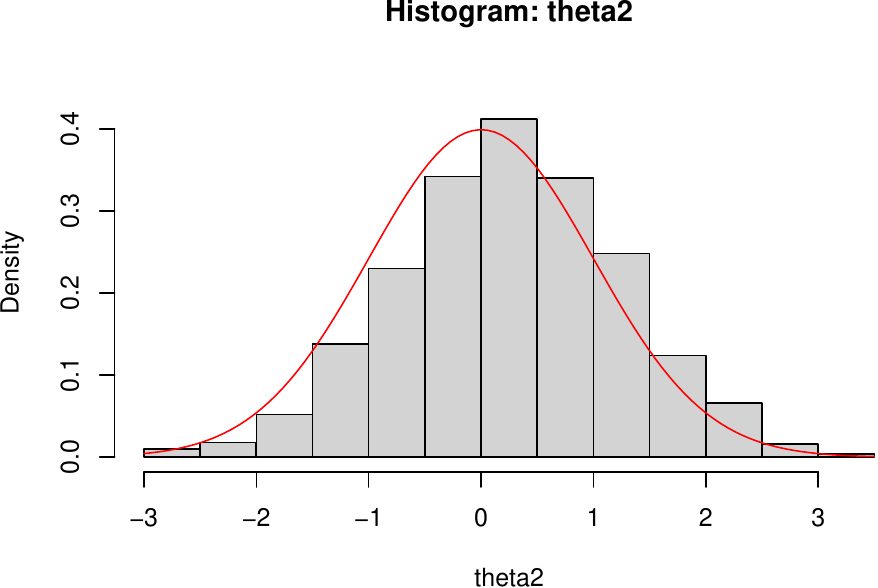}
\end{center}
\end{minipage}

\begin{minipage}{0.32 \hsize}
\begin{center}
\includegraphics[scale=0.25]{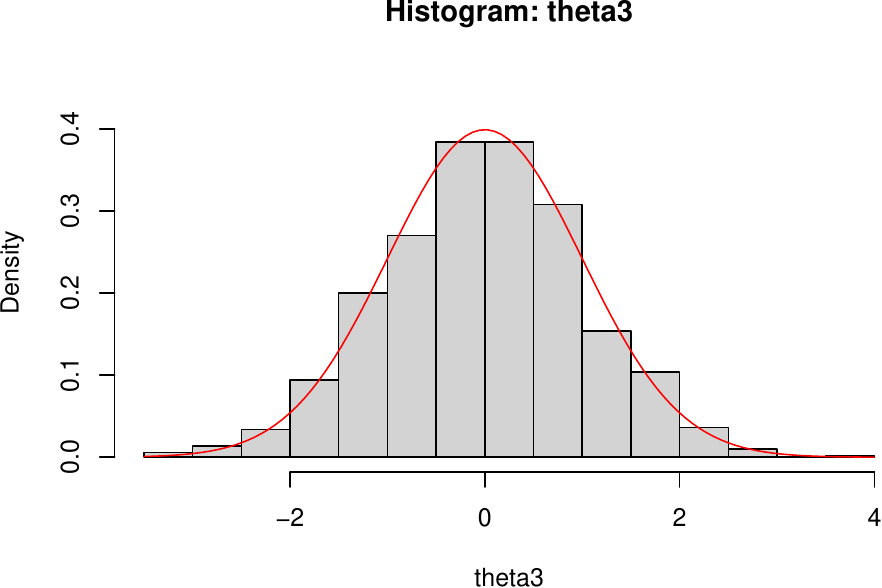}
\end{center}
\end{minipage}

\end{tabular}
\caption{Histograms of $u_{1,n}(\lambda)$, $u_{2,n}(\lambda)$, and $u_{3,n}(\lambda)$ corresponding to the H\"older-based estimator in Section \ref{se:simu12} ($\sigma^{2}=1,\mathfrak{p}=0.01,n=5000, \lambda=0.2$).}
\label{holplot12}
\end{figure}














\begin{table}[t]
\begin{center}
\caption{Frequency that the 95\% confidence interval corresponding to the density-power estimator contains the true value in each $\lambda$ in Section \ref{se:simu12} ($\sigma^{2}=1,\mathfrak{p}=0.01,n=5000$).}
\begin{tabular}{c | r r r r r r r r r r} \hline
& \multicolumn{9}{c}{} \\[-3mm]
$\lambda$ & 0.1 & 0.2 & 0.3 & 0.4 & 0.5 & 0.6 & 0.7 & 0.8 & 0.9 & 1 \\[1mm] \hline
& \multicolumn{9}{c}{} \\[-3mm]
$\theta_{1}$ & 0.88 & 0.92 & 0.93 & 0.93 & 0.93 & 0.93 & 0.93 & 0.93 & 0.93 & 0.94 \\[1mm]
$\theta_{2}$ & 0.83 & 0.92 & 0.94 & 0.94 & 0.94 & 0.94 & 0.93 & 0.93 & 0.92 & 0.92 \\[1mm]
$\theta_{3}$ & 0.89 & 0.91 & 0.91 & 0.92 & 0.93 & 0.93 & 0.94 & 0.94 & 0.94 & 0.94 \\[1mm] \hline
\end{tabular}
\label{ciratio12}
\end{center}
\end{table}


\begin{table}[t]
\begin{center}
\caption{Frequency that the 95\% confidence interval corresponding to the H\"older-based estimator contains the true value in each $\lambda$ in Section \ref{se:simu12} ($\sigma^{2}=1,\mathfrak{p}=0.01,n=5000$).}
\begin{tabular}{c | r r r r r r r r r r} \hline
& \multicolumn{9}{c}{} \\[-3mm]
$\lambda$ & 0.1 & 0.2 & 0.3 & 0.4 & 0.5 & 0.6 & 0.7 & 0.8 & 0.9 & 1 \\[1mm] \hline
& \multicolumn{9}{c}{} \\[-3mm]
$\theta_{1}$ & 0.92 & 0.95 & 0.95 & 0.94 & 0.95 & 0.95 & 0.95 & 0.95 & 0.95 & 0.94 \\[1mm]
$\theta_{2}$ & 0.88 & 0.94 & 0.95 & 0.95 & 0.95 & 0.96 & 0.96 & 0.95 & 0.95 & 0.95 \\[1mm]
$\theta_{3}$ & 0.94 & 0.94 & 0.94 & 0.94 & 0.94 & 0.94 & 0.94 & 0.94 & 0.94 & 0.94 \\[1mm] \hline
\end{tabular}
\label{holciratio12}
\end{center}
\end{table}

\subsubsection{With some jumps} \label{se:simu22}

We set $s=1$, $Y_{t_{j}}=\ly_{t_{j}}$, and $X_{t_{j}}=\lx_{t_{j}}$.
Moreover, we consider the following cases as jump-size distributions of the compound Poisson process $J$: (i) $\serev{\mathcal{N}}(0,3)$, (ii) $Gamma(1,1)$, where $Gamma$ denotes the Gamma distribution with both shape and scale parameters equal to $1$.
In these cases, simulations are performed for $q=0.01n$ and $0.05n$.
The plots in Figure \ref{pathplot3} show one of 1000 sample paths of cases (i) and (ii) for $q=0.01n$ and $n=5000$.

Tables \ref{esti33} and \ref{esti34} show the estimation results by using GQMLE and density-power GQMLE, and H\"{o}lder-based GQMLE in cases (i) and (ii), respectively.
We can observe that the performance of density-power and H\"{o}lder-based GQMLEs is better than that of GQMLE.
The calculation time of H\"{o}lder-based GQMLE is shorter than that of density-power GQMLE.

Figures \ref{estplot33} and \ref{estplot35} give the means and standard deviations of the density-power GQMLEs in each $\lambda$, and Figures \ref{holestplot33} and \ref{holestplot35} give the means and standard deviations of the H\"{o}lder-based GQMLEs in each $\lambda$.
From these figures, the density-power and H\"{o}lder-based GQMLEs have similar tendencies as in Section \ref{se:simu12} and seem to perform best when $\lambda=0.2$ in case (i) and when $\lambda=0.7$ in case (ii).

Figures \ref{plot32}, \ref{holplot32}, \serev{Tables \ref{ciratio31}, and \ref{holciratio31}} present the behaviors of the density-power and H\"{o}lder-based GQMLEs in case (i).
Figures \ref{plot32} and \ref{holplot32} give the histograms of $u_{1,n}(\lambda)$, $u_{2,n}(\lambda)$, and $u_{3,n}(\lambda)$ in the case of $\lambda=0.2$. 
Figure \ref{plot32} is based on the 1000th sample data and density-power GQMLE, and Figure \ref{holplot32} is based on the 1000th sample data and H\"{o}lder-based GQMLE.
Here, $u_{1,n}(\lambda)$, $u_{2,n}(\lambda)$, and $u_{3,n}(\lambda)$ are given by the corresponding to \eqref{se: asym.norm}.
\serev{Tables \ref{ciratio31} and \ref{holciratio31} summarize} the frequencies with which the true values of $\theta_{1}$, $\theta_{2}$, and $\theta_{3}$ are included in the 95\% confidence intervals in each $\lambda$ for density-power and H\"{o}lder-based GQMLEs, respectively.
The tendencies of these figures and tables are the same as those in Section \ref{se:simu12}.
In case (ii), we can also observe similar results as in case (i).
The simulation results for (ii) are provided in Appendix \ref{suppl_1}.


\begin{figure}[t]
\begin{tabular}{c}

\begin{minipage}{0.45 \hsize}
\begin{center}
\includegraphics[scale=0.35]{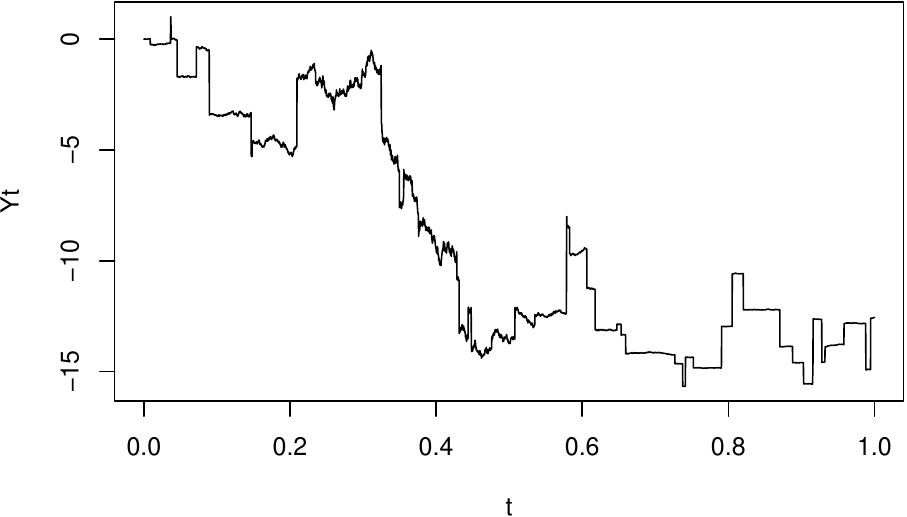}
\end{center}
\end{minipage}

\begin{minipage}{0.45 \hsize}
\begin{center}
\includegraphics[scale=0.35]{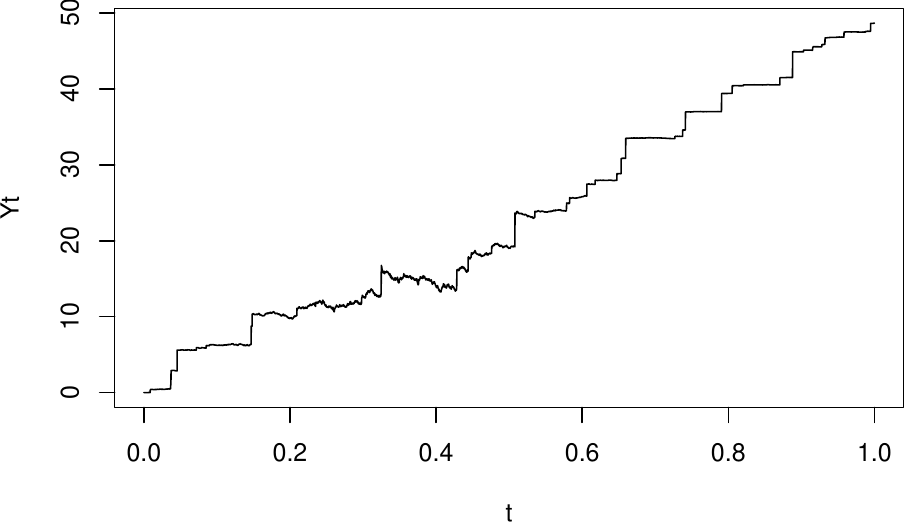}
\end{center}
\end{minipage}

\end{tabular}
\caption{One of 1000 sample paths in Section \ref{se:simu22} ($q=0.01n$, $n=5000$). The left one is the case of (i), and the right one is the case of (ii).}
\label{pathplot3}
\end{figure}


\begin{table}[t]
\begin{center}
\caption{GQMLE, density-power GQMLE, and H\"{o}lder-based GQMLE in Section \ref{se:simu22} (i) ($\theta_{0}=(-2,3,0)^{\top}$). ``time'' shows the mean calculation time.}
\scalebox{0.55}[0.55]{
\serev{
\begin{tabular}{r r | r r r | r r r | r r r | r r r} \hline
& & \multicolumn{3}{c}{} & \multicolumn{3}{c|}{} & \multicolumn{3}{c}{} & \multicolumn{3}{c}{} \\[-3mm]
& & \multicolumn{6}{l|}{$q=0.01n$} & \multicolumn{6}{l}{$q=0.05n$} \\ \cline{3-14}
& & \multicolumn{3}{c|}{} & \multicolumn{3}{c|}{} & \multicolumn{3}{c|}{} & \multicolumn{3}{c}{} \\[-3mm]
& & \multicolumn{3}{l|}{$n=1000$} & \multicolumn{3}{l|}{$n=5000$} & \multicolumn{3}{l|}{$n=1000$} & \multicolumn{3}{l}{$n=5000$} \\ \hline
& & \multicolumn{3}{c|}{} & \multicolumn{3}{c|}{} & \multicolumn{3}{c|}{} & \multicolumn{3}{c}{} \\[-3mm]
\multicolumn{1}{l}{GQMLE} & & $\hat{\theta}_{1,n}$ & $\hat{\theta}_{2,n}$ & $\hat{\theta}_{3,n}$ & $\hat{\theta}_{1,n}$ & $\hat{\theta}_{2,n}$ & $\hat{\theta}_{3,n}$ & $\hat{\theta}_{1,n}$ & $\hat{\theta}_{2,n}$ & $\hat{\theta}_{3,n}$ & $\hat{\theta}_{1,n}$ & $\hat{\theta}_{2,n}$ & $\hat{\theta}_{3,n}$ \\ \hline
& & \multicolumn{3}{c|}{} & \multicolumn{3}{c|}{} & \multicolumn{3}{c|}{} & \multicolumn{3}{c}{} \\[-3mm]
& mean & -0.5540 & 0.7562 & -0.1159 & -0.0496 & 0.0729 & -0.0087 & -0.0720 & 0.0956 & -0.0348 & -0.0078 & 0.0130 & -0.0139 \\[1mm] 
& s.d. & 1.0439 & 1.0756 & 1.0447 & 0.3639 & 0.3667 & 0.3830 & 0.3740 & 0.3766 & 0.3939 & 0.1605 & 0.1659 & 0.1621 \\[1mm] 
& & \multicolumn{3}{c|}{} & \multicolumn{3}{r|}{(time: 0.1515)} & \multicolumn{3}{c|}{} & \multicolumn{3}{r}{(time: 0.1871)} \\[1mm] \hline
& & \multicolumn{3}{c|}{} & \multicolumn{3}{c|}{} & \multicolumn{3}{c|}{} & \multicolumn{3}{c}{} \\[-3mm]
\multicolumn{1}{l}{Density-power:} & $\lambda=0.1$ & $\hat{\theta}_{1,n}(\lambda)$ & $\hat{\theta}_{2,n}(\lambda)$ & $\hat{\theta}_{3,n}(\lambda)$ & $\hat{\theta}_{1,n}(\lambda)$ & $\hat{\theta}_{2,n}(\lambda)$ & $\hat{\theta}_{3,n}(\lambda)$ & $\hat{\theta}_{1,n}(\lambda)$ & $\hat{\theta}_{2,n}(\lambda)$ & $\hat{\theta}_{3,n}(\lambda)$ & $\hat{\theta}_{1,n}(\lambda)$ & $\hat{\theta}_{2,n}(\lambda)$ & $\hat{\theta}_{3,n}(\lambda)$ \\ \hline
& & \multicolumn{3}{c|}{} & \multicolumn{3}{c|}{} & \multicolumn{3}{c|}{} & \multicolumn{3}{c}{} \\[-3mm]
& mean & -2.0009 & 3.0162 & 0.0027 & -2.0023 & 3.0062 & -0.0003 & -2.0322 & 3.0614 & -0.0100 & -2.0136 & 3.0281 & -0.0076 \\[1mm] 
& s.d. & 0.0664 & 0.0652 & 0.0656 & 0.0284 & 0.0291 & 0.0288 & 0.0727 & 0.0753 & 0.0727 & 0.0311 & 0.0306 & 0.0308 \\[1mm] 
& & \multicolumn{3}{c|}{} & \multicolumn{3}{r|}{(time: 0.4012)} & \multicolumn{3}{c|}{} & \multicolumn{3}{r}{(time: 0.5549)} \\[1mm]\hline
& & \multicolumn{3}{c|}{} & \multicolumn{3}{c|}{} & \multicolumn{3}{c|}{} & \multicolumn{3}{c}{} \\[-3mm]
& $\lambda=0.5$ & $\hat{\theta}_{1,n}(\lambda)$ & $\hat{\theta}_{2,n}(\lambda)$ & $\hat{\theta}_{3,n}(\lambda)$ & $\hat{\theta}_{1,n}(\lambda)$ & $\hat{\theta}_{2,n}(\lambda)$ & $\hat{\theta}_{3,n}(\lambda)$ & $\hat{\theta}_{1,n}(\lambda)$ & $\hat{\theta}_{2,n}(\lambda)$ & $\hat{\theta}_{3,n}(\lambda)$ & $\hat{\theta}_{1,n}(\lambda)$ & $\hat{\theta}_{2,n}(\lambda)$ & $\hat{\theta}_{3,n}(\lambda)$ \\ \hline
& & \multicolumn{3}{c|}{} & \multicolumn{3}{c|}{} & \multicolumn{3}{c|}{} & \multicolumn{3}{c}{} \\[-3mm]
& mean & -1.9925 & 3.0048 & 0.0051 & -1.9974 & 2.9982 & 0.0016 & -1.9863 & 2.9916 & 0.0019 & -1.9837 & 2.9805 & 0.0001 \\[1mm] 
& s.d. & 0.0822 & 0.0828 & 0.0808 & 0.0353 & 0.0376 & 0.0370 & 0.0818 & 0.0854 & 0.0801 & 0.0374 & 0.0364 & 0.0373 \\[1mm] 
& & \multicolumn{3}{c|}{} & \multicolumn{3}{r|}{(time: 0.4774)} & \multicolumn{3}{c|}{} & \multicolumn{3}{r}{(time: 0.6503)} \\[1mm]\hline
& & \multicolumn{3}{c|}{} & \multicolumn{3}{c|}{} & \multicolumn{3}{c|}{} & \multicolumn{3}{c}{} \\[-3mm]
& $\lambda=0.9$ & $\hat{\theta}_{1,n}(\lambda)$ & $\hat{\theta}_{2,n}(\lambda)$ & $\hat{\theta}_{3,n}(\lambda)$ & $\hat{\theta}_{1,n}(\lambda)$ & $\hat{\theta}_{2,n}(\lambda)$ & $\hat{\theta}_{3,n}(\lambda)$ & $\hat{\theta}_{1,n}(\lambda)$ & $\hat{\theta}_{2,n}(\lambda)$ & $\hat{\theta}_{3,n}(\lambda)$ & $\hat{\theta}_{1,n}(\lambda)$ & $\hat{\theta}_{2,n}(\lambda)$ & $\hat{\theta}_{3,n}(\lambda)$ \\ \hline
& & \multicolumn{3}{c|}{} & \multicolumn{3}{c|}{} & \multicolumn{3}{c|}{} & \multicolumn{3}{c}{} \\[-3mm]
& mean & -1.9896 & 2.9989 & 0.0082 & -1.9946 & 2.9917 & 0.0035 & -1.9704 & 2.9586 & 0.0098 & -1.9660 & 2.9446 & 0.0076 \\[1mm] 
& s.d. & 0.1031 & 0.1045 & 0.1104 & 0.0443 & 0.0478 & 0.0508 & 0.1010 & 0.1080 & 0.1101 & 0.0468 & 0.0460 & 0.0504 \\[1mm] 
& & \multicolumn{3}{c|}{} & \multicolumn{3}{r|}{(time: 0.5076)} & \multicolumn{3}{c|}{} & \multicolumn{3}{r}{(time: 0.6716)} \\[1mm] \hline
& & \multicolumn{3}{c|}{} & \multicolumn{3}{c|}{} & \multicolumn{3}{c|}{} & \multicolumn{3}{c}{} \\[-3mm]
\multicolumn{1}{l}{H\"{o}lder-based:} & $\lambda=0.1$ & $\hat{\theta}_{1,n}(\lambda)$ & $\hat{\theta}_{2,n}(\lambda)$ & $\hat{\theta}_{3,n}(\lambda)$ & $\hat{\theta}_{1,n}(\lambda)$ & $\hat{\theta}_{2,n}(\lambda)$ & $\hat{\theta}_{3,n}(\lambda)$ & $\hat{\theta}_{1,n}(\lambda)$ & $\hat{\theta}_{2,n}(\lambda)$ & $\hat{\theta}_{3,n}(\lambda)$ & $\hat{\theta}_{1,n}(\lambda)$ & $\hat{\theta}_{2,n}(\lambda)$ & $\hat{\theta}_{3,n}(\lambda)$ \\ \hline
& & \multicolumn{3}{c|}{} & \multicolumn{3}{c|}{} & \multicolumn{3}{c|}{} & \multicolumn{3}{c}{} \\[-3mm]
& mean & -2.0012 & 3.0167 & 0.0026 & -2.0025 & 3.0065 & -0.0004 & -2.0341 & 3.0641 & -0.0103 & -2.0148 & 3.0299 & -0.0078 \\[1mm] 
& s.d. & 0.0664 & 0.0651 & 0.0655 & 0.0284 & 0.0291 & 0.0288 & 0.0728 & 0.0754 & 0.0727 & 0.0311 & 0.0306 & 0.0308 \\[1mm]
& & \multicolumn{3}{c|}{} & \multicolumn{3}{r|}{(time: 0.2391)} & \multicolumn{3}{c|}{} & \multicolumn{3}{r}{(time: 0.3141)} \\[1mm] \hline
& & \multicolumn{3}{c|}{} & \multicolumn{3}{c|}{} & \multicolumn{3}{c|}{} & \multicolumn{3}{c}{} \\[-3mm]
& $\lambda=0.5$ & $\hat{\theta}_{1,n}(\lambda)$ & $\hat{\theta}_{2,n}(\lambda)$ & $\hat{\theta}_{3,n}(\lambda)$ & $\hat{\theta}_{1,n}(\lambda)$ & $\hat{\theta}_{2,n}(\lambda)$ & $\hat{\theta}_{3,n}(\lambda)$ & $\hat{\theta}_{1,n}(\lambda)$ & $\hat{\theta}_{2,n}(\lambda)$ & $\hat{\theta}_{3,n}(\lambda)$ & $\hat{\theta}_{1,n}(\lambda)$ & $\hat{\theta}_{2,n}(\lambda)$ & $\hat{\theta}_{3,n}(\lambda)$ \\ \hline
& & \multicolumn{3}{c|}{} & \multicolumn{3}{c|}{} & \multicolumn{3}{c|}{} & \multicolumn{3}{c}{} \\[-3mm]
& mean & -1.9951 & 3.0092 & 0.0040 & -2.0001 & 3.0024 & 0.0012 & -2.0009 & 3.0157 & -0.0005 & -1.9983 & 3.0039 & -0.0023 \\[1mm] 
& s.d. & 0.0804 & 0.0793 & 0.0780 & 0.0345 & 0.0363 & 0.0356 & 0.0802 & 0.0823 & 0.0772 & 0.0365 & 0.0354 & 0.0363 \\[1mm]
& & \multicolumn{3}{c|}{} & \multicolumn{3}{r|}{(time: 0.2658)} & \multicolumn{3}{c|}{} & \multicolumn{3}{r}{(time: 0.3573)} \\[1mm] \hline
& & \multicolumn{3}{c|}{} & \multicolumn{3}{c|}{} & \multicolumn{3}{c|}{} & \multicolumn{3}{c}{} \\[-3mm]
& $\lambda=0.9$ & $\hat{\theta}_{1,n}(\lambda)$ & $\hat{\theta}_{2,n}(\lambda)$ & $\hat{\theta}_{3,n}(\lambda)$ & $\hat{\theta}_{1,n}(\lambda)$ & $\hat{\theta}_{2,n}(\lambda)$ & $\hat{\theta}_{3,n}(\lambda)$ & $\hat{\theta}_{1,n}(\lambda)$ & $\hat{\theta}_{2,n}(\lambda)$ & $\hat{\theta}_{3,n}(\lambda)$ & $\hat{\theta}_{1,n}(\lambda)$ & $\hat{\theta}_{2,n}(\lambda)$ & $\hat{\theta}_{3,n}(\lambda)$ \\ \hline
& & \multicolumn{3}{c|}{} & \multicolumn{3}{c|}{} & \multicolumn{3}{c|}{} & \multicolumn{3}{c}{} \\[-3mm]
& mean & -1.9947 & 3.0110 & 0.0043 & -2.0009 & 3.0028 & 0.0017 & -2.0008 & 3.0174 & 0.0004 & -1.9974 & 3.0030 & -0.0022 \\[1mm] 
& s.d. & 0.1013 & 0.1003 & 0.0993 & 0.0434 & 0.0465 & 0.0455 & 0.1000 & 0.1046 & 0.0989 & 0.0460 & 0.0454 & 0.0460 \\[1mm]
& & \multicolumn{3}{c|}{} & \multicolumn{3}{r|}{(time: 0.2690)} & \multicolumn{3}{c|}{} & \multicolumn{3}{r}{(time: 0.3700)} \\[1mm] \hline
\end{tabular}
}
}
\label{esti33}
\end{center}
\end{table}


\begin{table}[t]
\begin{center}
\caption{GQMLE, density-power GQMLE, and H\"{o}lder-based GQMLE in Section \ref{se:simu22} (ii) ($\theta_{0}=(-2,3,0)^{\top}$). ``time'' shows the mean calculation time.}
\scalebox{0.55}[0.55]{
\serev{
\begin{tabular}{r r | r r r | r r r | r r r | r r r} \hline
& & \multicolumn{3}{c}{} & \multicolumn{3}{c|}{} & \multicolumn{3}{c}{} & \multicolumn{3}{c}{} \\[-3mm]
& & \multicolumn{6}{l|}{$q=0.01n$} & \multicolumn{6}{l}{$q=0.05n$} \\ \cline{3-14}
& & \multicolumn{3}{c|}{} & \multicolumn{3}{c|}{} & \multicolumn{3}{c|}{} & \multicolumn{3}{c}{} \\[-3mm]
& & \multicolumn{3}{l|}{$n=1000$} & \multicolumn{3}{l|}{$n=5000$} & \multicolumn{3}{l|}{$n=1000$} & \multicolumn{3}{l}{$n=5000$} \\ \hline
& & \multicolumn{3}{c|}{} & \multicolumn{3}{c|}{} & \multicolumn{3}{c|}{} & \multicolumn{3}{c}{} \\[-3mm]
\multicolumn{1}{l}{GQMLE} & & $\hat{\theta}_{1,n}$ & $\hat{\theta}_{2,n}$ & $\hat{\theta}_{3,n}$ & $\hat{\theta}_{1,n}$ & $\hat{\theta}_{2,n}$ & $\hat{\theta}_{3,n}$ & $\hat{\theta}_{1,n}$ & $\hat{\theta}_{2,n}$ & $\hat{\theta}_{3,n}$ & $\hat{\theta}_{1,n}$ & $\hat{\theta}_{2,n}$ & $\hat{\theta}_{3,n}$ \\ \hline
& & \multicolumn{3}{c|}{} & \multicolumn{3}{c|}{} & \multicolumn{3}{c|}{} & \multicolumn{3}{c}{} \\[-3mm]
& mean & -0.6687 & 1.0463 & -0.0841 & -0.1022 & 0.1396 & -0.0541 & -0.1008 & 0.1627 & -0.0537 & -0.0103 & 0.0376 & 0.0003 \\[1mm] 
& s.d. & 1.0368 & 1.1050 & 0.9986 & 0.4938 & 0.4685 & 0.4762 & 0.4917 & 0.4579 & 0.4895 & 0.2280 & 0.2141 & 0.2236 \\[1mm] 
& & \multicolumn{3}{c|}{} & \multicolumn{3}{r|}{(time: 0.1561)} & \multicolumn{3}{c|}{} & \multicolumn{3}{r}{(time: 0.1432)} \\[1mm] \hline
& & \multicolumn{3}{c|}{} & \multicolumn{3}{c|}{} & \multicolumn{3}{c|}{} & \multicolumn{3}{c}{} \\[-3mm]
\multicolumn{1}{l}{Density-power:} & $\lambda=0.1$ & $\hat{\theta}_{1,n}(\lambda)$ & $\hat{\theta}_{2,n}(\lambda)$ & $\hat{\theta}_{3,n}(\lambda)$ & $\hat{\theta}_{1,n}(\lambda)$ & $\hat{\theta}_{2,n}(\lambda)$ & $\hat{\theta}_{3,n}(\lambda)$ & $\hat{\theta}_{1,n}(\lambda)$ & $\hat{\theta}_{2,n}(\lambda)$ & $\hat{\theta}_{3,n}(\lambda)$ & $\hat{\theta}_{1,n}(\lambda)$ & $\hat{\theta}_{2,n}(\lambda)$ & $\hat{\theta}_{3,n}(\lambda)$ \\ \hline
& & \multicolumn{3}{c|}{} & \multicolumn{3}{c|}{} & \multicolumn{3}{c|}{} & \multicolumn{3}{c}{} \\[-3mm]
& mean & -2.0004 & 3.0158 & -0.0015 & -2.0046 & 3.0091 & -0.0017 & -2.0336 & 3.0631 & -0.0036 & -2.0263 & 3.0395 & -0.0079 \\[1mm] 
& s.d. & 0.0646 & 0.0671 & 0.0658 & 0.0292 & 0.0289 & 0.0289 & 0.0736 & 0.0754 & 0.0801 & 0.0310 & 0.0304 & 0.0315 \\[1mm]
& & \multicolumn{3}{c|}{} & \multicolumn{3}{r|}{(time: 0.4029)} & \multicolumn{3}{c|}{} & \multicolumn{3}{r}{(time: 0.4025)} \\[1mm]\hline
& & \multicolumn{3}{c|}{} & \multicolumn{3}{c|}{} & \multicolumn{3}{c|}{} & \multicolumn{3}{c}{} \\[-3mm]
& $\lambda=0.5$ & $\hat{\theta}_{1,n}(\lambda)$ & $\hat{\theta}_{2,n}(\lambda)$ & $\hat{\theta}_{3,n}(\lambda)$ & $\hat{\theta}_{1,n}(\lambda)$ & $\hat{\theta}_{2,n}(\lambda)$ & $\hat{\theta}_{3,n}(\lambda)$ & $\hat{\theta}_{1,n}(\lambda)$ & $\hat{\theta}_{2,n}(\lambda)$ & $\hat{\theta}_{3,n}(\lambda)$ & $\hat{\theta}_{1,n}(\lambda)$ & $\hat{\theta}_{2,n}(\lambda)$ & $\hat{\theta}_{3,n}(\lambda)$ \\ \hline
& & \multicolumn{3}{c|}{} & \multicolumn{3}{c|}{} & \multicolumn{3}{c|}{} & \multicolumn{3}{c}{} \\[-3mm]
& mean & -1.9924 & 3.0035 & -0.0008 & -1.9963 & 2.9973 & -0.0007 & -1.9832 & 2.9873 & 0.0028 & -1.9871 & 2.9793 & 0.0008 \\[1mm] 
& s.d. & 0.0802 & 0.0823 & 0.0814 & 0.0363 & 0.0369 & 0.0356 & 0.0824 & 0.0855 & 0.0871 & 0.0361 & 0.0365 & 0.0372 \\[1mm]
& & \multicolumn{3}{c|}{} & \multicolumn{3}{r|}{(time: 0.4819)} & \multicolumn{3}{c|}{} & \multicolumn{3}{r}{(time: 0.4813)} \\[1mm]\hline
& & \multicolumn{3}{c|}{} & \multicolumn{3}{c|}{} & \multicolumn{3}{c|}{} & \multicolumn{3}{c}{} \\[-3mm]
& $\lambda=0.9$ & $\hat{\theta}_{1,n}(\lambda)$ & $\hat{\theta}_{2,n}(\lambda)$ & $\hat{\theta}_{3,n}(\lambda)$ & $\hat{\theta}_{1,n}(\lambda)$ & $\hat{\theta}_{2,n}(\lambda)$ & $\hat{\theta}_{3,n}(\lambda)$ & $\hat{\theta}_{1,n}(\lambda)$ & $\hat{\theta}_{2,n}(\lambda)$ & $\hat{\theta}_{3,n}(\lambda)$ & $\hat{\theta}_{1,n}(\lambda)$ & $\hat{\theta}_{2,n}(\lambda)$ & $\hat{\theta}_{3,n}(\lambda)$ \\ \hline
& & \multicolumn{3}{c|}{} & \multicolumn{3}{c|}{} & \multicolumn{3}{c|}{} & \multicolumn{3}{c}{} \\[-3mm]
& mean & -1.9898 & 2.9986 & 0.0004 & -1.9921 & 2.9899 & 0.0008 & -1.9640 & 2.9523 & 0.0087 & -1.9687 & 2.9433 & 0.0081 \\[1mm] 
& s.d. & 0.1014 & 0.1052 & 0.1112 & 0.0456 & 0.0469 & 0.0490 & 0.1046 & 0.1073 & 0.1147 & 0.0443 & 0.0459 & 0.0505 \\[1mm] 
& & \multicolumn{3}{c|}{} & \multicolumn{3}{r|}{(time: 0.5132)} & \multicolumn{3}{c|}{} & \multicolumn{3}{r}{(time: 0.5116)} \\[1mm] \hline
& & \multicolumn{3}{c|}{} & \multicolumn{3}{c|}{} & \multicolumn{3}{c|}{} & \multicolumn{3}{c}{} \\[-3mm]
\multicolumn{1}{l}{H\"{o}lder-based:} & $\lambda=0.1$ & $\hat{\theta}_{1,n}(\lambda)$ & $\hat{\theta}_{2,n}(\lambda)$ & $\hat{\theta}_{3,n}(\lambda)$ & $\hat{\theta}_{1,n}(\lambda)$ & $\hat{\theta}_{2,n}(\lambda)$ & $\hat{\theta}_{3,n}(\lambda)$ & $\hat{\theta}_{1,n}(\lambda)$ & $\hat{\theta}_{2,n}(\lambda)$ & $\hat{\theta}_{3,n}(\lambda)$ & $\hat{\theta}_{1,n}(\lambda)$ & $\hat{\theta}_{2,n}(\lambda)$ & $\hat{\theta}_{3,n}(\lambda)$ \\ \hline
& & \multicolumn{3}{c|}{} & \multicolumn{3}{c|}{} & \multicolumn{3}{c|}{} & \multicolumn{3}{c}{} \\[-3mm]
& mean & -2.0007 & 3.0163 & -0.0016 & -2.0049 & 3.0095 & -0.0017 & -2.0358 & 3.0664 & -0.0039 & -2.0279 & 3.0419 & -0.0081 \\[1mm] 
& s.d. & 0.0646 & 0.0670 & 0.0658 & 0.0292 & 0.0289 & 0.0289 & 0.0737 & 0.0754 & 0.0802 & 0.0310 & 0.0304 & 0.0315 \\[1mm] 
& & \multicolumn{3}{c|}{} & \multicolumn{3}{r|}{(time: 0.2393)} & \multicolumn{3}{c|}{} & \multicolumn{3}{r}{(time: 0.2393)} \\[1mm] \hline
& & \multicolumn{3}{c|}{} & \multicolumn{3}{c|}{} & \multicolumn{3}{c|}{} & \multicolumn{3}{c}{} \\[-3mm]
& $\lambda=0.5$ & $\hat{\theta}_{1,n}(\lambda)$ & $\hat{\theta}_{2,n}(\lambda)$ & $\hat{\theta}_{3,n}(\lambda)$ & $\hat{\theta}_{1,n}(\lambda)$ & $\hat{\theta}_{2,n}(\lambda)$ & $\hat{\theta}_{3,n}(\lambda)$ & $\hat{\theta}_{1,n}(\lambda)$ & $\hat{\theta}_{2,n}(\lambda)$ & $\hat{\theta}_{3,n}(\lambda)$ & $\hat{\theta}_{1,n}(\lambda)$ & $\hat{\theta}_{2,n}(\lambda)$ & $\hat{\theta}_{3,n}(\lambda)$ \\ \hline
& & \multicolumn{3}{c|}{} & \multicolumn{3}{c|}{} & \multicolumn{3}{c|}{} & \multicolumn{3}{c}{} \\[-3mm]
& mean & -1.9953 & 3.0075 & -0.0014 & -1.9995 & 3.0024 & -0.0011 & -1.9995 & 3.0133 & 0.0001 & -2.0023 & 3.0038 & -0.0020 \\[1mm] 
& s.d. & 0.0784 & 0.0798 & 0.0787 & 0.0354 & 0.0355 & 0.0346 & 0.0810 & 0.0831 & 0.0849 & 0.0354 & 0.0358 & 0.0363 \\[1mm] 
& & \multicolumn{3}{c|}{} & \multicolumn{3}{r|}{(time: 0.2665)} & \multicolumn{3}{c|}{} & \multicolumn{3}{r}{(time: 0.2661)} \\[1mm] \hline
& & \multicolumn{3}{c|}{} & \multicolumn{3}{c|}{} & \multicolumn{3}{c|}{} & \multicolumn{3}{c}{} \\[-3mm]
& $\lambda=0.9$ & $\hat{\theta}_{1,n}(\lambda)$ & $\hat{\theta}_{2,n}(\lambda)$ & $\hat{\theta}_{3,n}(\lambda)$ & $\hat{\theta}_{1,n}(\lambda)$ & $\hat{\theta}_{2,n}(\lambda)$ & $\hat{\theta}_{3,n}(\lambda)$ & $\hat{\theta}_{1,n}(\lambda)$ & $\hat{\theta}_{2,n}(\lambda)$ & $\hat{\theta}_{3,n}(\lambda)$ & $\hat{\theta}_{1,n}(\lambda)$ & $\hat{\theta}_{2,n}(\lambda)$ & $\hat{\theta}_{3,n}(\lambda)$ \\ \hline
& & \multicolumn{3}{c|}{} & \multicolumn{3}{c|}{} & \multicolumn{3}{c|}{} & \multicolumn{3}{c}{} \\[-3mm]
& mean & -1.9967 & 3.0086 & -0.0021 & -1.9989 & 3.0021 & -0.0010 & -1.9970 & 3.0115 & 0.0001 & -2.0006 & 3.0020 & -0.0014 \\[1mm] 
& s.d. & 0.1001 & 0.1025 & 0.0993 & 0.0448 & 0.0450 & 0.0443 & 0.1034 & 0.1063 & 0.1055 & 0.0440 & 0.0460 & 0.0468 \\[1mm] 
& & \multicolumn{3}{c|}{} & \multicolumn{3}{r|}{(time: 0.2687)} & \multicolumn{3}{c|}{} & \multicolumn{3}{r}{(time: 0.2692)} \\[1mm] \hline
\end{tabular}
}
}
\label{esti34}
\end{center}
\end{table}


\begin{figure}[t]
\begin{tabular}{c}
\includegraphics[scale=0.09]{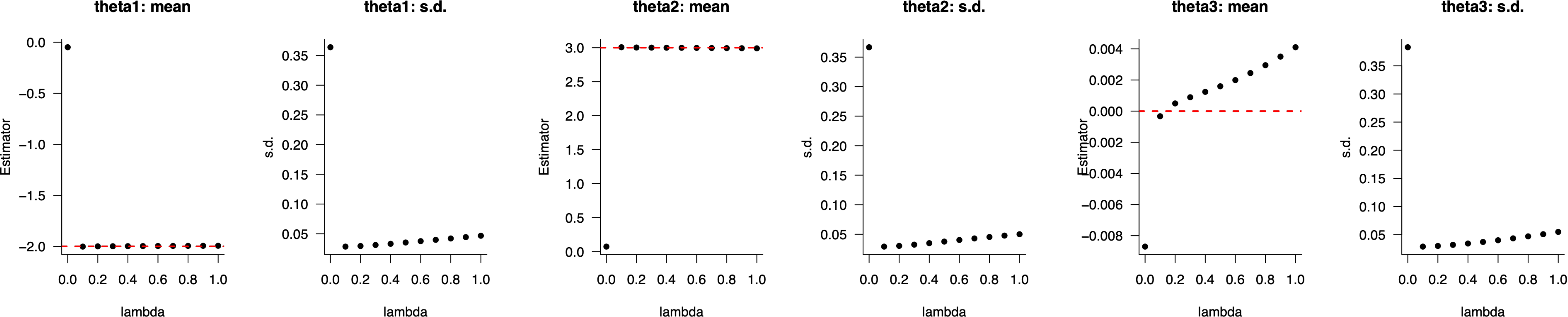}
\end{tabular}
\caption{Mean and standard deviation of the density-power estimator in each $\lambda$ in Section \ref{se:simu22} (i) ($q=0.01n, n=5000$).}
\label{estplot33}
\end{figure}

\begin{figure}[t]
\begin{tabular}{c}
\includegraphics[scale=0.09]{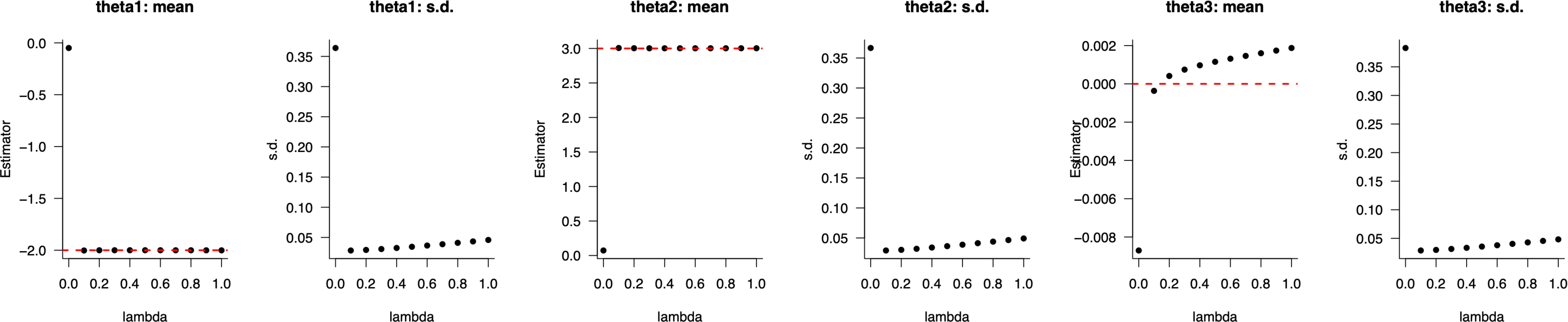}
\end{tabular}
\caption{Mean and standard deviation of the H\"older-based estimator in each $\lambda$ in Section \ref{se:simu22} (i) ($q=0.01n, n=5000$).}
\label{holestplot33}
\end{figure}


\begin{figure}[t]
\begin{tabular}{c}
\includegraphics[scale=0.09]{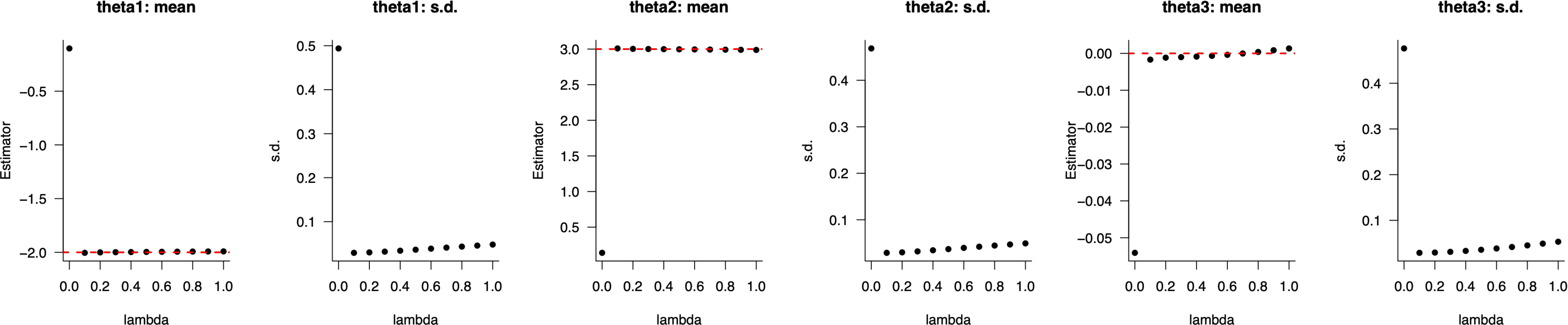}
\end{tabular}
\caption{Mean and standard deviation of the density-power estimator in each $\lambda$ in Section \ref{se:simu22} (ii) ($q=0.01n, n=5000$).}
\label{estplot35}
\end{figure}

\begin{figure}[t]
\begin{tabular}{c}
\includegraphics[scale=0.09]{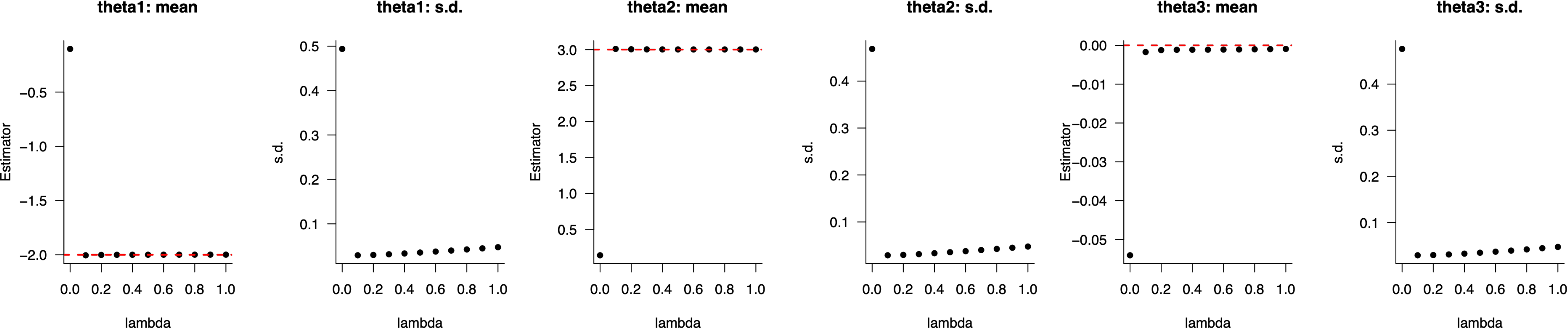}
\end{tabular}
\caption{Mean and standard deviation of the H\"older-based estimator in each $\lambda$ in Section \ref{se:simu22} (ii) ($q=0.01n, n=5000$).}
\label{holestplot35}
\end{figure}


\begin{figure}[t]
\begin{tabular}{c}

\begin{minipage}{0.32 \hsize}
\begin{center}
\includegraphics[scale=0.25]{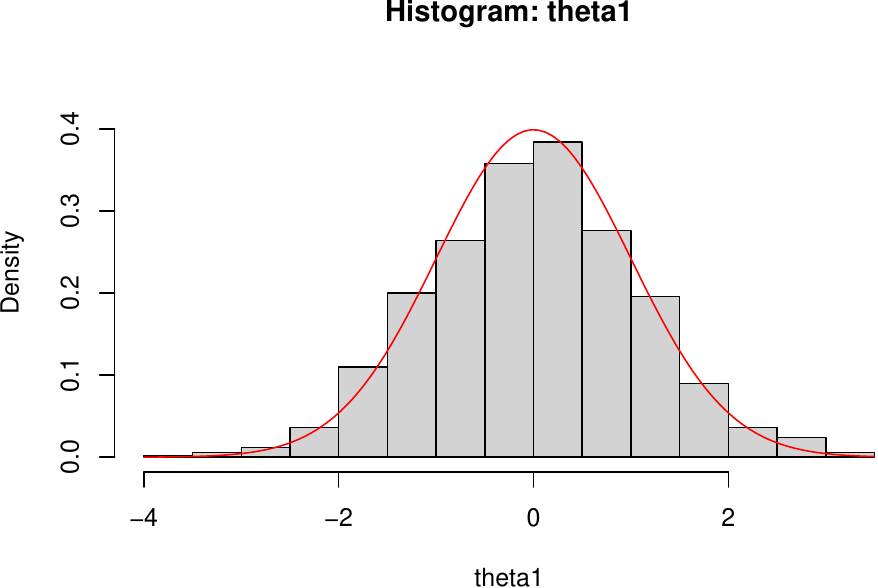}
\end{center}
\end{minipage}

\begin{minipage}{0.32 \hsize}
\begin{center}
\includegraphics[scale=0.25]{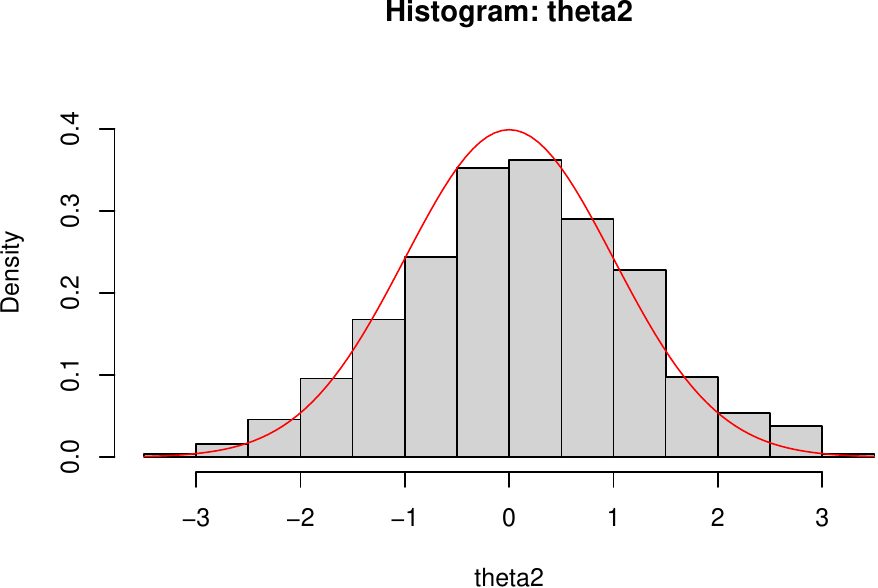}
\end{center}
\end{minipage}

\begin{minipage}{0.32 \hsize}
\begin{center}
\includegraphics[scale=0.25]{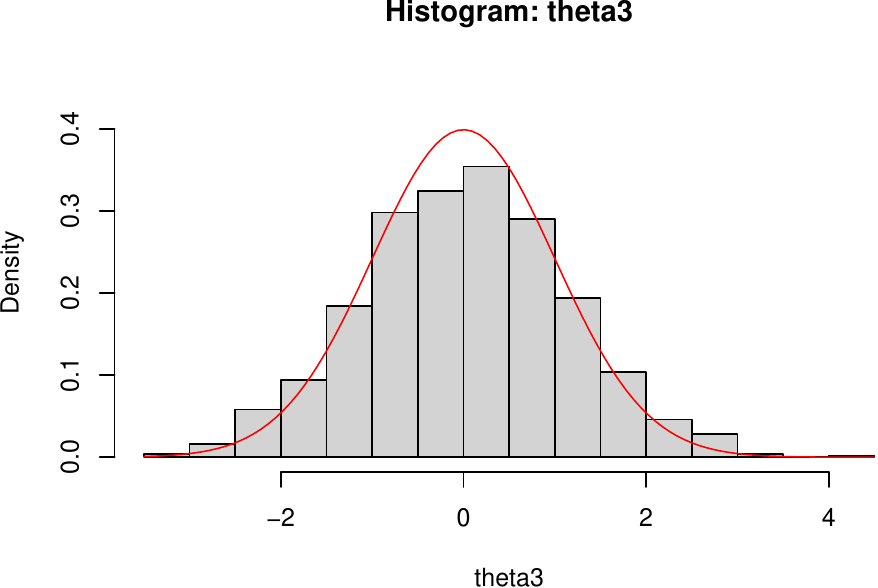}
\end{center}
\end{minipage}

\end{tabular}
\caption{Histograms of $u_{1,n}(\lambda)$, $u_{2,n}(\lambda)$, and $u_{3,n}(\lambda)$ corresponding to the density-power estimator in Section \ref{se:simu22} (i) ($q=0.01n$, $n=5000$, $\lambda=0.2$).}
\label{plot32}
\end{figure}


\begin{figure}[t]
\begin{tabular}{c}

\begin{minipage}{0.32 \hsize}
\begin{center}
\includegraphics[scale=0.25]{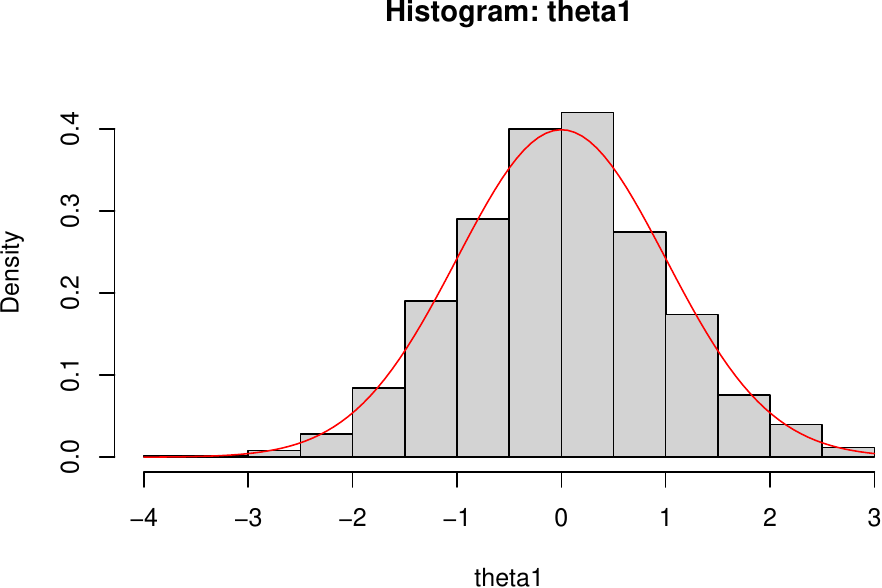}
\end{center}
\end{minipage}

\begin{minipage}{0.32 \hsize}
\begin{center}
\includegraphics[scale=0.25]{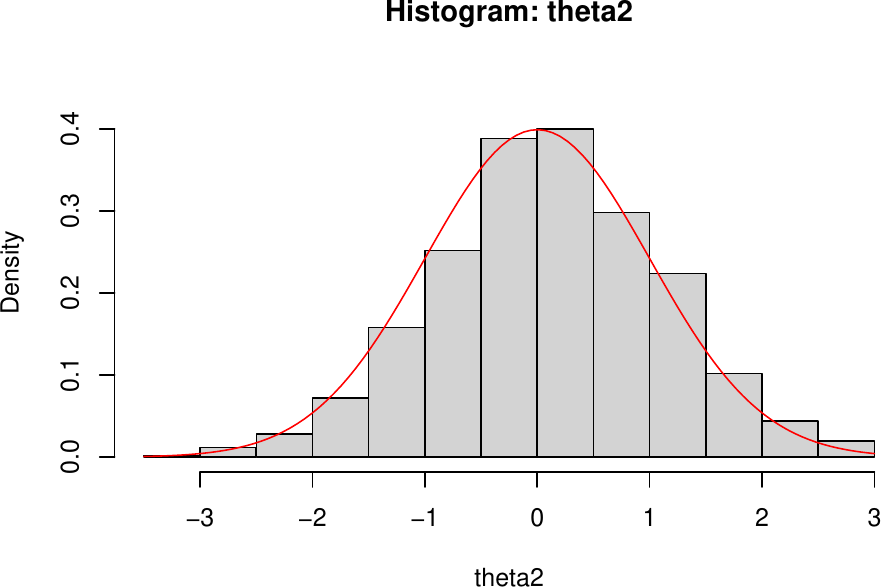}
\end{center}
\end{minipage}

\begin{minipage}{0.32 \hsize}
\begin{center}
\includegraphics[scale=0.25]{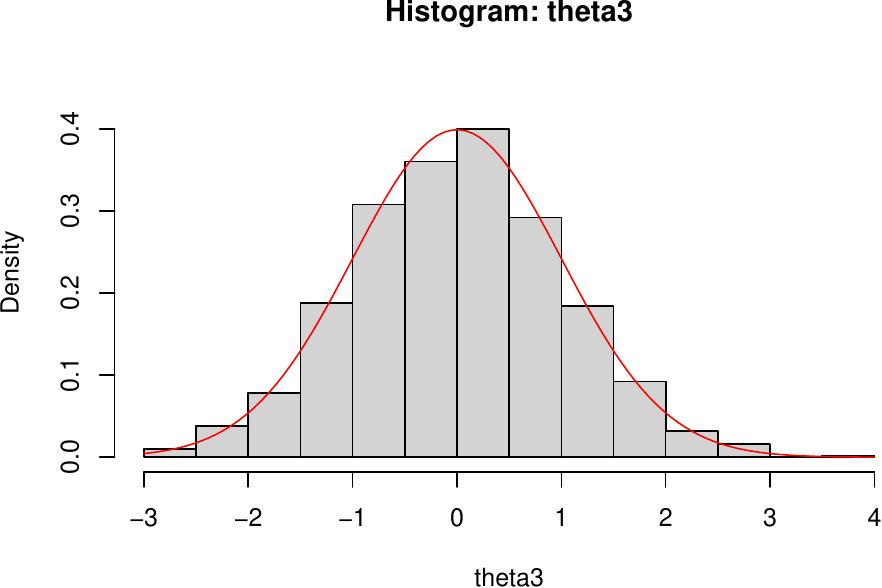}
\end{center}
\end{minipage}

\end{tabular}
\caption{Histograms of $u_{1,n}(\lambda)$, $u_{2,n}(\lambda)$, and $u_{3,n}(\lambda)$ corresponding to the H\"{o}lder-based estimator in Section \ref{se:simu22} (i) ($q=0.01n$, $n=5000$, $\lambda=0.2$).}
\label{holplot32}
\end{figure}














\begin{table}[t]
\begin{center}
\caption{Frequency that the 95\% confidence interval corresponding to the density-power estimator contains the true value in each $\lambda$ in Section \ref{se:simu22} (i) ($q=0.01n$, $n=5000$).}
\begin{tabular}{r | r r r r r r r r r r} \hline
& \multicolumn{9}{c}{} \\[-3mm]
$\lambda$ & 0.1 & 0.2 & 0.3 & 0.4 & 0.5 & 0.6 & 0.7 & 0.8 & 0.9 & 1 \\[1mm] \hline
& \multicolumn{9}{c}{} \\[-3mm]
$\theta_{1}$ & 0.92 & 0.94 & 0.93 & 0.93 & 0.94 & 0.94 & 0.93 & 0.93 & 0.94 & 0.94 \\[1mm]
$\theta_{2}$ & 0.90 & 0.92 & 0.92 & 0.92 & 0.92 & 0.93 & 0.93 & 0.93 & 0.93 & 0.94 \\[1mm]
$\theta_{3}$ & 0.90 & 0.92 & 0.92 & 0.92 & 0.92 & 0.93 & 0.94 & 0.94 & 0.94 & 0.94 \\[1mm] \hline
\end{tabular}
\label{ciratio31}
\end{center}
\end{table}


\begin{table}[t]
\begin{center}
\caption{Frequency that the 95\% confidence interval corresponding to the H\"{o}lder-based estimator contains the true value in each $\lambda$ in Section \ref{se:simu22} (i) ($q=0.01n$, $n=5000$).}
\serev{
\begin{tabular}{r | r r r r r r r r r r} \hline
& \multicolumn{9}{c}{} \\[-3mm]
$\lambda$ & 0.1 & 0.2 & 0.3 & 0.4 & 0.5 & 0.6 & 0.7 & 0.8 & 0.9 & 1 \\[1mm] \hline
& \multicolumn{9}{c}{} \\[-3mm]
$\theta_{1}$ & 0.95 & 0.95 & 0.95 & 0.95 & 0.94 & 0.94 & 0.94 & 0.94 & 0.94 & 0.94 \\[1mm]
$\theta_{2}$ & 0.94 & 0.94 & 0.94 & 0.94 & 0.93 & 0.93 & 0.93 & 0.93 & 0.93 & 0.93 \\[1mm]
$\theta_{3}$ & 0.95 & 0.95 & 0.94 & 0.94 & 0.94 & 0.94 & 0.94 & 0.94 & 0.94 & 0.93 \\[1mm] \hline
\end{tabular}
}
\label{holciratio31}
\end{center}
\end{table}

\subsection{Jump-diffusion process} \label{se:simu5}

The sample data $(\ly_{t_{j}})_{j=0}^{n}$ with $t_{j}=j/n$ is obtained from 
\begin{align*}
d\ly_{t}=\ly_{t}dt+\frac{2+3Y_{t}^{\star 2}}{1+Y_{t}^{\star 2}}dw_{t}+dJ_{t}, \quad Y_{0}=0, \quad t\in[0,1],
\end{align*}
where $J$ is a compound Poisson process with intensity $q$, and the distribution representing the jump-size of the compound Poisson process is given by $\serev{\mathcal{N}}(0,3)$.
We set $Y_{t_{j}}=\ly_{t_{j}}$ and consider the following model for the estimation:
\begin{align*}
d\ly_{t}=\frac{\theta_{1}+\theta_{2}Y_{t}^{\star 2}}{1+Y_{t}^{\star 2}}dw_{t}, \qquad t\in[0,1].
\end{align*}
Then, the true parameter $\tz=(\theta_{1,0},\theta_{2,0})^{\top}=(2,3)^{\top}$.
The initial value, lower bound, and upper bound in numerical optimization are given by $5$, $0$, and $10$, respectively.
The simulations are done for $q=0.01n$, $0.05n$.
Figure \ref{pathplot4} shows one of 1000 sample paths for $q=0.01n$ and $n=5000$.
Moreover, Table \ref{esti42} summarizes the estimation results using GQMLE, density-power GQMLE and H\"{o}lder-based GQMLE.
Even in this case, we obtained that the estimation results and behaviors of density-power and H\"{o}lder-based GQMLEs are similar to those of Section \ref{se:simu1}.

\serev{
\begin{rem}
The choice of the tuning parameter $\lambda$ is important, and the optimal value may depend on the configuration considered.
Several studies have investigated the selection of $\lambda$; see, for example, \cite{WarJon05, BasBasJon21}.
In the literature on density-power divergence-based estimators, values around $0.2$ are commonly used as a tuning parameter, providing a good balance between efficiency and robustness.
\end{rem}
}


\begin{figure}[t]
\begin{tabular}{c}
\includegraphics[scale=0.35]{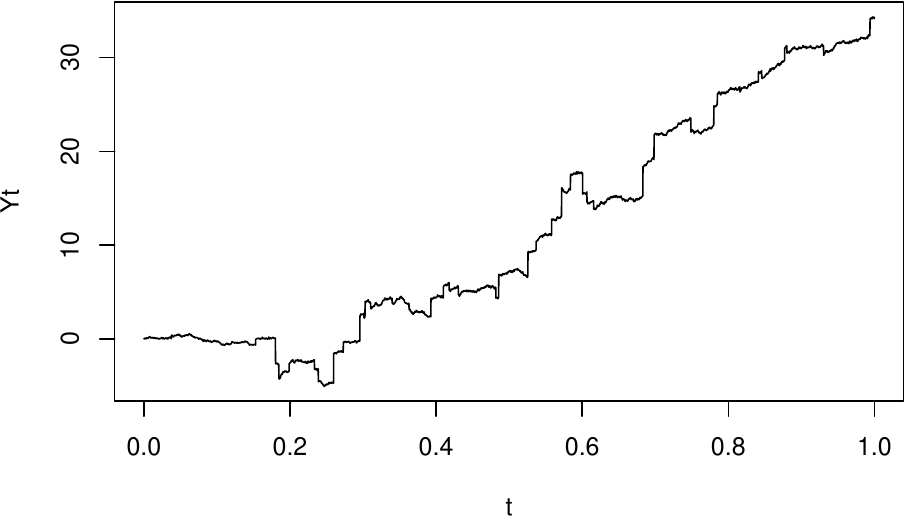}
\end{tabular}
\caption{One of 1000 sample paths in Section \ref{se:simu5} ($q=0.01n$, $n=5000$).}
\label{pathplot4}
\end{figure}


\begin{table}[t]
\begin{center}
\caption{GQMLE, Density-power GQMLE, and H\"{o}lder-based GQMLE in Section \ref{se:simu5} ($\theta_{0}=(2,3)^{\top}$). ``time'' shows the mean calculation time.}
\scalebox{0.73}[0.73]{
\serev{
\begin{tabular}{r r | r r | r r | r r | r r} \hline
& & \multicolumn{2}{c}{} & \multicolumn{2}{c|}{} & \multicolumn{2}{c}{} & \multicolumn{2}{c}{} \\[-3mm]
& & \multicolumn{4}{l|}{$q=0.01n$} & \multicolumn{4}{l}{$q=0.05n$} \\ \cline{3-10}
& & \multicolumn{2}{c|}{} & \multicolumn{2}{c|}{} & \multicolumn{2}{c|}{} & \multicolumn{2}{c}{} \\[-3mm]
& & \multicolumn{2}{l|}{$n=1000$} & \multicolumn{2}{l|}{$n=5000$} & \multicolumn{2}{l|}{$n=1000$} & \multicolumn{2}{l}{$n=5000$} \\ \hline
& & \multicolumn{2}{c|}{} & \multicolumn{2}{c|}{} & \multicolumn{2}{c|}{} & \multicolumn{2}{c}{} \\[-3mm]
\multicolumn{1}{l}{GQMLE} & & $\hat{\theta}_{1,n}$ & $\hat{\theta}_{2,n}$ & $\hat{\theta}_{1,n}$ & $\hat{\theta}_{2,n}$ & $\hat{\theta}_{1,n}$ & $\hat{\theta}_{2,n}$ & $\hat{\theta}_{1,n}$ & $\hat{\theta}_{2,n}$ \\ \hline
& & \multicolumn{2}{c|}{} & \multicolumn{2}{c|}{} & \multicolumn{2}{c|}{} & \multicolumn{2}{c}{} \\[-3mm]
& mean & 5.9430 & 6.0008 & 9.5942 & 9.9750 & 9.6044 & 9.9660 & 10.0000 & 10.0000 \\[1mm] 
& s.d. & 2.8684 & 1.6749 & 1.1607 & 0.1631 & 1.2025 & 0.2013 & 0.0000 & 0.0000 \\[1mm] 
& & \multicolumn{2}{c|}{} & \multicolumn{2}{r|}{(time: 0.4775)} & \multicolumn{2}{c|}{} & \multicolumn{2}{r}{(time: 0.4286)} \\[1mm] \hline
& & \multicolumn{2}{c|}{} & \multicolumn{2}{c|}{} & \multicolumn{2}{c|}{} & \multicolumn{2}{c}{} \\[-3mm]
\multicolumn{1}{l}{Density-power:} & $\lambda=0.1$ & $\hat{\theta}_{1,n}(\lambda)$ & $\hat{\theta}_{2,n}(\lambda)$ & $\hat{\theta}_{1,n}(\lambda)$ & $\hat{\theta}_{2,n}(\lambda)$ & $\hat{\theta}_{1,n}(\lambda)$ & $\hat{\theta}_{2,n}(\lambda)$ & $\hat{\theta}_{1,n}(\lambda)$ & $\hat{\theta}_{2,n}(\lambda)$ \\ \hline
& & \multicolumn{2}{c|}{} & \multicolumn{2}{c|}{} & \multicolumn{2}{c|}{} & \multicolumn{2}{c}{} \\[-3mm]
& mean & 2.0207 & 3.0405 & 2.0044 & 3.0206 & 2.1023 & 3.2017 & 2.0293 & 3.1111 \\[1mm] 
& s.d. & 0.1703 & 0.0979 & 0.1157 & 0.0382 & 0.4867 & 0.1142 & 0.2652 & 0.0573 \\[1mm]
 & & \multicolumn{2}{c|}{} & \multicolumn{2}{r|}{(time: 0.3659)} & \multicolumn{2}{c|}{} & \multicolumn{2}{r}{(time: 0.3779)} \\[1mm] \hline
& & \multicolumn{2}{c|}{} & \multicolumn{2}{c|}{} & \multicolumn{2}{c|}{} & \multicolumn{2}{c}{} \\[-3mm]
& $\lambda=0.5$ & $\hat{\theta}_{1,n}(\lambda)$ & $\hat{\theta}_{2,n}(\lambda)$ & $\hat{\theta}_{1,n}(\lambda)$ & $\hat{\theta}_{2,n}(\lambda)$ & $\hat{\theta}_{1,n}(\lambda)$ & $\hat{\theta}_{2,n}(\lambda)$ & $\hat{\theta}_{1,n}(\lambda)$ & $\hat{\theta}_{2,n}(\lambda)$ \\ \hline
& & \multicolumn{2}{c|}{} & \multicolumn{2}{c|}{} & \multicolumn{2}{c|}{} & \multicolumn{2}{c}{} \\[-3mm]
& mean & 2.0215 & 3.0236 & 2.0089 & 3.0169 & 2.0763 & 3.0966 & 2.0403 & 3.0901 \\[1mm] 
& s.d. & 0.1916 & 0.1087 & 0.1476 & 0.0424 & 0.3761 & 0.1042 & 0.2554 & 0.0576 \\[1mm]
& & \multicolumn{2}{c|}{} & \multicolumn{2}{r|}{(time: 0.3775)} & \multicolumn{2}{c|}{} & \multicolumn{2}{r}{(time: 0.3874)} \\[1mm] \hline
& & \multicolumn{2}{c|}{} & \multicolumn{2}{c|}{} & \multicolumn{2}{c|}{} & \multicolumn{2}{c}{} \\[-3mm]
& $\lambda=0.9$ & $\hat{\theta}_{1,n}(\lambda)$ & $\hat{\theta}_{2,n}(\lambda)$ & $\hat{\theta}_{1,n}(\lambda)$ & $\hat{\theta}_{2,n}(\lambda)$ & $\hat{\theta}_{1,n}(\lambda)$ & $\hat{\theta}_{2,n}(\lambda)$ & $\hat{\theta}_{1,n}(\lambda)$ & $\hat{\theta}_{2,n}(\lambda)$ \\ \hline
& & \multicolumn{2}{c|}{} & \multicolumn{2}{c|}{} & \multicolumn{2}{c|}{} & \multicolumn{2}{c}{} \\[-3mm]
& mean & 2.0323 & 3.0294 & 2.0169 & 3.0243 & 2.1198 & 3.1256 & 2.0763 & 3.1272 \\[1mm] 
& s.d. & 0.2276 & 0.1256 & 0.1755 & 0.0478 & 0.4682 & 0.1165 & 0.2897 & 0.0615 \\[1mm] 
& & \multicolumn{2}{c|}{} & \multicolumn{2}{r|}{(time: 0.4023)} & \multicolumn{2}{c|}{} & \multicolumn{2}{r}{(time: 0.4423)} \\[1mm] \hline
& & \multicolumn{2}{c|}{} & \multicolumn{2}{c|}{} & \multicolumn{2}{c|}{} & \multicolumn{2}{c}{} \\[-3mm]
\multicolumn{1}{l}{H\"{o}lder-based:} & $\lambda=0.1$ & $\hat{\theta}_{1,n}(\lambda)$ & $\hat{\theta}_{2,n}(\lambda)$ & $\hat{\theta}_{1,n}(\lambda)$ & $\hat{\theta}_{2,n}(\lambda)$ & $\hat{\theta}_{1,n}(\lambda)$ & $\hat{\theta}_{2,n}(\lambda)$ & $\hat{\theta}_{1,n}(\lambda)$ & $\hat{\theta}_{2,n}(\lambda)$ \\ \hline
& & \multicolumn{2}{c|}{} & \multicolumn{2}{c|}{} & \multicolumn{2}{c|}{} & \multicolumn{2}{c}{} \\[-3mm]
& mean & 2.0189 & 3.0393 & 2.0023 & 3.0191 & 2.0895 & 3.1945 & 2.0193 & 3.1031 \\[1mm] 
& s.d. & 0.1699 & 0.0979 & 0.1193 & 0.0382 & 0.4991 & 0.1137 & 0.2722 & 0.0571 \\[1mm]
& & \multicolumn{2}{c|}{} & \multicolumn{2}{r|}{(time: 0.2610)} & \multicolumn{2}{c|}{} & \multicolumn{2}{r}{(time: 0.2960)} \\[1mm] \hline
& & \multicolumn{2}{c|}{} & \multicolumn{2}{c|}{} & \multicolumn{2}{c|}{} & \multicolumn{2}{c}{} \\[-3mm]
& $\lambda=0.5$ & $\hat{\theta}_{1,n}(\lambda)$ & $\hat{\theta}_{2,n}(\lambda)$ & $\hat{\theta}_{1,n}(\lambda)$ & $\hat{\theta}_{2,n}(\lambda)$ & $\hat{\theta}_{1,n}(\lambda)$ & $\hat{\theta}_{2,n}(\lambda)$ & $\hat{\theta}_{1,n}(\lambda)$ & $\hat{\theta}_{2,n}(\lambda)$ \\ \hline
& & \multicolumn{2}{c|}{} & \multicolumn{2}{c|}{} & \multicolumn{2}{c|}{} & \multicolumn{2}{c}{} \\[-3mm]
& mean & 2.0107 & 3.0152 & 1.9962 & 3.0075 & 2.0106 & 3.0520 & 1.9891 & 3.0412 \\[1mm] 
& s.d. & 0.1924 & 0.1105 & 0.1483 & 0.0428 & 0.3119 & 0.1040 & 0.2459 & 0.0571 \\[1mm]
& & \multicolumn{2}{c|}{} & \multicolumn{2}{r|}{(time: 0.2351)} & \multicolumn{2}{c|}{} & \multicolumn{2}{r}{(time: 0.2731)} \\[1mm] \hline
& & \multicolumn{2}{c|}{} & \multicolumn{2}{c|}{} & \multicolumn{2}{c|}{} & \multicolumn{2}{c}{} \\[-3mm]
& $\lambda=0.9$& $\hat{\theta}_{1,n}(\lambda)$ & $\hat{\theta}_{2,n}(\lambda)$ & $\hat{\theta}_{1,n}(\lambda)$ & $\hat{\theta}_{2,n}(\lambda)$ & $\hat{\theta}_{1,n}(\lambda)$ & $\hat{\theta}_{2,n}(\lambda)$ & $\hat{\theta}_{1,n}(\lambda)$ & $\hat{\theta}_{2,n}(\lambda)$ \\ \hline
& & \multicolumn{2}{c|}{} & \multicolumn{2}{c|}{} & \multicolumn{2}{c|}{} & \multicolumn{2}{c}{} \\[-3mm]
& mean & 2.0057 & 3.0138 & 1.9923 & 3.0070 & 2.0061 & 3.0434 & 1.9795 & 3.0375 \\[1mm] 
& s.d. & 0.2680 & 0.1375 & 0.1894 & 0.0513 & 0.4720 & 0.1233 & 0.3187 & 0.0625 \\[1mm]
& & \multicolumn{2}{c|}{} & \multicolumn{2}{r|}{(time: 0.2345)} & \multicolumn{2}{c|}{} & \multicolumn{2}{r}{(time: 0.2635)} \\[1mm] \hline
\end{tabular}
}
}
\label{esti42}
\end{center}
\end{table}

\subsection{Clustering} \label{se:simu6}

In this section, we consider the procedure for clustering into jump (or spike) and continuous parts using one sample \serev{dataset} selected from the 1000 independent \serev{datasets} generated in Sections \ref{se:simu12} and \ref{se:simu22} (i).
The primary purpose here is to propose a practical method, although we do not investigate the theoretical properties.
We use the cases $\sigma^{2}=1$, $\mathfrak{p}=0.01$, and $n=5000$ for the data in Section \ref{se:simu12} and $q=0.01n$ and $n=5000$ for the data in Section \ref{se:simu22} (i).
The clustering procedure is as follows.
\begin{itemize}
\item First, we compute the residuals
\begin{align*}
\hat{\epsilon}_{j}=h^{-1/2}\sigma_{j-1}^{-1}\big(\hat{\theta}_{n}(\lambda)\big)\D_{j}Y
\end{align*}
for $j=1,2,\ldots,n$.
In this simulation, $\sigma(\theta)=\sigma(\theta_{1},\theta_{2},\theta_{3})=\exp\big\{(\theta_{1}X_{1,t}+\theta_{2}X_{2,t}+\theta_{3}X_{3,t})/2\big\}$.

\item Next, we apply the $K$-means clustering ($K\geq2$) to $\{|\hat{\epsilon}_{j}|\}_{j=1,2,\ldots,n}$. We denote the clusters $\kappa_{n}^{1}$, $\kappa_{n}^{2}$, $\ldots$, $\kappa_{n}^{K}$ in ascending order of the number of elements in the cluster.

\item \serev{
Our clustering approach is based on differences in the observed data. Since we assume finite-activity jumps and spikes, their total number is expected to be much smaller than $n$. Consequently, we regard the cluster with the largest number of elements as the continuous component. The remaining clusters are classified as the jump components, which are separated according to their respective magnitudes (or levels).
Specifically,} we set $\mathfrak{C}_{n}=\kappa_{n}^{K}$ and $\mathfrak{D}_{n}=\kappa_{n}^{1}\sqcup\kappa_{n}^{2}\sqcup\cdots\sqcup\kappa_{n}^{K-1}$, respectively. Then, $\{|\hat{\epsilon}_{j}|\}_{j=1,2,\ldots,n}=\mathfrak{C}_{n}\sqcup \mathfrak{D}_{n}$, and $\mathfrak{C}_{n}$ and $\mathfrak{D}_{n}$ are regarded as the continuous and jump parts, respectively.

\item \serev{In case of spike contamination (Section \ref{se:simu12}), if a spike occurs at time $t_{j}$, no spikes typically occur at $t_{j-1}$ or $t_{j+1}$; then, both $|\Delta_{j}Y|$ and $|\Delta_{j+1}Y|$, which are used for the clustering, tend to be large. 
In this case, although the correct classification assigns $j$ to $\mathfrak{D}_{n}$ and $j+1$ to $\mathfrak{C}_{n}$, this clustering tends to classify both $j$ and $j+1$ into $\mathfrak{D}_{n}$.
Hence, for the spike data, we reassign $j+1$ to $\mathfrak{C}_{n}$ if $j\in \mathfrak{D}_{n}$ and $j+1\in \mathfrak{D}_{n}$.}

\end{itemize}

To select the number of clusters $K$, we run the $K$-means for several $K$ and observe the number of elements in $\mathfrak{D}_{n}$.
The detection of the jump components is insufficient when $K$ is too small; on the other hand, even the continuous components are detected as a jump when $K$ is too large.
Therefore, we consider the value of $K$ when the number of elements in $\mathfrak{D}_{n}$ suddenly increases to the point at which continuous components start to be misclassified as jumps; that is, if the number of elements in $\mathfrak{D}_{n}$ changes significantly when $K=k_{0}$, we set $K=k_{0}-1$ for the actual clustering of $K$-means.

\serev{The clustering simulations are conducted for the density-power GQMLE with $\lambda=0.1,0.2,0.5$, and $0.9$.
Since the elements classified into $\mathfrak{C}_n$ and $\mathfrak{D}_n$ do not change significantly with respect to $\lambda$, we present only the results for $\lambda = 0.2$.}
Figure \ref{cluster1} shows the logarithm of the number of elements in $\mathfrak{D}_{n}$ when the $K$-means is applied for each of $K=2,3,\ldots,15$.
The left one is the case of data in Section \ref{se:simu12}, and the right one is the case of data in \ref{se:simu22} (i).
From Figure \ref{cluster1}, the number of elements in $\mathfrak{D}_{n}$ increases abruptly at $K=5$ and $K=7$.
Therefore, we use the 4-means for the clustering of the data in Section \ref{se:simu12} and the 6-means for the clustering of the data in Section \ref{se:simu22} (i).
Figure \ref{cluster2} shows the path of data used for clustering with the results of $K$-means.
The left one is the case of 4-means for data in Section \ref{se:simu12}, and the right one is the case of 6-means for data in Section \ref{se:simu22} (i).
The cross mark points in Figure \ref{cluster2} mean the elements included in $\mathfrak{D}_{n}$.
In the case of 4-means for data in Section \ref{se:simu12}, the proportion of correctly identified noise points is $29/39=0.74$.
Moreover, in the case of 6-means for data in Section \ref{se:simu22} (i), the intensity of the compound Poisson process is $q=0.01,n=50$, while the number of elements of $\mathfrak{D}_{n}$ is 38.
The simulation results for $\lambda=0.1, 0.5$, and $0.9$ are provided in Appendix \ref{suppl_2}.


\begin{figure}[t]
\begin{tabular}{c}

\begin{minipage}{0.45 \hsize}
\begin{center}
\includegraphics[scale=0.33]{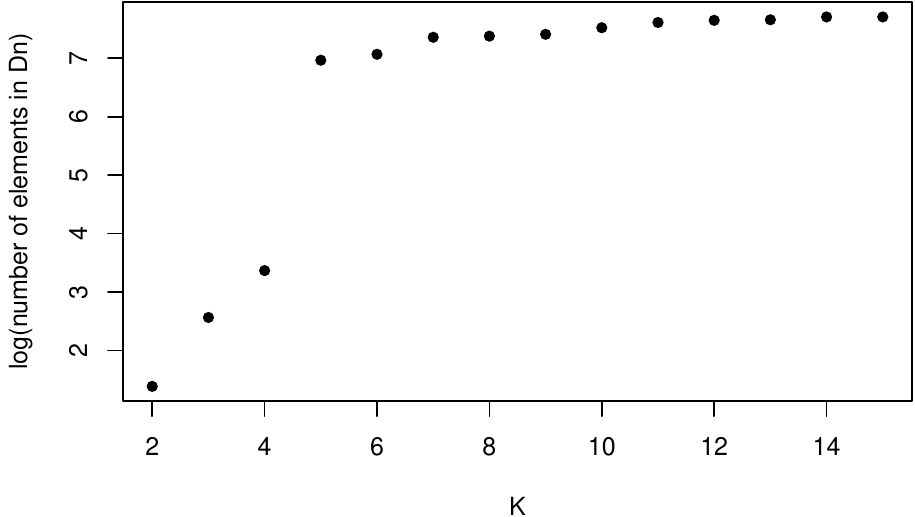}
\end{center}
\end{minipage}

\begin{minipage}{0.45 \hsize}
\begin{center}
\includegraphics[scale=0.33]{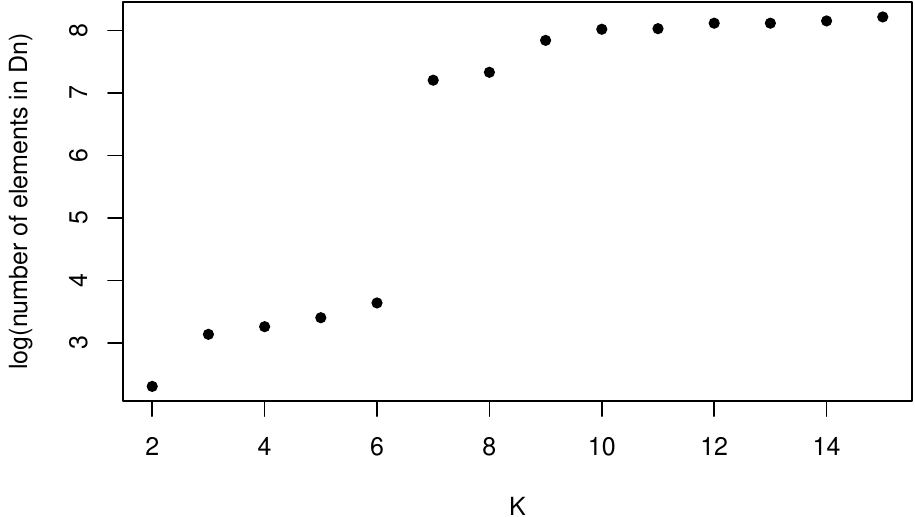}
\end{center}
\end{minipage}

\end{tabular}
\caption{Logarithmic of the number of elements in $\mathfrak{D}_{n}$ in $K$-means for each $K$. The data in Section \ref{se:simu12} is used for the left one, and the data in Section \ref{se:simu22} (i) is used for the right one ($\lambda=0.2$).}
\label{cluster1}
\end{figure}


\begin{figure}[t]
\begin{tabular}{c}

\begin{minipage}{0.45 \hsize}
\begin{center}
\includegraphics[scale=0.33]{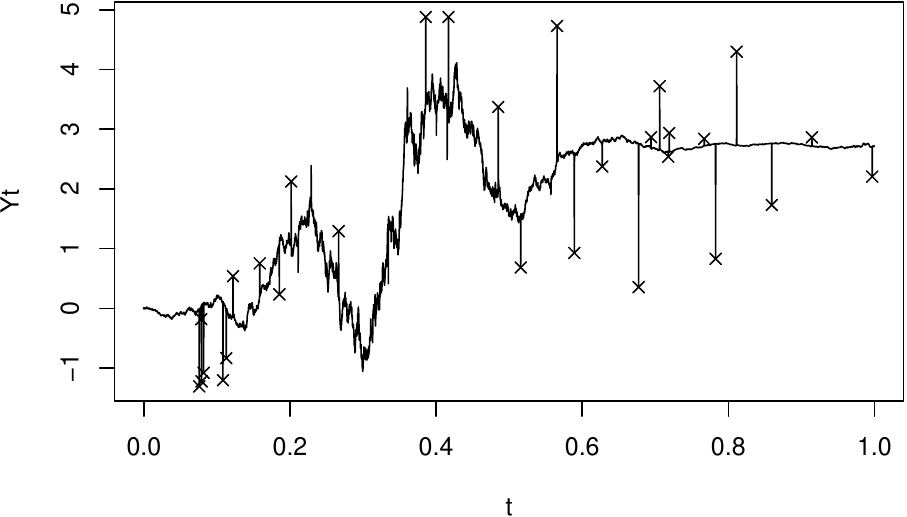}
\end{center}
\end{minipage}

\begin{minipage}{0.45 \hsize}
\begin{center}
\includegraphics[scale=0.33]{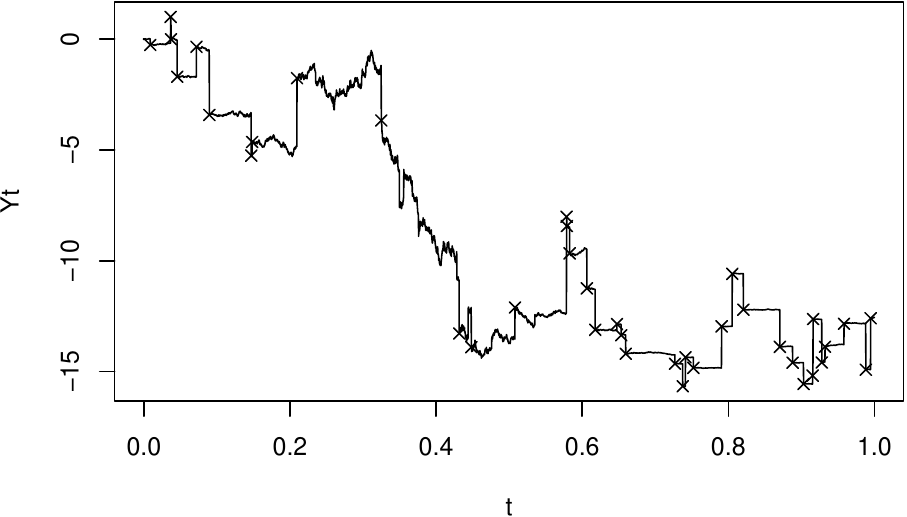}
\end{center}
\end{minipage}

\end{tabular}
\caption{The path of data used for clustering with the results of $K$-means ($\lambda=0.2$). The cross mark points mean elements of $\mathfrak{D}_{n}$. The 4-means for data in Section \ref{se:simu12} are used for the left one, and the 6-means for data in Section \ref{se:simu22} (i) are used for the right one.}
\label{cluster2}
\end{figure}


\appendix

\section{Auxiliary asymptotics}
\label{hm:sec_aux.asymp}

Throughout this section, we will work under Assumptions \ref{hm:A_diff.coeff} to \ref{hm:A_lam}.

\subsection{A class of random functions}
\label{hm:sec_class.U}

To deal with the density-power GQLF \eqref{hm:def_dp-H-0} and the H\"{o}lder-based GQLF \eqref{hm:def_ldp-H-0} in a unified manner, we temporarily consider a class of random functions of the form
\begin{equation}\label{hm:U_def}
    U_n(\theta;\lam) := \sumj \zeta(X_{t_{j-1}},\nY_j,\theta;\lam)=:\sumj \zeta_j(\theta;\lam),
\end{equation}
where the function $(x,y,\theta) \mapsto \zeta(x,y,\theta;\lam)$ is smooth enough for each $\lam\in(0,\overline{\lam}]$ to satisfy the following conditions:
\begin{align}
    & \max_{0\le k\le 3}\sup_{(y,\theta)}\left|\p_\theta^k \zeta(x,y,\theta;\lam)\right| \lesssim \frac{1}{\lam}(1+|x|^C),
    \label{hm:zeta-A1}\\
    & \max_{0\le k\le 3}\sup_{(y,\theta)}\left| \p_y^m \p_\theta^k \zeta(x,y,\theta;\lam)\right| \lesssim \lam^{m/2 -1} (1+|x|^C),
    \qquad m=1,2,
    \label{hm:zeta-A2}\\
    & \max_{0\le k\le 3}\sup_{(y,\theta)}\left| \p_y^m \p_\theta^k \zeta(x,y,\theta;\lam)\right| \lesssim (1+|x|^C)(1+|y|^C),
    \qquad m\ge 0,
    \label{hm:zeta-A3}\\
    & \text{The function $y \mapsto \p_\theta^k \zeta(x,y,\theta;\lam)$ is even for each $(x,\theta,\lam)$ and $k\le 4$.}
    \label{hm:zeta-A4}
\end{align}
These will be used subsequently in different ways.

To proceed, let us introduce the following more concise notation:
\begin{align}
    \mbbh_n(\theta;\lam) &= \sumj \beta_{j-1}(\theta;\lam) \left(\frac{1}{\lam} \vp_j(\theta)^\lam -K_{\lam,d}\right),
    \label{hm:def_dp-H-0--}\\
    \mbbh_n(\theta;\lam) &= \sumj \gam_{j-1}(\theta;\lam) \, \frac{1}{\lam} \vp_j(\theta)^\lam,
    \label{hm:def_ldp-H-0--}
\end{align}
for \eqref{hm:def_dp-H-0} and \eqref{hm:def_ldp-H-0}, respectively, 
where 
\begin{align}
    \beta_{j-1}(\theta;\lam) &:= \mathsf{d}_{j-1}(\theta)^{-\lam/2},
    \nn\\
    \gam_{j-1}(\theta;\lam) &:= \mathsf{d}_{j-1}(\theta)^{-\lam/(2(\lam+1))},
    \nn\\
    \vp_j(\theta) &:= \phi\big(S_{j-1}(\theta)^{-1/2}\nY_j\big).
\end{align}
We abbreviate ``$\sup_{(\lam,\theta)\in(0,\overline{\lam}]\times\overline{\Theta}}$'' as ``$\sup_{\lam,\theta}$''.

\begin{lem}\label{hm:lem_dp.ldp_U.verify}
Both the density-power GQLF and the H\"{o}lder GQLF belong to the class \eqref{hm:U_def}:
the four properties \eqref{hm:zeta-A1} to \eqref{hm:zeta-A4} are fulfilled by $U_n(\theta;\lam)=\mbbh_n(\theta;\lam)$ for both \eqref{hm:def_dp-H-0} and \eqref{hm:def_ldp-H-0}.
\end{lem}

\begin{proof}
We begin with \eqref{hm:zeta-A1}. By Assumption \ref{hm:A_diff.coeff},
\begin{align}\label{hm:pre.cal-1}
    & \max_{1\le l\le 4}\sup_{\lam,\theta} \left|\frac{1}{\lam}\p_\theta^l \beta_{j-1}(\theta;\lam)\right| 
    + \max_{1\le l\le 4}\sup_{\lam,\theta} \left|\frac{1}{\lam}\p_\theta^l \gam_{j-1}(\theta;\lam)\right| 
    \nn\\
    &{}\qquad + \max_{0\le m\le 4}\sup_\theta \left|\p_\theta^m (S^{-1})_{j-1}(\theta)\right| \lesssim 1+|X_{t_{j-1}}|^C.
\end{align}
Direct computations using \eqref{hm:pre.cal-1} inductively show that for $k\ge 1$,
\begin{align}
\sup_{\theta} \left|\p_\theta^k \left(\vp_j(\theta)^\lam\right) \right|
&\lesssim \lambda(1+|X_{t_{j-1}}|^C) \bigg(1+\sum_{l=1}^{k} \lambda^{l-1} |\nY_j|^{2l}\bigg) \vp_j(\theta)^\lambda 
\nn\\
&\lesssim (1+|X_{t_{j-1}}|^C) \bigg(\lam+\sum_{l=1}^{k}(|\sqrt{\lam}\,\nY_j|^2)^l \bigg)
\nn\\
&{}\qquad \times 
\exp\left(- C' (1+|X_{t_{j-1}}|)^{-c_S'}\,|\sqrt{\lam}\,\nY_j|^2\right).
\label{hm:pre.cal-2}
\end{align}
The following elementary inequality is valid for $c>0$ and $\al\ge 0$ ($0^0:=1$):
\begin{equation}\label{hm:elementary.ineq}
\sup_{x\ge 0} \, x^{\al} e^{-cx} \le \al^\al e^{-\al} c^{-\al}.
\end{equation}
Applying \eqref{hm:elementary.ineq} with $x=|\sqrt{\lam}\,\nY_j|^2$ to the upper bound in \eqref{hm:pre.cal-2}, we get
\begin{equation}\label{hm:deriv.estimate}
    \sup_{\lam,\theta} 
    \left| \p_\theta^k \left(\vp_j(\theta)^\lam\right) \right|
    \lesssim 1+|X_{t_{j-1}}|^C
\end{equation}
for $k\ge 0$; the case of $k=0$ is trivial. Using \eqref{hm:pre.cal-1} and \eqref{hm:deriv.estimate}, it is straightforward to verify \eqref{hm:zeta-A1}. 

Turning to \eqref{hm:zeta-A2}, we may and do suppose that all the involved random variables are one-dimensional. Direct calculations lead to the expression
\begin{align}
    \p_\theta^k \left(\vp_j(\theta)^\lam\right)
    &= \mathsf{C}_\lam \vp_j(\theta)^\lam \sum_{l=1}^{k} A_{l,j-1}^{(1)}(\theta) \left(|\sqrt{\lam}\,\nY_j|^2\right)^l,
\end{align}
and then
\begin{align}
    \p_y \p_\theta^k \left(\vp_j(\theta)^\lam\right) 
    &= \mathsf{C}_\lam \vp_j(\theta)^\lam \sum_{l=1}^{k} \left( A_{l,j-1}^{(2)}(\theta) \lam^{l+1} \,\nY_j^{2l+1} 
    + A_{l,j-1}^{(3)}(\theta) \lam^{l} \,\nY_j^{2l-1}\right).
\end{align}
Here, we generically denoted by $\mathsf{C}_\lam$ a positive constant depending on $\lam$ such that $\sup_\lam \mathsf{C}_\lam <\infty$, and by $A_l^{(m)}(x,\theta)$ ($l\ge 1$; $m=1,2,3$) sufficiently smooth functions satisfying that $\sup_\theta |\p_\theta^k A_l^{(m)}(x,\theta)| \lesssim 1+|x|^C$. 
Hence, by \eqref{hm:elementary.ineq},
\begin{align}
    \left|\p_y \p_\theta^k \left(\vp_j(\theta)^\lam\right) \right|
    &\lesssim \sqrt{\lam}\, (1+|X_{t_{j-1}}|^C).
\end{align}
Similarly, we get
\begin{align}
    \left|\p_y^2 \p_\theta^k \left(\vp_j(\theta)^\lam\right) \right|
    &\lesssim \lam\, (1+|X_{t_{j-1}}|^C).
\end{align}
Using the last two estimates, we can verify \eqref{hm:zeta-A2} for both \eqref{hm:def_dp-H-0} and \eqref{hm:def_ldp-H-0}; in particular, it follows that
\begin{equation}
    \left|\p_y^m \p_\theta^k U_n(\theta;\lam)\right| 
    \lesssim \lam^{m/2 -1} \sumj (1+|X_{t_{j-1}}|^C).
\end{equation}

The property \eqref{hm:zeta-A3} is easily seen from the essential boundedness of $\vp_j(\theta)$. Finally, \eqref{hm:zeta-A4} is trivial. The proof is complete.
\end{proof}

\medskip

In the rest of this section and Sections \ref{hm:sec_robustification} and \ref{hm:sec_lim.thm}, we work with the random function \eqref{hm:U_def} satisfying the properties \eqref{hm:zeta-A1} to \eqref{hm:zeta-A4}. 
To decompose $U_n(\theta;\lam)$ into leading and negligible parts, we need the stochastic expansions of $\nY_j$ and $X_{t_{j-1}}$ on the event $G_j$; recall the definition \eqref{hm:def_Gj}. 
Let us write $\sig^\star_{j-1} = \sig(X^\star_{t_{j-1}},\tz)$ and $\p_\theta^k \sig^\star_{j-1} = \p_\theta^k\sig(X^\star_{t_{j-1}},\tz)$ ($k\ge 1$).
By \eqref{hm:DY_Gj}, we have
\begin{align}
    \D_j Y 
    &= \sig^\star_{j-1}\D_j w + \int_j \mu_s ds 
    + \int_j (\sig(X^\star_{s-},\tz) - \sig^\star_{j-1} )dw_s
    \nn
\end{align}
on $G_j$. By expanding $\sig(X^\star_{s-},\tz) - \sig^\star_{j-1}$ and noting that $\lx_s-\lx_{t_{j-1}}=\cx_s-\cx_{t_{j-1}}$ ($s\in\mbbi_j$) on $G_j$, we can write
\begin{align}\label{hm:dp-const-5}
\nY_j &= \sig^\star_{j-1} Z_j 
+ \sqrt{h}\, \overline{R}_{j}
\end{align}
on $G_j$, where
\begin{equation}\label{hm:Rj_form}
    \overline{R}_j = \mu_{j-1} + \del_{\mu,j} + \del_{\sig,1,j} + \sqrt{h}\,\del_{\sig,2,j}
\end{equation}
with
\begin{align}
    \del_{\mu,j} &:= \frac1h \int_j (\mu_s-\mu_{j-1})ds,
    \nn\\
    \del_{\sig,1,j} &:= \frac{1}{\sqrt{h}}\int_j \frac{1}{\sqrt{h}} (\p_x\sig^\star_{j-1})[\cx_s - \cx_{t_{j-1}}] dw_s,
    \label{hm:def_del.sig.1} \\
    \del_{\sig,2,j} &:= \frac{1}{\sqrt{h}}\int_j 
    \bigg(\int_0^1\!\!\int_0^1 v\p_x^2\sig\big(\lx_{t_{j-1}}+uv(\cx_s - \cx_{t_{j-1}});\tz \big)dudv \bigg)
    \nn\\
    &{}\qquad\qquad \left[\left\{\frac{1}{\sqrt{h}}(\cx_s-\cx_{t_{j-1}})\right\}^{\otimes 2}\right] dw_s.
    \label{hm:def_ny_R}
\end{align}
By Assumption \ref{hm:A_drif.coeff&X} and \eqref{hm:X-moment.estimates}, we immediately get for any $K>0$,
\begin{equation}\label{hm:ol.dels_moment}
    \sup_{n}\max_{j\le n} \E\left[ |\del_{\mu,j}|^K 
    + |\del_{\sig,1,j}|^K 
    + |\del_{\sig,2,j}|^K ;\,G_j \right] 
    < \infty.
\end{equation}
It follows that
\begin{equation}\label{hm:ol.R_moment}
    \sup_{n}\max_{j\le n} \E\left[ |\overline{R}_j|^K ;\,G_j\right] < \infty.
\end{equation}
The expression \eqref{hm:Rj_form} will be used later in estimating several remainder terms.

Write 
\begin{equation}
    \zeta^\star_j(\theta;\lam)=\zeta(\lx_{t_{j-1}},\sig^\star_{j-1}Z_j,\theta;\lam).
\end{equation}
Since $X_{t_{j-1}}=\lx_{t_{j-1}}$ on $G_j$, we have
\begin{align}
    U_n(\theta;\lam) 
    &= \sumj \zeta_j(\theta;\lam)I_{G_j} +\sumj \zeta_j(\theta;\lam)I_{G_j^c}
    \nn\\
    &= \sumj \zeta(\lx_{t_{j-1}},\sig^\star_{j-1} Z_j + \sqrt{h}\, \overline{R}_{j},\theta;\lam) I_{G_j} +\sumj \zeta_j(\theta;\lam)I_{G_j^c}
    \nn\\
    &= \sumj \zeta^\star_j(\theta;\lam) + \sumj \left\{\zeta_j(\theta;\lam) - \zeta^\star_j(\theta;\lam) \right\}I_{G_j^c}
    \nn\\
    &{}\qquad + \sumj \left\{\zeta(\lx_{t_{j-1}},\sig^\star_{j-1} Z_j + \sqrt{h}\, \overline{R}_{j},\theta;\lam) - \zeta^\star_j(\theta;\lam)\right\}I_{G_j}
    \nn\\
    &=: U^\star_n(\theta;\lam) + U_{n}^{\ve,1}(\theta;\lam) + U_{n}^{\ve,2}(\theta;\lam).
\end{align}
In Section \ref{hm:sec_robustification} below, we will show that both $U_n^{\ve,1}(\theta;\lam)$ and $U_n^{\ve,2}(\theta;\lam)$ are asymptotically negligible uniformly in $\theta$ in the sense that \begin{align}\label{hm:disconti.removed}
    \sup_\theta \left|\frac{1}{\sqrt{n}}\p_\theta^k U_{n}^{\ve,1}(\theta;\lam)\right| + \sup_\theta \left|\frac{1}{\sqrt{n}}\p_\theta^k U_{n}^{\ve,2}(\theta;\lam)\right| = o_p(1)
\end{align}
for $k=0,1,2,3$ under Assumption \ref{hm:A_lam}. 
Then, Section \ref{hm:sec_lim.thm} will present some asymptotic properties of $U^\star_n(\theta;\lam)$ with non-trivial limits.

\subsection{Removing discontinuity}
\label{hm:sec_robustification}

This section is devoted to proving \eqref{hm:disconti.removed}, the negligibility of the ``contamination'' terms. We will separately show
\begin{align}
    & \sup_\theta \left|\frac{1}{\sqrt{n}}\p_\theta^k U_{n}^{\ve,1}(\theta;\lam)\right| = o_p(1),
    \label{hm:disconti.removed-part.1}\\
    & \sup_\theta \left|\frac{1}{\sqrt{n}}\p_\theta^k U_{n}^{\ve,2}(\theta;\lam)\right| = o_p(1).
    \label{hm:disconti.removed-part.2}
\end{align}

\subsubsection{Proof of \eqref{hm:disconti.removed-part.1}}

By \eqref{hm:zeta-A1}, \eqref{hm:A_J-1}, \eqref{hm:A_spike-1}, and \eqref{hm:X_m.bound}, we have
\begin{align}
    & \E\left[\sup_\theta \left|\frac{1}{\sqrt{n}}\p_\theta^k U_{n}^{\ve,1}(\theta;\lam)\right|\right]
    \nn\\
    &\lesssim \frac1n \sumj \sqrt{n} \, \E\left[ \sup_\theta \,|\p_\theta^k \zeta_j(\theta;\lam)| + \sup_\theta \,|\p_\theta^k \zeta^\star_j(\theta;\lam)|;\, G_j^c\right] \nn\\
    &\lesssim \frac1n \sumj \E\left[ (1 + |\lx_{t_{j-1}}|^C + |\ox_{t_{j-1}}|^C)\, \pr^{j-1}[G_j^c] \right] \frac{\sqrt{n}}{\lam} \nn\\
    &\lesssim \frac1n \sumj \E\left[1 + |\lx_{t_{j-1}}|^C + |\ox_{t_{j-1}}|^C\right] \frac{\sqrt{n}\,h^\kappa}{\lam} \nn\\
    &\lesssim \frac{\sqrt{n}\,h^\kappa}{\lam}.
\end{align}
The upper bound goes to $0$ as $n\to\infty$ under \eqref{hm:lam.condition-2}, concluding \eqref{hm:disconti.removed-part.1}. 

\subsubsection{Proof of \eqref{hm:disconti.removed-part.2}}
\label{hm:sec_disc.remove-subsec.2}

Let
\begin{equation}
    \zeta^\star_{j-1}(y,\theta;\lam):=\zeta(\lx_{t_{j-1}},y,\theta;\lam)
\end{equation}
for brevity. Then,
\begin{align}
    & \frac{1}{\sqrt{n}}\p_\theta^k U_{n}^{\ve,2}(\theta;\lam)
    \nn\\
    &= \frac{T}{n} \sumj I_{G_j} \p_y\p_\theta^k \zeta^\star_{j-1}(\sig^\star_{j-1}Z_j,\theta;\lam)[\overline{R}_j] \nn\\
    &{}\qquad + \frac{T}{\sqrt{n}} \frac1n \,\sumj I_{G_j} 
    \int_0^1\!\!\int_0^1 v\p_y^2\zeta^\star_{j-1}\left(\sig^\star_{j-1}Z_j+uv\sqrt{h}\overline{R}_j,\theta;\lam\right)dudv [\overline{R}_j^{\otimes 2}] \nn\\
    &=: \overline{U}^{\ve,2}_{1,k,n}(\theta;\lam) + \overline{U}^{\ve,2}_{2,k,n}(\theta;\lam).
    \label{hm:disconti.removed-1}
\end{align}
We have $\sup_\theta |\overline{U}^{\ve,2}_{2,k,n}(\theta;\lam)| = O_p(n^{-1/2})=o_p(1)$ since, by \eqref{hm:zeta-A2} and \eqref{hm:ol.R_moment},
\begin{equation}
    \sup_\theta|\overline{U}^{\ve,2}_{2,k,n}(\theta;\lam)| 
    \lesssim \frac{1}{\sqrt{n}} \frac1n \sumj I_{G_j} (1+|\lx_{t_{j-1}}|^C)|I_{G_j}\overline{R}_j|^2
    =O_p\left(\frac{1}{\sqrt{n}}\right).
\end{equation}
As for $\overline{U}^{\ve,2}_{1,k,n}(\theta;\lam)$, we recall the expression \eqref{hm:Rj_form} of $\overline{R}_j$. 
Let
\begin{equation}\label{hm:def_s.star}
    \mathsf{g}^\star_{k,j}(\theta;\lam) := \p_y\p_\theta^k \zeta^\star_{j-1}(\sig^\star_{j-1}Z_j,\theta;\lam).
\end{equation}
We have the decomposition
\begin{align}
    \overline{U}^{\ve,2}_{1,k,n}(\theta;\lam) = \overline{U}^{\ve,2(1)}_{1,k,n}(\theta;\lam) + \overline{U}^{\ve,2(2)}_{1,k,n}(\theta;\lam) + \overline{U}^{\ve,2(3)}_{1,k,n}(\theta;\lam),
\end{align}
where
\begin{align}
    \overline{U}^{\ve,2(1)}_{1,k,n}(\theta;\lam)
    &:= \frac{T}{n} \sumj I_{G_j} \mathsf{g}^\star_{k,j}(\theta;\lam)\big[\mu_{j-1} + \del_{\sig,1,j}\big],
    \nn\\
    \overline{U}^{\ve,2(2)}_{1,k,n}(\theta;\lam)
    &:= \frac{T}{n} \sumj I_{G_j} \mathsf{g}^\star_{k,j}(\theta;\lam)[\del_{\mu,j}],
    \nn\\
    \overline{U}^{\ve,2(3)}_{1,k,n}(\theta;\lam)
    &:= \frac{T}{n} \sumj I_{G_j} \mathsf{g}^\star_{k,j}(\theta;\lam)\big[\sqrt{h}\,\del_{\sig,2,j} \big].
    \nn
\end{align}
Applying the Cauchy-Schwarz inequality with 
\eqref{hm:zeta-A3}, \eqref{hm:ol.dels_moment}, and \eqref{hm:nY_moment}, we get
\begin{align}
    \frac{1}{h} \sup_{\theta,\lam}\big| \overline{U}^{\ve,2(3)}_{1,k,n}(\theta;\lam) \big|^2
    & \lesssim 
    \frac{1}{n} \sumj I_{G_j}\sup_{\theta,\lam}|\mathsf{g}^\star_{k,j}(\theta;\lam)|^2 
    \times \frac{1}{n} \sumj I_{G_j}|\del_{\sig,2,j}|^2
    \label{hm:CS.apply.ex}\\
    &\lesssim \frac{1}{n} \sumj I_{G_j} (1+|\lx_{t_{j-1}}|^C)(1+|\nY_j|^C)
    \times \frac{1}{n} \sumj I_{G_j}|\del_{\sig,2,j}|^2
    \nn\\
    &=O_p(1)\times O_p(1) = O_p(1).
\end{align}
This concludes that $\sup_{\theta,\lam}\big| \overline{U}^{\ve,2(3)}_{1,k,n}(\theta;\lam) \big| = O_p(\sqrt{h})$. 
Likewise,
\begin{align}
    \sup_{\theta} \big|\overline{U}^{\ve,2(2)}_{1,k,n}(\theta;\lam)\big|^2
    &\lesssim O_p(1) \times \frac{1}{n} \sumj I_{G_j}|\del_{\mu,j}|^2
    \nn\\
    &\lesssim O_p(1) \times \frac{1}{n} \sumj \frac{1}{h}\int_j \big| I_{G_j}(\mu_s - \mu_{t_{j-1}}) \big|^2 ds = o_p(1),
\end{align}
where the last equality immediately follows from Assumption \ref{hm:A_drif.coeff&X}.

It remains to look at $\overline{U}^{\ve,2(1)}_{1,k,n}(\theta;\lam)$. 
To deal with the $w'$-functional part in its summands, we rewrite $I_{G_j}=1-I_{G_j^c}$ and first observe that
\begin{align}
    & \sup_{\theta,\lam}\left|\frac{T}{n} \sumj I_{G_j^c}\, \mathsf{g}^\star_{k,j}(\theta;\lam)[\mu_{j-1} + \del_{\sig,1,j}]\right|^2 \nn\\
    &\lesssim \frac1n \sumj 
    \sup_{\theta,\lam}\left| \mathsf{g}^\star_{k,j}(\theta;\lam)[\mu_{j-1} + \del_{\sig,1,j}] \right|^2
    \times \frac1n\sumj I_{G_j^c}
    \nn\\
    &= O_p(1) \times O_p(h^\kappa) = O_p(h^\kappa) = o_p(1).
\end{align}
Hence,
\begin{equation}\label{hm:del.sig-0}
    \sup_{\theta,\lam}\left| \overline{U}^{\ve,2(1)}_{1,k,n}(\theta;\lam)
    - \frac{T}{n} \sumj \mathsf{g}^\star_{k,j}(\theta;\lam)[\mu_{j-1}]
    - \frac{T}{n} \sumj \mathsf{g}^\star_{k,j}(\theta;\lam)[\del_{\sig,1,j}] \right| = o_p(1).
\end{equation}
The randomness of $\mathsf{g}^\star_{k,j}(\theta;\lam)$ solely comes through $\lx_{t_{j-1}}$ and $\sig^\star_{j-1}Z_j$. 
In view of the definition \eqref{hm:def_s.star} and \eqref{hm:zeta-A4}, the mapping $Z_j \mapsto \mathsf{g}^\star_{k,j}(\theta;\lam)$ is a.s. odd, implying that
\begin{equation}\label{hm:s.star_ce0}
    \E\left[\p_\theta^m \mathsf{g}^\star_{k,j}(\theta;\lam) \right] = 0\qquad \text{a.s. for $k\le 3$ and $m=0,1$.}
\end{equation}

We note the following version of Sobolev's inequality (for example, see \cite{Ada73}): since we are assuming that $\Theta$ is a convex domain, for any $\mcc^1(\Theta)$-function $F:\,\Theta\to \mbbr^m$ (for some $m\ge 1$) and any $K > p$, 
\begin{align}\label{hm:sobolev.ineq}
\sup_{\theta}|F(\theta)|^K \le \mathsf{C}_{\Theta,K} \left( \int_{\Theta}|F(\theta)|^K d\theta + \int_{\Theta}|\p_\theta F(\theta)|^K d\theta \right),
\end{align}
where the constant $\mathsf{C}_{\Theta,K}>0$ only depends on $\Theta$ and $K$.
Thanks to \eqref{hm:s.star_ce0}, we have
\begin{align}
    \frac{T}{n} \sumj \p_\theta^k \mathsf{g}^\star_{k,j}(\theta;\lam)[\mu_{j-1}]
    &= \frac{T}{n} \sumj \left( \p_\theta^k \mathsf{g}^\star_{k,j}(\theta;\lam) - \E^{j-1}\left[\p_\theta^k \mathsf{g}^\star_{k,j}(\theta;\lam)\right]\right)[\mu_{j-1}]
\end{align}
for $k=0,1$. 
For each $\theta$, the Burkholder inequality for martingale difference arrays ensures that the right-hand side is $O_p(n^{-1/2})$. An application of \eqref{hm:sobolev.ineq} with \eqref{hm:s.star_ce0} then concludes that the stochastic order is valid uniformly in $(\theta,\lam)$, resulting in
\begin{equation}\label{hm:del.sig-1}
    \sup_{\theta,\lam}\left| \frac{T}{n} \sumj \mathsf{g}^\star_{k,j}(\theta;\lam)[\mu_{j-1}] \right| = o_p(1).
\end{equation}

In view of \eqref{hm:del.sig-0}, we are left to proving
\begin{equation}\label{hm:del.sig-2}
    \sup_{\theta,\lam}\left| \frac{T}{n} \sumj \mathsf{g}^\star_{k,j}(\theta;\lam)[\del_{\sig,1,j}] \right| = o_p(1).
\end{equation}
To this end, we need a finer expression of $\del_{\sig,1,j}$ (recall the definition \eqref{hm:def_del.sig.1}):
\begin{align}
    \del_{\sig,1,j} 
    &= \sqrt{h} \times \frac{1}{\sqrt{h}}\int_j 
    \p_x\sig^\star_{j-1} \left[\frac1h \int_{t_{j-1}}^s \mu'_u du \right] dw_s
    \nn\\
    &{}\qquad + \frac{1}{\sqrt{h}}\int_j \p_x\sig^\star_{j-1}
    \left[\frac{1}{\sqrt{h}} \int_{t_{j-1}}^s (\sig'_u - \sig'_{j-1}) dw'_u \right] dw_s
    \nn\\
    &{}\qquad + \frac{1}{\sqrt{h}}\int_j \p_x\sig^\star_{j-1}
    \left[ \sig'_{j-1} \frac{1}{\sqrt{h}}(w'_s - w'_{t_{j-1}}) \right] dw_s
    \nn\\
    &=: \del_{\sig,1,j}^{(1)} + \del_{\sig,1,j}^{(2)}+ \del_{\sig,1,j}^{(3)}.
\end{align}
Here again, we may and do suppose that all the involved random variables are one-dimensional (both $w$ and $w^\dagger$ are one-dimensional). 

Obviously, \eqref{hm:zeta-A3} gives
\begin{align}
    & \sup_{\theta,\lam}\left| \frac{T}{n} \sumj \mathsf{g}^\star_{k,j}(\theta;\lam)[\del_{\sig,1,j}^{(1)}] \right| 
    \nn\\
    &\le \frac{T}{n} \sumj \sup_{\theta,\lam}\left|\mathsf{g}^\star_{k,j}(\theta;\lam)\right||\del_{\sig,1,j}^{(1)}|
    = O_p(\sqrt{h}) = o_p(1).
    \label{hm:del1-esti}
\end{align}
Next, through the Cauchy-Schwarz inequality as in \eqref{hm:CS.apply.ex} we get
\begin{align}
    & \sup_{\theta,\lam}\left| \frac{T}{n} \sumj \mathsf{g}^\star_{k,j}(\theta;\lam)[\del_{\sig,1,j}^{(2)}] \right|^2 
    \nn\\
    & \lesssim 
    \left(\frac{1}{n} \sumj |\p_x\sig^\star_{j-1}|^2 \sup_{\theta,\lam}|\mathsf{g}^\star_{k,j}(\theta;\lam)|^2 \right)
    \nn\\
    &{}\qquad 
    \times \left\{\frac{1}{n} \sumj \left|
    \frac{1}{\sqrt{h}}\int_j 
    \left[\frac{1}{\sqrt{h}} \int_{t_{j-1}}^s (\sig'_u - \sig'_{j-1}) dw'_u \right] dw_s
    \right|^2\right\}
    \label{hm:del1-esti+1}
\end{align}
The first part $(\dots)$ in the above display is $O_p(1)$. 
By applying the Burkholder inequality twice, the expectation of the second part $\{\dots\}$ can be bounded by a constant multiple of
\begin{align}
    & \frac1n\sumj \frac1h \int_j
    \E\left[ \left| 
    \frac{1}{\sqrt{h}} \int_{t_{j-1}}^s (\sig'_u - \sig'_{j-1}) dw'_u 
    \right|^2\right] ds
    \nn\\
    &\lesssim 
    \frac1n\sumj \frac1h \int_j 
    \frac1h \int_{t_{j-1}}^s \E\left[|\sig'_u - \sig'_{j-1}|^2\right] du ds = o(1),
\end{align}
where the last step is due to Assumption \ref{hm:A_drif.coeff&X}.
Thus, the left-hand side of \eqref{hm:del1-esti+1} turns out to be $o_p(1)$, 
and \eqref{hm:del.sig-2} can be concluded if we show
\begin{equation}\label{hm:del.sig-2-3}
    \sup_{\theta,\lam}\left| \frac{T}{n} \sumj \mathsf{g}^\star_{k,j}(\theta;\lam)[\del^{(3)}_{\sig,1,j}] \right| = O_p\left(\frac{1}{\sqrt{n}}\right).
\end{equation}
The term inside the absolute value sign is a functional of $\lx_{t_{j-1}}$ and the increments $\{w'_t-w'_s:\, t,s\in\mbbi_j\}$. 
Thanks to the definition \eqref{hm:def_s.star}, the condition \eqref{hm:zeta-A4}, the symmetry $\mcl(w')=\mcl(-w')$, and the self-renewing property of $w'$ (see \cite[Theorem I.32]{Pro04}), 
we have $\E^{j-1}[\mathsf{g}^\star_{k,j}(\theta;\lam)[\del^{(3)}_{\sig,1,j}]]=-\E^{j-1}[\mathsf{g}^\star_{k,j}(\theta;\lam)[\del^{(3)}_{\sig,1,j}]]$ a.s., hence
\begin{equation}
    \E^{j-1}\left[
    \mathsf{g}^\star_{k,j}(\theta;\lam)[\del^{(3)}_{\sig,1,j}]
    \right]
    = 0 \qquad \text{a.s.}
\end{equation}
This, combined with the Sobolev inequality argument used in proving \eqref{hm:del.sig-1} gives \eqref{hm:del.sig-2-3}.

Combining \eqref{hm:del.sig-0}, \eqref{hm:del.sig-1}, and \eqref{hm:del.sig-2} now yields
\begin{equation}
    \sup_{\theta,\lam} \left| \overline{U}^{\ve,2(1)}_{1,k,n}(\theta;\lam)\right| = o_p(1),
\end{equation}
followed by the desired \eqref{hm:disconti.removed-part.2}. 

\subsection{Basic limit theorems}
\label{hm:sec_lim.thm}

In this section, we will present limit theorems for the ``leading'' term
\begin{equation}\label{hm:U.star_notation}
    U^\star_n(\theta;\lam) = \sumj \zeta^\star_j(\theta;\lam) 
    = \sumj \zeta^\star_{j-1}(\sig^\star_{j-1}Z_j,\theta;\lam),
\end{equation}
which will give rise to non-trivial limits of $\p_\theta^k U_n(\theta;\lam)$ after suitably normalized.
Due to the i.i.d. Gaussian nature of $(Z_j)$, we can proceed free from discontinuous variations caused by jumps and spikes.

We will write
\begin{align}
    & \text{$\xi_n(\theta;\lam)=O_{u,p}(a_n)$ \quad if \quad $\sup_{\lam,\theta} |a_n^{-1} \xi_n(\theta;\lam)| = O_p(1)$},
    \nn\\
    & \text{$\xi_n(\theta;\lam)=o_{u,p}(a_n)$ \quad if \quad $\sup_{\lam,\theta} |a_n^{-1} \xi_n(\theta;\lam)| = o_p(1)$}
\end{align}
for any sequence of random functions $\{\xi_n(\theta;\lam)\}_n$ and any positive sequence $(a_n)_n$.

\subsubsection{Approximation of Riemann integral}

By \eqref{hm:disconti.removed}, we have
\begin{equation}\label{hm:lim.thm-1}
\left|\p_\theta^k \left(\frac1n U_n(\theta;\lam) - \frac1n U^\star_n(\theta;\lam)\right) \right| = o_p\left(\frac{1}{\sqrt{n}}\right) 
\end{equation}
for $k=0,1,2,3$ in either case of Assumption \ref{hm:A_lam} (about the behavior of $\lam=\lam_n$).


Recall the constant $c'>0$ given in \eqref{hm:A_J-2} and the notation rule \eqref{hm:gen.notation_f}.

\begin{lem}
    \label{hm:lim.thm_lem-1}
    We have
    \begin{equation}
        \left|\frac1n U^\star_n(\theta;\lam) - \frac1T \int_0^T \int \zeta\left(\lx_t,\sig^\star_t z,\theta;\lam\right) \phi_r(z)dz dt\right| 
        = O_{u,p}\big(h^{(1\wedge c')/2)}\big).
    \end{equation}
\end{lem}

\begin{proof}
By compensating each summand $\zeta^\star_j(\theta;\lam)$ and then applying the Burkholder and Sobolev inequalities (as in Section \ref{hm:sec_disc.remove-subsec.2}), we get
\begin{align}
    \frac1n U^\star_n(\theta;\lam) = O_{u,p}(n^{-1/2}) 
    + \frac1n \sumj \int \zeta^\star_{j-1}(\sig^\star_{j-1} z,\theta;\lam)\phi_r(z)dz.
\end{align}
Observe that
\begin{align}
    & \sup_\theta \left| \frac1n \sumj \int \zeta^\star_{j-1}(\sig^\star_{j-1}z,\theta;\lam)\phi_r(z)dz - \frac1T \int_0^T \int \zeta\left(\lx_t,\sig^\star_t z,\theta;\lam\right) \phi_r(z)dz dt\right|
    \nn\\
    &\lesssim \frac1n \sumj \frac1h \int_j \int 
    \sup_\theta \left| \zeta\left(\lx_t,\sig^\star_t z,\theta;\lam\right) - \zeta^\star_{j-1}(\sig^\star_{j-1}z,\theta;\lam)\right| \phi_r(z)dz dt
    \nn\\
    &\lesssim \frac1n \sumj \frac1h \int_j \int 
    (1+|\lx_{t_{j-1}}|^C + |\lx_{t}|^C)(1+|\sig^\star_t|^C |z|^C)
    \nn\\
    &{}\qquad \times ( |\lx_t-\lx_{t_{j-1}}| + |\sig^\star_t-\sig^\star_{j-1}||z| )\phi_r(z)dz dt
    \nn\\
    &\lesssim 
    \frac1n \sumj \frac1h \int_j (1+|\lx_{t_{j-1}}|^C + |\lx_{t}|^C) \, \left( h^{-(1\wedge c')/2}|\lx_t - \lx_{t_{j-1}}|\right) dt \times h^{(1\wedge c')/2}
    \nn\\
    &= O_{u,p}(h^{(1\wedge c')/2}).
\end{align}
Here, we used the latter part of \eqref{hm:X-moment.estimates} in the last step.
The proof is complete.
\end{proof}

The following corollary is immediate from the proof of Lemma \ref{hm:lim.thm_lem-1}.

\begin{cor}\label{hm:lim.thm_lem1-cor1}
    For any measurable function $\mathsf{g}:\,\mbbr^{d'}\times\Theta\times(0,\overline{\lam}]$ such that
    \begin{equation}
        \sup_\theta |\p_x\mathsf{g}(x,\theta;\lam)| \lesssim 1+|x|^C,
    \end{equation}
    we have
    \begin{equation}
        \left| \frac1n \sumj \mathsf{g}(\lx_{t_{j-1}},\theta;\lam) - \frac1T \int_0^T \mathsf{g}^\star(\lx_t,\theta;\lam) dt\right| 
        = O_{u,p}\big(h^{(1\wedge c')/2)}\big).
    \end{equation}
\end{cor}

We can deduce the following Lemma \ref{hm:lim.thm_lem-2} similarly to Lemma \ref{hm:lim.thm_lem-1}, and also Lemma \ref{hm:lim.thm_lem-3} directly. 
    
\begin{lem}
    \label{hm:lim.thm_lem-2}
    We have
    \begin{equation}
        \left|- \frac1n \p_\theta^2 U^\star_n(\tz;\lam) + \frac1T \int_0^T \int \p_\theta^2 \zeta\left(\lx_t,\sig^\star_t z,\tz;\lam\right) \phi_r(z)dz dt\right| 
        = O_{u,p}\big(h^{(1\wedge c')/2)}\big).
    \end{equation}
\end{lem}

\begin{lem}
    \label{hm:lim.thm_lem-3}
    We have
    \begin{equation}
        \left|\frac1n \p_\theta^3 U^\star_n(\tz;\lam)\right| = O_{u,p}(1).
    \end{equation}
\end{lem}

\subsubsection{Asymptotic mixed normality}

By \eqref{hm:disconti.removed},
\begin{equation}
    \sup_\lam \left|  \frac{1}{\sqrt{n}} \p_\theta U_n(\lam) - \frac{1}{\sqrt{n}} \p_\theta U^\star_n(\lam)\right| \cip 0.
\end{equation}
We recall that $\overline{U}'_n := n^{-1/2}\p_\theta U^\star_n(\lam)$ is said to converge ($\mcf$-)stably in law to $\overline{U}'_0$, denoted by $\overline{U}'_n \scl \overline{U}'_0$, if $(\overline{U}'_n, G_{n}) \cil (\overline{U}'_0, G)$ for every $\mcf$-measurable random variables $G_{n}$ and $G$ such that $G_{n}\cip G$, where $\overline{U}'_0$ is defined on an extended probability space of the original one; it is the tailor-made mode of convergence for verifying \eqref{hm:joint.conv}. 
See \cite{Jac97} for a general result of the stable convergence in law designed for the present high-frequency sampling.

Let
\begin{equation}\label{hm:def_Sigma}
    \Sig_{0,t}(\lam) := \frac1T \int_0^t \int \left( \p_\theta \zeta\left(\lx_s,{S^\star_s}^{1/2} z,\tz;\lam\right) \right)^{\otimes 2} \phi(z)dz ds, \quad t\in[0,T].
\end{equation}
The objective here is to prove the stable convergence in law of this process to the centered mixed-normal ($\mcf$-conditionally Gaussian) distribution with possibly random asymptotic covariance $\Sig_0(\lambda):=\Sig_{0,T}(\lambda)$.
The statement is given as follows.

\begin{lem}\label{hm:common.stable.CLT}
    \hmrev{Assume additionally that $\E^{j-1}[\p_\theta\zeta^\star_j(\tz;\lam)]=0$ a.s. for any $\lam>0$.} 
    Then, we have
    \begin{equation}
    \overline{U}'_n \scl MN_p\left(0,\Sig_0(\lambda)\right).
    \nonumber
    \end{equation}
\end{lem}

\begin{proof}
We apply the criterion in \cite{Jac97} for the random {\cadlag} step process
\begin{equation}
t \mapsto \sum_{j=1}^{[nt/T]} \chi_j(\lambda),\qquad t\in[0,T],
\nonumber
\end{equation}
where $\chi_j(\lam) := \p_\theta\zeta^\star_j(\tz;\lam)$, that is, $n^{-1/2}\p_\theta U^\star_n(\lam) = \sumj \chi_j(\lam)$.
We set the reference continuous martingale $M$ in \cite{Jac97} to be $w'$. Fix a $u\in\mbbr^p$ in the rest of this proof. 

\hmrev{Under the assumption, we have $\E^{j-1}[\chi_j(\lambda)]=0$ a.s. for any $\lam>0$.} 
In view of \cite[Theorem 3-2]{Jac97} and the Cram\'{e}r-Wold device, it suffices to verify the following convergences for each $t\in[0,T]$:
\begin{align}
& \sum_{j=1}^{[nt/T]} \E^{j-1}\left[\left|\chi_j(\lambda)[u]\right|^4\right] \cip 0,
\label{hm:score-p1}\\
& \sum_{j=1}^{[nt/T]} \E^{j-1}\left[\chi_j(\lam)^{\otimes 2}[u^{\otimes 2}] \right] \cip \Sig_{0,t}(\lambda)[u^{\otimes 2}],
\label{hm:score-p2}\\
& \left|\sum_{j=1}^{[nt/T]} \E^{j-1}\left[\chi_j(\lambda) [u]\, \D_j w'\right]\right|
+ \left|\sum_{j=1}^{[nt/T]} \E^{j-1}\left[\chi_j(\lambda) [u]\, \D_j N\right]\right| \cip 0,
\label{hm:score-p3}
\end{align}
where the asymptotic orthogonality condition \eqref{hm:score-p3}, which is essential to ensure the stability of convergence, has to hold for any bounded $(\mcf_t)$-martingale $N$ orthogonal to $w'$ (namely, the quadratic-variation process $[N,w']_{\cdot}$ is identically zero).

By \eqref{hm:zeta-A3}, we have
\begin{equation}
    \sup_{\theta}\left| \p_\theta \zeta(x,\sig(x;\tz)z,\theta;\lam)\right| \lesssim (1+|x|^C)(1+|z|^C).
\end{equation}
Since $\chi_j(\lam)$ only contains $\lx_{t_{j-1}}$ and $Z_j \sim N_r(0,I_r)$, we 
easily get \eqref{hm:score-p1}:
\begin{equation}
    \E\left[\left|\sum_{j=1}^{[nt/T]} \E^{j-1}\left[\left|\chi_j(\lambda)[u]\right|^4\right] \right|\right]
    \lesssim \frac{|u|^4}{n} \frac1n \sumj \left(1+\E[|\lx_{t_{j-1}}|^C]\right)
    \lesssim \frac{|u|^4}{n} \cip 0.
\end{equation}

For the convergence \eqref{hm:score-p2}, through the compensation and Burkholder inequality as in the proof of Lemma \ref{hm:lim.thm_lem-1}, we can deduce
\begin{align}
    & \sum_{j=1}^{[nt/T]} \E^{j-1}\left[\chi_j(\lam)^{\otimes 2}[u^{\otimes 2}] \right]
    \nn\\
    &= \frac1T \int_0^t \int \left( \p_\theta \zeta\left(\lx_t,\sig^\star_t z,\tz;\lam\right) \right)^{\otimes 2} \phi_r(z)dz dt\,[u^{\otimes 2}] + O_{p}(h^{(1\wedge c')/2})
    \nn\\
    &= \frac1T \int_0^t \int \left( \p_\theta \zeta\big(\lx_t,{S^\star_t}^{1/2} z,\tz;\lam\big) \right)^{\otimes 2} \phi(z)dz dt\,[u^{\otimes 2}] + O_{p}(h^{(1\wedge c')/2})
    \nn\\
    &= \Sig_{0,t}(\lam)\,[u^{\otimes 2}] + O_{p}(h^{(1\wedge c')/2})
\end{align}
as was desired.

For \eqref{hm:score-p3}, we note the following identities (a.s.):
\begin{align}
    \E^{j-1}\left[\chi_j(\lambda) [u]\, \D_j w'\right]=0,
    \qquad 
    \E^{j-1}\left[\chi_j(\lambda) [u]\, \D_j N\right]=0.
\end{align}
The first one is obvious due to \eqref{hm:zeta-A4}. 
As for the second one, supposing that $p=r=1$ without loss of generality, we note that $\chi_j(\lam)$ is measurable with respect to the filtration generated by $\lx_{t_{j-1}}$ and the family of $w'$-increments $\{w'_t-w'_s\}_{s,t\in\mbbi_j}$. Then, we can apply the martingale representation theorem \cite[Theorem III.4.34]{JacShi03} 
with setting $X=w'$ and $\mathscr{H}=\mcf_{t_{j-1}}$ for the elements $X$ and $\mathscr{H}$ therein, to conclude that we can write
\begin{equation}
    \chi_j(\lam) = \int_j \psi_{u-} dw'_u
\end{equation}
for some process $\psi$ adapted to the filtration $\mcf^{\sharp}_{t}$ with $\mcf^{\sharp}_{t} := \mcf_{t_{j-1}} \cap \mcf^{w'}_t$. The orthogonality between $w'$ and $N$ gives
\begin{align}
    \E^{j-1}\left[\chi_j(\lambda) [u]\, \D_j N\right]
    &= \E^{j-1}\left[\int_j \psi_{u-} dw'_u \, \int_j dN_u\right] \nn\\
    &= \E^{j-1}\left[\int_j \psi_{u-} d[N,w']_u\right] = 0.
\end{align}
Hence, we get \eqref{hm:score-p3} and the proof is complete.
\end{proof}

\section{Proof of Theorem \ref{hm:thm_AMN}}
\label{hm:sec_main.results.proof}

In this section, building on what we have proved in Section \ref{hm:sec_aux.asymp}, we will complete the proof of Theorem \ref{hm:thm_AMN}.
We will keep using the generic notation introduced in Section \ref{hm:sec_aux.asymp} such as \eqref{hm:gen.notation_f}.

\subsection{Introductory remarks}

First, we describe the outline that will commonly appear in the density-power and H\"{o}lder-based GQLFs.

Recall the notation \eqref{hm:def_Del.n} and \eqref{hm:def_Gam.n}, both of which will be used for both density-power and H\"{o}lder-based GQLFs. Further, we will write
\begin{align}
\mbby_{n}(\theta;\lambda) &:= \frac{1}{n}\left(\mbbh_{n}(\theta;\lambda) - \mbbh_{n}(\tz;\lambda)\right).
\nonumber
\end{align}
We consider the following conditions for $\lam_n\equiv\lam>0$.
\begin{itemize}
\item 
There exist a constant $\ep_0>0$ and a random function $\mbby_0(\cdot;\lam):\,\overline{\Theta}\to\mbbr$ such that
\begin{equation}
\sup_{\theta} \left|n^{\ep_0}\left(\mbby_{n}(\theta;\lambda)-\mbby_{0}(\theta;\lambda)\right)\right| = O_p(1),
\label{hm:consistency-1}
\end{equation}
and that there exists an a.s. positive random variable $\chi_{0}(\lambda)$ for which the following identifiability condition holds:
\begin{align}
& \forall \theta\in\Theta,\quad \mbby_{0}(\theta;\lambda) \le -\chi_{0}(\lam)|\theta-\tz|^{2}.
\label{hm:consistency-2}
\end{align}

\item 
There exist random positive definite random matrices $\Sig_{0}(\lam)$, $\Gam_0(\lambda)\in\mbbr^{p}\otimes\mbbr^{p}$ such that
\begin{equation}\label{hm:joint.conv}
\left(\Delta_{n}(\lambda),\, \Gam_{n}(\lambda)\right) \cil \big( \Sig_{0}(\lam)^{1/2}\eta,\, \Gam_{0}(\lam)\big),
\end{equation}
where $\eta \sim N_{p}(0,I_{p})$ independent of $\mcf$, defined on an extended probability space $(\overline{\Omega},\overline{\mcf},\overline{\pr})$.

\item 
We have
\begin{equation}\label{hm:p3H}
\sup_{\theta,\lam}\left| \frac1n \p_{\theta}^{3}\mbbh_{n}(\theta;\lambda)\right| = O_{p}(1).
\end{equation}

\end{itemize}
We will adopt the above conditions in either case of Assumption \ref{hm:A_lam}:
when $\lam=\lam_n \to 0$ ($n\to\infty$), we regard \eqref{hm:consistency-1}, \eqref{hm:consistency-2}, and \eqref{hm:joint.conv} as those with $\mbby_0(\theta;\lam)$ replaced by $\lim_{\lam \to 0}\mbby_0(\theta;\lam)$, $\Sig_0(\lam)$ by $\lim_{\lam \to 0}\Sig_0(\lam)$, and $\Gam_0(\lam)$ by $\lim_{\lam \to 0}\Gam_0(\lam)$, respectively, all taken in the a.s sense.

\medskip

By the standard $M$-estimation argument \cite{Yos11}, the consistency $\tes(\lambda) \cip \tz\in\Theta$ follows from \eqref{hm:consistency-1} and \eqref{hm:consistency-2}. 
We have $\p_\theta\mbbh_n(\tes(\lambda);\lambda)=0$ on the event $\{\tes(\lambda)\in\Theta\}$ whose probability tends to $1$.
Then, the joint convergence \eqref{hm:joint.conv} and the tightness \eqref{hm:p3H} combined with the second-order Taylor expansion
\begin{align}
& \left(\Gam_n(\lambda) - 
\iint_{(0,1)^2}\frac{s}{n}\p_\theta^3\mbbh_n\big(\tz+ss'(\tes(\lambda)-\tz);\lambda\big)dsds' \, [\tes(\lam)-\tz]
\right)[\hat{u}_n(\lambda)]
\nn\\
&= \D_n(\lambda) - n^{-1/2}\p_\theta\mbbh_n(\tes(\lambda);\lambda)
\nonumber
\end{align}
give the asymptotic mixed normality \eqref{hm:AMN_pre} of the scaled estimator $\hat{u}_n(\lambda)$ with $V_0=\Gam_0(\lam)^{-1}\Sig_0(\lam)\Gam_0(\lam)^{-1}$:
\begin{itemize}
    \item When $\lam_n\equiv\lam>0$,
    \begin{align}
    \hat{u}_n(\lambda) &= \Gam_0(\lambda)^{-1}\D_n(\lambda) + o_p(1) 
    \nn\\
    &\cil MN_p\left(0,\, \Gam_0(\lam)^{-1}\Sig_0(\lam)\Gam_0(\lam)^{-1}\right).
    \nn
    \end{align}
    
    \item When $\lam \to 0$, we will have $\lim_{\lam \to 0}\Gam_0(\lam) = \lim_{\lam \to 0}\Sig_0(\lam) = \mci(\tz)$ a.s. with $\mci(\tz)=(\mci_{kl}(\tz))_{k,l}$ given by \eqref{hm:def_FI.lim}.
\end{itemize}
In the latter case, the asymptotic distribution of $\hat{u}_n(\lam)$ becomes $MN_p\left(0,\, \mci(\tz)^{-1}\right)$, meaning that a suitable control $\lam_n\to 0$ enables us to estimate $\tz$ asymptotically efficiently as if we observed a non-contaminated continuous process $(\check{X},\check{Y})$ without jumps and spike noises.

\medskip

The tightness \eqref{hm:p3H} is automatic from Lemma \ref{hm:lim.thm_lem-3}. Hence, we only need to verify \eqref{hm:consistency-1}, \eqref{hm:consistency-2}, and \eqref{hm:joint.conv}.

\subsection{Divergence: proof of consistency}

The uniform convergence \eqref{hm:consistency-1} directly follows from \eqref{hm:lim.thm-1} and Lemma \ref{hm:lim.thm_lem-1}. To conclude the consistency, we only need to verify \eqref{hm:consistency-2} by specifying the limit $\mbby_0(\theta;\lam)$.

Let ($\int = \int_{\mbbr^d}$)
\begin{align}
    \mbby_0(\theta;\lam) 
    &:= \frac1T \int_0^T \int_{\mbbr^r} \zeta\left(\lx_t,\sig^\star_t z,\theta;\lam\right) \phi_r(z)dz dt
    \nn\\
    &{}\qquad 
    - \frac1T \int_0^T \int_{\mbbr^r} \zeta\left(\lx_t,\sig^\star_t z,\tz;\lam\right) \phi_r(z)dz dt
    \nn\\
    &= \frac1T \int_0^T \int \zeta\left(\lx_t,S^{\star\,1/2}_t z,\theta;\lam\right) \phi(z)dz dt
    \nn\\
    &{}\qquad 
    - \frac1T \int_0^T \int \zeta\left(\lx_t,S^{\star\,1/2}_t z,\tz;\lam\right) \phi(z)dz dt
    \label{hm:Y0_def}
\end{align}
for both the density-power and H\"{o}lder-based divergences, where we used the following basic change of variables for the second equality:
\begin{equation}\label{hm:gauss_c.o.v.}
    \int_{\mbbr^r} \psi(x,\sig z)\phi_r(z)dz = \int_{\mbbr^d} \psi(x,S^{1/2}z)\phi(z)dz
\end{equation}
for a measurable function $\psi$ on $\mbbr^{d'} \times \mbbr^d$ and a positive-definite constant matrix $\sig\in\mbbr^d\otimes\mbbr^r$ with $S:=\sig^{\otimes 2}$.
This $\mbby_0(\theta;\lam)$ corresponds to the robustified version of the Gaussian quasi-Kullback-Leibler-divergence, the random function $\mbby(\theta)$ given in \cite[p.2857]{UchYos13}.

We will apply the following lemma for verifying \eqref{hm:consistency-2}.

\begin{lem}\label{hm:lem_consistency}
Let $\Theta\subset\mbbr^{p}$ be a bounded convex domain, and suppose that a random function $\mathsf{Y}:\,\overline{\Theta} \to \mbbr$ a.s. fulfils the following:
it belongs to the class $\mcc(\overline{\Theta})\cap\mcc^{2}(\Theta)$; $\mathsf{Y}(\theta)\le 0$ and $\mathsf{Y}$ takes its maximum $0$ only for $\theta=\tz\in\Theta$; and finally, $-\p_{\theta}^{2}\mathsf{Y}(\tz)$ is positive definite.
Then, we have
\begin{equation}
\chi_0=\chi_0(\tz) := \inf_{\theta:\,\theta\ne\tz}\frac{-\mathsf{Y}(\theta)}{|\theta-\tz|^2} > 0 \qquad \text{a.s.}
\nonumber
\end{equation}
\end{lem}

\begin{proof}
Fix an arbitrary $\ep>0$. It suffices to show that $\pr[\chi_0 < \del] < \ep$ for some $\del>0$.
Let
\begin{equation}
G(r,\del') := \left\{ \inf_{|\theta-\tz| \le r} \lam_{\min}\left( -\p_{\theta}^{2}\mathsf{Y}(\theta)\right) < 2\del' \right\}.
\nonumber
\end{equation}
Then, we can pick constants $r,\del'>0$ for which $\pr[G(r,\del')]<\ep/2$, hence
\begin{equation}
\pr[\chi_0 < \del] \le \frac{\ep}{2} + \pr\left[\{\chi_0 < \del\} \cap G(r,\del')^c\right].
\nonumber
\end{equation}

We have
\begin{align}
\inf_{\theta:\,\theta\ne\tz}\frac{-\mathsf{Y}(\theta)}{|\theta-\tz|^2}
&= \min\left\{
\inf_{\theta:\,|\theta - \tz| \le r}\frac{-\mathsf{Y}(\theta)}{|\theta-\tz|^2},~
\inf_{\theta:\,|\theta - \tz| > r}\frac{-\mathsf{Y}(\theta)}{|\theta-\tz|^2}
\right\} \nn\\
&=: \min\left\{ A'(r),~A''(r)\right\}.
\nonumber
\end{align}
We look at the rightmost side on $\omega\in G(r,\del')^c$.
Since $\mathsf{Y}(\tz)=0$ and $\p_\theta\mathsf{Y}(\tz)=0$, expanding $\mathsf{Y}(\theta)$ around $\tz$ concludes that $A'(r) \ge \del'$ on $G(r,\del')^c$:
\begin{align}
A'(r) 
&= \inf_{\theta:\,|\theta - \tz| \le r}
|\theta-\tz|^{-2}\left( -\int_0^1\int_0^1 v \p_\theta^2 \mathsf{Y}(\tz+uv(\theta-\tz)) dudv \right) [(\theta-\tz)^{\otimes 2}]
\nn\\
&\ge \frac12 \inf_{\theta:\,|\theta - \tz| \le r} \lam_{\min}\left( -\p_{\theta}^{2}\mathsf{Y}(\theta) \right) \ge \del'.
\nonumber
\end{align}
The compactness of $\overline{\Theta}$ ensures that $\sup_{\theta}|\theta-\tz|\le \mathsf{c}$ for some $\mathsf{c}>0$. Hence,
\begin{equation}
A''(r) \ge \mathsf{c}^{-2} \inf_{\theta:\,|\theta - \tz| > r } \left(-\mathsf{Y}(\theta) \right).
\nonumber
\end{equation}
Pick any $\del''>0$ for which $\pr[\mathsf{c}^{-2} \inf_{\theta:\,|\theta - \tz| > r } \left(-\mathsf{Y}(\theta) \right) < \del''] < \ep/2$. 

Now, for any $\del<\min\{\del', \del''\}$,
\begin{align}
\pr[\chi_0 < \del] &\le \pr[G(r,\del')] 
+ \pr\left[ \mathsf{c}^{-2} \inf_{\theta:\,|\theta - \tz| > r } \left(-\mathsf{Y}(\theta) \right) < \del'' \right] \nn\\
&{}\qquad + \pr\left[ \min\{\del', \del''\} \le \chi_0 <\del \right]
\nn\\
&\le \ep/2 + \ep/2 + 0 = \ep.
\nonumber
\end{align}
This completes the proof.
\end{proof}

\subsubsection{Density-power GQLF}

Recall the definition of the density-power GQLF \eqref{hm:def_dp-H-0--}. 
Let
\begin{equation}
    V(x,\theta) := S(x,\tz)^{-1/2} S(x,\theta) S(x,\tz)^{-1/2}.
\end{equation}
Then, $V(x,\tz) \equiv I_d$.
Through \eqref{hm:disconti.removed}, Lemma \ref{hm:lim.thm_lem-1} and \eqref{hm:def_dp-H-0--}, 
we get \eqref{hm:consistency-1} with $\ep_0=1/2$ and
\begin{align}
    \mbby_0(\theta;\lam) 
    &= \frac1T \int_0^T \int \bigg\{ \mathsf{d}^\star_t(\theta)^{-\lam/2}
    \left( \frac{1}{\lam} \phi\left(S^\star_t(\theta)^{-1/2} S^{\star\,1/2}_t z\right)^\lam - K_{\lam,d} \right) \nn\\
    &{}\qquad - {\mathsf{d}^{\star}_t}^{-\lam/2}
    \left( \frac{1}{\lam} \phi(z)^\lam - K_{\lam,d} \right) \bigg\} \phi(z)dz dt
    \nn\\
    &= \frac1T \int_0^T {\mathsf{d}^{\star}_t}^{-\lam/2}
    \bigg\{ 
    \left(\frac{1}{\lam} \int \phi\left(z;0,V^\star_t(\theta)\right)^\lam \phi(z)dz 
    - \frac{1}{\lam} \int \phi(z)^{\lam+1}dz \right)
    \nn\\
    &{}\qquad - K_{\lam,d}\left(\det(V^\star_t(\theta))^{-\lam/2} - 1\right) 
    \bigg\} dt.
    \label{hm:dp.pr-c1}
\end{align}
By means of the identity \eqref{hm:phi.power-identity}, we have
\begin{align}
    K_{\lam,d}\left(\det(V^\star_t(\theta))^{-\lam/2} - 1\right) 
    &= K_{\lam,d} \det(V^\star_t(\theta))^{-\lam/2} - K_{\lam,d} \nn\\
    &= \frac{1}{\lam+1} \int\phi(z;0,V^\star_t(\theta))^{\lam+1} dz - \frac{1}{\lam+1} \int\phi(z)^{\lam+1} dz.
    \nn
\end{align}
Substituting this into \eqref{hm:dp.pr-c1} and then continuing some calculations, we arrive at the key expression:
\begin{align}\label{hm:dp.pr-c2}
    \mbby_0(\theta;\lam) 
    &= -\frac1T \int_0^T (\lambda+1)^{-1} {\mathsf{d}^{\star}_t}^{-\lam/2} 
    \int \bigg\{ \phi(z;0,V^\star_t(\theta))^{\lam+1}
    \nn\\
    &{}\qquad     
    - \left(1+\frac{1}{\lambda}\right) \phi(z;0,V^\star_t(\theta))^{\lam} \phi(z) 
    + \frac{1}{\lambda} \phi(z)^{\lambda+1} \bigg\} dz dt.
\end{align}

Now, we can apply Lemma \ref{hm:lem_consistency} (with $\mathsf{Y}(\cdot) = \mbby_0(\cdot;\lam)$) to prove the consistency of the density-power GQMLE $\tes(\lam)$.
On the one hand, when $\lam_n\to 0$, the limit function $\lim_{\lam \to 0}\mbby_0(\cdot;\lam)$ becomes the usual quasi-Kullback-Leibler divergence, hence Lemma \ref{hm:lim.thm_lem-1} and Lemma \ref{hm:lem_consistency} under Assumption \ref{hm:A_iden} conclude the proof (see \cite{GenJac93}, and also \cite{UchYos13}).
On the other hand, when $\lam_n\equiv \lam>0$, we first note the basic property of the density-power divergence (\cite[Theorem 1]{BasHarHjoJon98}: we have the inequality $x^{\lambda+1}-(1+1/\lambda)x^{\lambda}+1/\lambda \ge 0$ for $x\ge 0$ with the equality holding if and only if $x=1$). This implies that the integral $\int \{\dots\}dz$ in \eqref{hm:dp.pr-c2} is strictly positive unless the densities $\phi(\cdot;0,V^\star_t(\theta))$ and $\phi(\cdot)$ are identical. The integral becomes zero if and only if $V^\star_{\cdot}(\theta)\equiv I_d$ identically a.s., which in turn holds if and only if $\theta=\tz$ under Assumption \ref{hm:A_iden}. 
Moreover, the positive definiteness of $-\p_\theta^2\mbby_0(\tz;\lam)$ is equivalent to that of $\Gam_0(\lam)$ of \eqref{hm:Gam_gen.form} below, which corresponds to \eqref{hm:Gam0_dp}; see Section \ref{hm:sec_Hes.mat}. 
With these observations, we can apply Lemma \ref{hm:lem_consistency} to verify \eqref{hm:consistency-2}, hence the consistency $\tes(\lambda) \cip \tz$.

\subsubsection{H\"{o}lder-based GQLF}

We keep using the notation \eqref{hm:Y0_def}.
Analogously to the derivation of \eqref{hm:dp.pr-c2}, for \eqref{hm:def_ldp-H-0--} we can deduce \eqref{hm:consistency-1} with $\ep_0=1/2$ and
\begin{align}
    \mbby_0(\theta;\lam) 
    &= \frac1T \int_0^T \bigg\{ \mathsf{d}^\star_t(\theta)^{-\frac{\lam}{2(\lam+1)}}
    \frac{1}{\lam} \phi\left(S^\star_t(\theta)^{-1/2} S^{\star\,1/2}_t z\right)^\lam
    \nn\\
    &{}\qquad - {\mathsf{d}^{\star}_t}^{-\frac{\lam}{2(\lam+1)}}
    \frac{1}{\lam} \phi(z)^\lam \bigg\} \phi(z)dz \,dt
    \nn\\
    &= -\frac1T \int_0^T \frac{1}{\lam} 
    \mathsf{d}^\star_t(\theta)^{\frac{\lam^2}{2(\lam+1)}}
    {\mathsf{d}^\star_t}^{-\frac{\lam}{2}} 
    \bigg(
    \mathsf{d}^\star_t(\theta)^{-\frac{\lam^2}{2(\lam+1)}}
    {\mathsf{d}^{\star}_t}^{\frac{\lam^2}{2(\lam+1)}} \int \phi(z)^{\lam+1} dz
    \nn\\
    &{}\qquad 
    - \int \phi(z;0,V^\star_t(\theta))^{\lam} \phi(z) dz \bigg)dt.
    \label{hm:dp.pr-c3}
\end{align}
By \eqref{hm:phi.power-identity}, we have
\begin{equation}
    \int \phi(z;0,V^\star_t(\theta))^{\lam+1} dz = 
    \det(V^\star_t(\theta))^{-\frac{\lam^2}{2(\lam+1)}}
    \left( \int \phi(z)^{\lam+1} dz \right)^{\frac{\lam}{\lam+1}}
\end{equation}
Applying this identity together with the H\"{o}lder inequality \eqref{hm:holder.ineq} with $g = \phi(\cdot)$ and $f = \phi(\cdot;0,V^\star_t(\theta))$, we get
\begin{align}
    \int \phi(z;0,V^\star_t(\theta))^{\lam} \phi(z) dz
    &\le \mathsf{d}^\star_t(\theta)^{-\frac{\lam^2}{2(\lam+1)}}
    {\mathsf{d}^\star_t}^{\frac{\lam^2}{2(\lam+1)}} \int \phi(z)^{\lam+1} dz.
\end{align}
This implies that $\mbby_0(\theta;\lam)$ of \eqref{hm:dp.pr-c3} is a.s. non-positive. The identity $\mbby_0(\theta;\lam)=0$ a.s. holds if and only if the two densities $\phi(\cdot;0,V^\star_t(\theta))$ and $\phi(\cdot)$ are equal, which holds in turn if and only if 
$V^\star_{\cdot}(\theta) = I_d$ 
identically. This concludes that, $\mbby_0(\theta;\lam) = 0$ a.s. if and only if $\theta=\tz$. 
As in the density-power case, we can deduce the positive definiteness of $-\p_\theta^2\mbby_0(\tz;\lam)=\Gam_0(\lam)$ of \eqref{hm:Gam_gen.form} below, which now corresponds to \eqref{hm:Gam0_ldp}. 
By Lemma \ref{hm:lem_consistency} we conclude \eqref{hm:consistency-2}, followed by the consistency $\tes(\lambda) \cip \tz$.

\subsection{Hessian matrix}
\label{hm:sec_Hes.mat}

As a part of deriving \eqref{hm:joint.conv}, we need to compute the limits of $\Gam_n(\lam)$. To that end, we prove the following lemma.

\begin{lem}
\label{hm:lem_Gauss.integrals}
Let $A_1,A_2\in\mbbr^d\otimes \mbbr^d$ be symmetric and positive definite non-random matrices.
Then, we have the following identities ($\phi(z)$ denotes the $d$-dimensional standard normal density):
\begin{align}
& \int \phi(z)^{\lam+1} A_1[z^{\otimes 2}] dz
= K_{\lam,d} \trace(A_1),
\label{hm:lem_Gauss.integrals-1} \\
& \int \phi(z)^{\lam+1} A_1[z^{\otimes 2}] A_2[z^{\otimes 2}] dz
\nn\\
&{}\qquad =\frac{K_{\lam,d}}{\lam+1}
\left\{ \trace(A_1)\trace(A_2) + 2\trace(A_1 A_2) \right\}.
\label{hm:lem_Gauss.integrals-2}
\end{align}
\end{lem}

\begin{proof}
By change of variables,
\begin{equation}
    \int \phi(z)^{\lam+1} f(z) dz = 
    (\lam+1)K_{\lam,d} \int f\left(\frac{y}{\sqrt{\lam+1}}\right) \phi(y)dy
\end{equation}
for any measurable function $f$ for which the integrals are well-defined. 
Then, \eqref{hm:lem_Gauss.integrals-1} is trivial and \eqref{hm:lem_Gauss.integrals-2} follows on applying the following formula \cite[Theorem 4.2(i)]{MagNeu79}:
\begin{equation}
    \int A[y^{\otimes 2}] \, B[y^{\otimes 2}] \, \phi(y)dy = \trace(A) \trace(B) + 2 \trace(AB),
\end{equation}
valid for any $d\times d$-symmetric $A$ and $B$.
\end{proof}

Building on \eqref{hm:disconti.removed} and Lemma \ref{hm:lim.thm_lem-2}, straightforward (yet lengthy) computations will give the expressions of the limit in probability:
\begin{align}\label{hm:Gam_gen.form}
    \Gam_0(\lam) := -\frac1T \int_0^T \int_{\mbbr^r} \p_\theta^2 \zeta\left(\lx_t, \sig^{\star}_t z,\tz;\lam\right) \phi_r(z)dz dt
\end{align}
of the normalzied Hessian matrix $\Gam_n(\lam) = -n^{-1}\p_\theta^2 U^\star_n(\tz;\lam)$; of course, we have $\Gam_0(\lam) = -\p_\theta^2 \mbby_0(\tz;\lam)$. 
To derive the expression of $\Gam_0(\lam)$ for the density-power and H\"{o}lder-based GQLFs, it is convenient to introduce some notational abbreviations:
\begin{align}
& \mathsf{t}_k = \trace(S^{-1}\dot{S}_k), \quad \mathsf{v}_{kl} = \trace(S^{-1}\dot{S}_k S^{-1}\dot{S}_l), \quad \mathsf{u}_{kl} = \trace(S^{-1}\ddot{S}_{kl}),   
\end{align}
where $S:=S^\star_t$, $\dot{S}_k:=\p_{\theta_k}S^\star_t$, and $\ddot{S}_{kl}:=\p_{\theta_k}\p_{\theta_l}S^\star_t$. 
Also defining $\beta$, $\gam$, $\dot{\beta}_k$, $\dot{\gam}_k$, and so on similarly, we get the expressions
\begin{align}
\dot{\beta}_k &= -\frac{\lam}{2} \beta \,\mathsf{t}_k, \qquad 
\ddot{\beta}_{kl} = \frac{\lam^2}{4} \beta \, \mathsf{t}_k \mathsf{t}_l - \frac{\lam}{2}\beta\,(\mathsf{u}_{kl} - \mathsf{v}_{kl}), \nn\\
\dot{\gam}_k &= -\frac{\lam}{2(\lam+1)} \gam \, \mathsf{t}_k, \qquad 
\ddot{\gam}_{kl} = \left(\frac{\lam}{2(\lam+1)}\right)^2 \gam \, \mathsf{t}_k \mathsf{t}_l - \frac{\lam}{2(\lam+1)} \gam\, (\mathsf{u}_{kl} - \mathsf{v}_{kl}).
\nn
\end{align}
Further, let $\mathsf{A}_k := \trace(S^{-1}\dot{S}_k S^{-1})$ and then for $z\in\mbbr^d$ and $\dot{\mathsf{A}}_{kl}:=\p_{\theta_l}\mathsf{A}_{k}$, we have
\begin{align}
    \dot{\mathsf{A}}_{kl} = -S^{-1} \dot{S}_l S^{-1} \dot{S}_k S^{-1} -S^{-1} \dot{S}_k S^{-1} \dot{S}_l S^{-1} + S^{-1} \ddot{S}_{kl} S^{-1}.
\end{align}

For the density-power GQLF, the second-order derivative of the summands of \eqref{hm:U.star_notation} takes the following form: 
letting $y:=\sig^\star_{j-1} z$,
\begin{align}
    \p_{\theta_k}\p_{\theta_l}\zeta^\star_{j-1}(\sig^\star_{j-1}z,\tz;\lam)
    &= \beta\left(\frac{\lam}{4} \, \mathsf{t}_k \mathsf{t}_l - \frac{1}{2}\,(\mathsf{u}_{kl} - \mathsf{v}_{kl})\right)(\vp_j(\theta)^\lam - \lam K_{\lam,d})
    \nn\\
    &{}\qquad -\frac{\lam}{4} \beta \,\mathsf{t}_k \mathsf{A}_l [y^{\otimes 2}] \vp_j(\theta)^\lam
    -\frac{\lam}{4} \beta \,\mathsf{t}_l \mathsf{A}_k [y^{\otimes 2}] \vp_j(\theta)^\lam
    \nn\\
    &{}\qquad + \frac{\lam}{4} \beta \,\mathsf{A}_k [y^{\otimes 2}] \mathsf{A}_l [y^{\otimes 2}] \vp_j(\theta)^\lam 
    + \frac{1}{2} \beta\, \dot{\mathsf{A}}_{kl} [y^{\otimes 2}] \vp_j(\theta)^\lam.
\end{align}
Thus, by Lemma \ref{hm:lim.thm_lem-2} we get the expression of $\Gam_{0,kl}(\lambda)$ given by \eqref{hm:Gam0_dp} as the limit in probability of the $(k,l)$-entry of $\Gam_n(\lam)$.
As for the H\"{o}lder-based GQLF, we have
\begin{align}
    & \p_{\theta_k}\p_{\theta_l}\zeta^\star_{j-1}(\sig^\star_{j-1}z,\tz;\lam)
    \nn\\
    &= \gam \left(\frac{\lam}{4(\lam+1)^2} \mathsf{t}_k \mathsf{t}_l
    -\frac{1}{2(\lam+1)} (\mathsf{u}_{kl} - \mathsf{v}_{kl})
    \right) \vp_j(\theta)^\lam 
    \nn\\
    &{}\qquad -\frac{\lam}{4(\lam+1)} \gam\, \mathsf{t}_k \mathsf{A}_l [y^{\otimes 2}] \vp_j(\theta)^\lam 
    -\frac{\lam}{4(\lam+1)} \gam \, \mathsf{t}_l \mathsf{A}_k [y^{\otimes 2}] \vp_j(\theta)^\lam 
    \nn\\
    &{}\qquad + \gam \left(
    \frac{\lam}{4} \mathsf{A}_k [y^{\otimes 2}] \mathsf{A}_l [y^{\otimes 2}] \vp_j(\theta)^\lam 
    + \frac12 \dot{\mathsf{A}}_{kl} [y^{\otimes 2}] \vp_j(\theta)^\lam
    \right).
\end{align}
Again, Lemma \ref{hm:lim.thm_lem-2} leads to the expression $\Gam_{0,kl}(\lambda)$ of \eqref{hm:Gam0_ldp}. Obviously, both $\Gam_0(\lambda)=(\Gam_{0,kl}(\lambda))_{k,l=1}^{p}$ corresponding to \eqref{hm:Gam0_dp} and \eqref{hm:Gam0_ldp} fulfills that
\begin{equation}
\lim_{\lambda\downarrow 0}\Gam_0(\lambda) = \mci(\tz)
\quad\text{a.s.}
\label{hm:Gam_lim}
\end{equation}
for $\mci(\tz)$ given in \eqref{hm:def_FI.lim}.

Although this limit is formally the same as in the case of continuous semimartingale regression \cite{UchYos13}, the explanatory process $\lx$ may contain jumps while all spikes are removed.

\subsection{Gradient}

To conclude \eqref{hm:joint.conv}, it suffices to show the stable convergence in law of $\D_n^\star(\lam)$, which is almost done by Lemma \ref{hm:common.stable.CLT}.
\hmrev{
Let us recall the shorthands $\zeta^\star_j(\theta;\lam)=\zeta(\lx_{t_{j-1}},\sig^\star_{j-1}Z_j,\theta;\lam)$, \eqref{hm:def_dp-H-0--}, \eqref{hm:def_ldp-H-0--}, and also the notation introduced in Section \ref{hm:sec_Hes.mat}. The additional condition ``$\E^{j-1}[\p_\theta\zeta^\star_j(\tz;\lam)]=0$ a.s. for any $\lam>0$'' therein can be shown as follows through the identity \eqref{hm:intro.ldp-4} and Lemma \ref{hm:lem_Gauss.integrals}.
\begin{itemize}
    \item For the density-power case,
    \begin{align}
        & \E^{j-1}[\p_{\theta_k}\zeta^\star_j(\tz;\lam)]
        \nn\\
        &= \E^{j-1}\left[\dot{\beta}_k \left(\frac{1}{\lam}\vp_j^\lam - K_{\lam,d}\right) + \frac12 \beta_k \vp_j^\lam (S^{-1}\dot{S}_k S^{-1}) [(\sig^\star_{j-1} Z_j)^{\otimes 2}]\right]
        \nn\\
        &= -\frac12 \beta_k \mathsf{t}_k \E^{j-1}[\vp_j^\lam] + \frac{\lam}{2} \beta_k K_{\lam,d}\,\mathsf{t}_k
        \nn\\
        &{}\qquad + \frac12 \beta_k \E^{j-1}\left[ \vp_j^\lam (S^{-1}\dot{S}_k S^{-1}) [(\sig^\star_{j-1} Z_j)^{\otimes 2}] \right] \nn\\
        &= -\frac12 \beta_k K_{\lam,d} (\lam+1) \mathsf{t}_k + \frac{1}{2}\beta_k K_{\lam,d} \lam \mathsf{t}_k + \frac{1}{2}\beta_k K_{\lam,d} \mathsf{t}_k = 0\quad \text{a.s.}
    \end{align}
    \item Similarly, as for the H\"{o}lder-based case,
    \begin{align}
        & \E^{j-1}[\p_{\theta_k}\zeta^\star_j(\tz;\lam)]
        \nn\\
        &= \E^{j-1}\left[ \dot{\gam}_k \frac{1}{\lam}\vp_j^\lam + \frac{1}{2} \gam_k \vp_j^\lam (S^{-1}\dot{S}_k S^{-1}) [(\sig^\star_{j-1} Z_j)^{\otimes 2}] \right]
        \nn\\
        &= -\frac12 \gam_k \mathsf{t}_k K_{\lam,d} + \frac12 \gam_k \mathsf{t}_k K_{\lam,d} = 0 \quad \text{a.s.}
    \end{align}
\end{itemize}
}
\noindent
Then, it remains to compute $\Sig_0(\lam)=\Sig_{0,T}(\lam)$ given by \eqref{hm:def_Sigma}; by the definition, it is obvious that $\Sig_0(\lam)$ is a.s. non-negative definite.
Direct calculations give \eqref{hm:Sig0_dp} in the density-power case, and \eqref{hm:Sig0_ldp} in the H\"{o}lder-based case. In either case, it is obvious that
\begin{equation}
    \lim_{\lambda\downarrow 0}\Sig_0(\lambda) = \mci(\tz)
    \quad\text{a.s.}    
\end{equation}

\section{Additional simulations}
\label{sec:add.sim}

\subsection{Results for Section \ref{se:simu22}} \label{suppl_1}
Here, we present the results of parameter estimation results for case (ii).

Figures \ref{plot33} and \ref{holplot33} give the histograms of $u_{1,n}(\lambda)$, $u_{2,n}(\lambda)$, and $u_{3,n}(\lambda)$ in the case of $\lambda=0.7$. 
Figure \ref{plot33} is based on the 1000th sample data and density-power GQMLE, and Figure \ref{holplot33} is based on the 1000th sample data and H\"{o}lder-based GQMLE.
Moreover, Tables \ref{ciratio32} and \ref{holciratio32} summarize the frequencies with which the true values of $\theta_{1}$, $\theta_{2}$, and $\theta_{3}$ are included in the 95\% confidence intervals in each $\lambda$ for density-power and H\"{o}lder-based GQMLEs, respectively.
From these figures and tables, we observe similar tendencies as in case (i).


\begin{figure}[ht]
\begin{tabular}{c}

\begin{minipage}{0.32 \hsize}
\begin{center}
\includegraphics[scale=0.25]{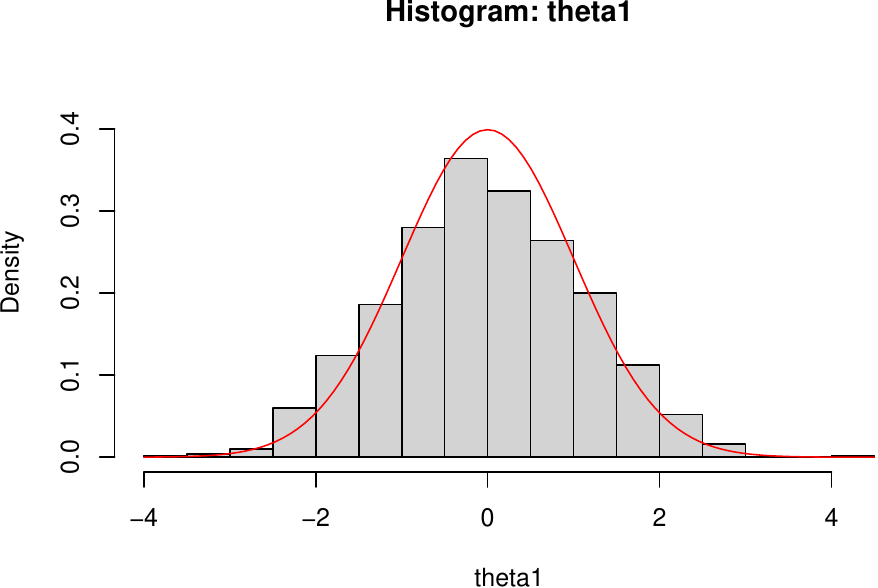}
\end{center}
\end{minipage}

\begin{minipage}{0.32 \hsize}
\begin{center}
\includegraphics[scale=0.25]{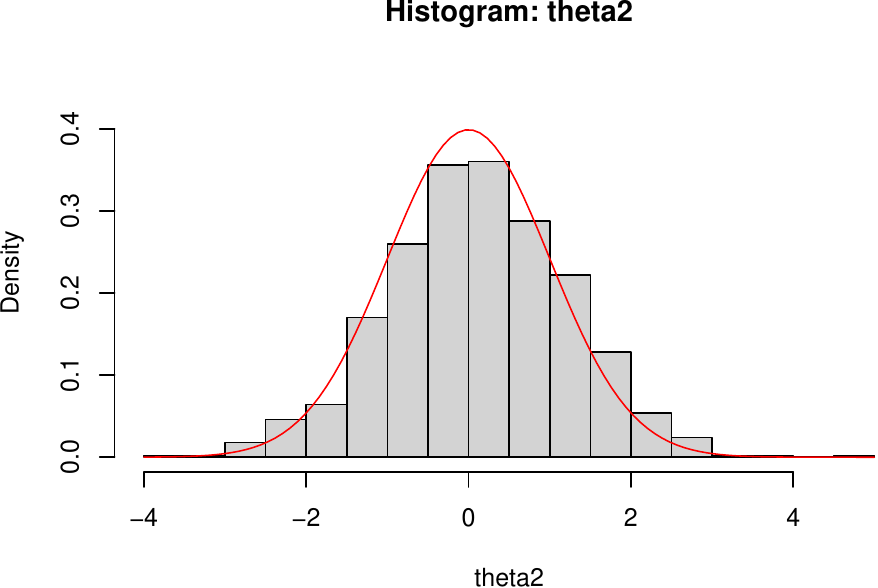}
\end{center}
\end{minipage}

\begin{minipage}{0.32 \hsize}
\begin{center}
\includegraphics[scale=0.25]{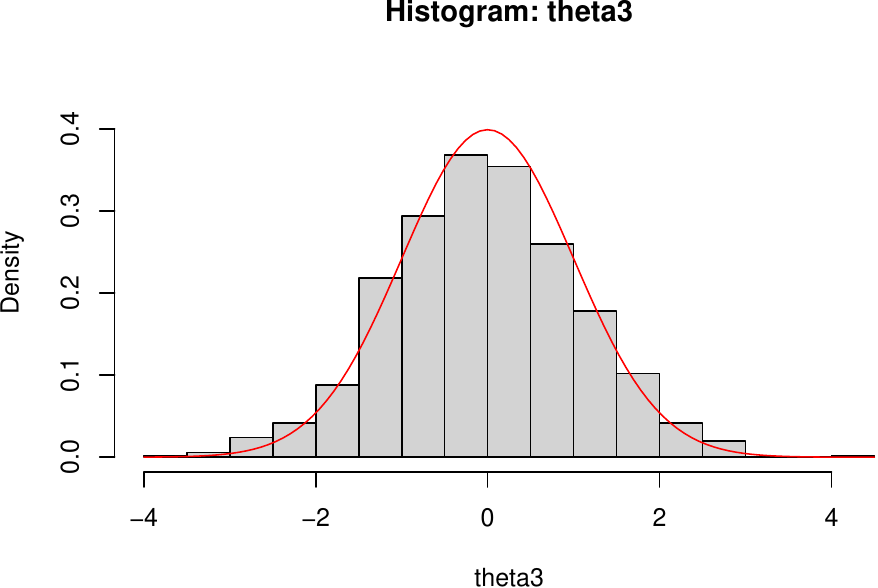}
\end{center}
\end{minipage}

\end{tabular}
\caption{Histograms of $u_{1,n}(\lambda)$, $u_{2,n}(\lambda)$, and $u_{3,n}(\lambda)$ corresponding to the density-power estimator ($q=0.01n$, $n=5000$, $\lambda=0.7$).}
\label{plot33}
\end{figure}


\begin{figure}[ht]
\begin{tabular}{c}

\begin{minipage}{0.32 \hsize}
\begin{center}
\includegraphics[scale=0.25]{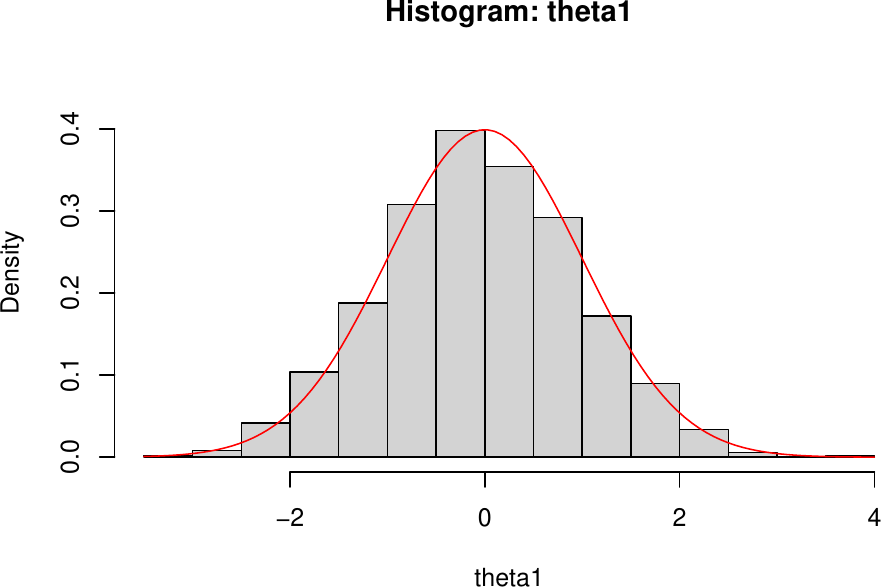}
\end{center}
\end{minipage}

\begin{minipage}{0.32 \hsize}
\begin{center}
\includegraphics[scale=0.25]{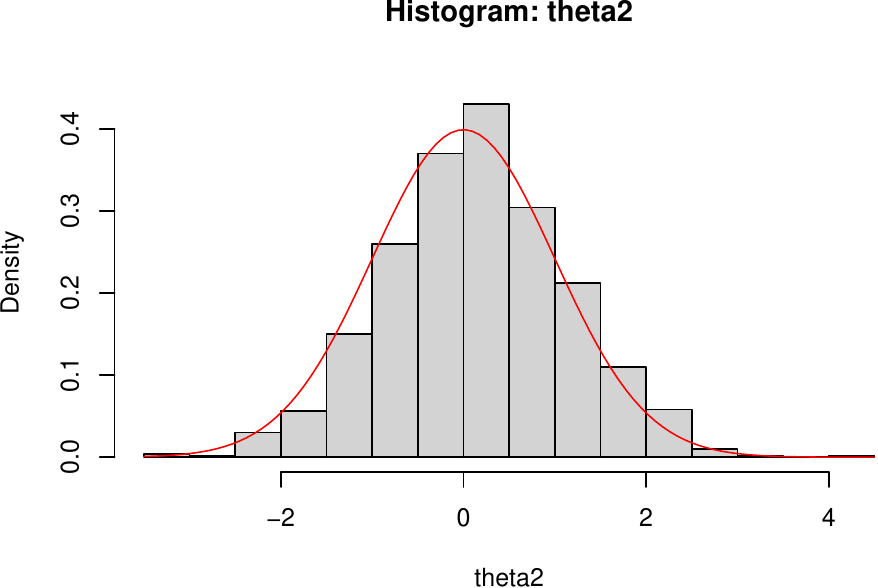}
\end{center}
\end{minipage}

\begin{minipage}{0.32 \hsize}
\begin{center}
\includegraphics[scale=0.25]{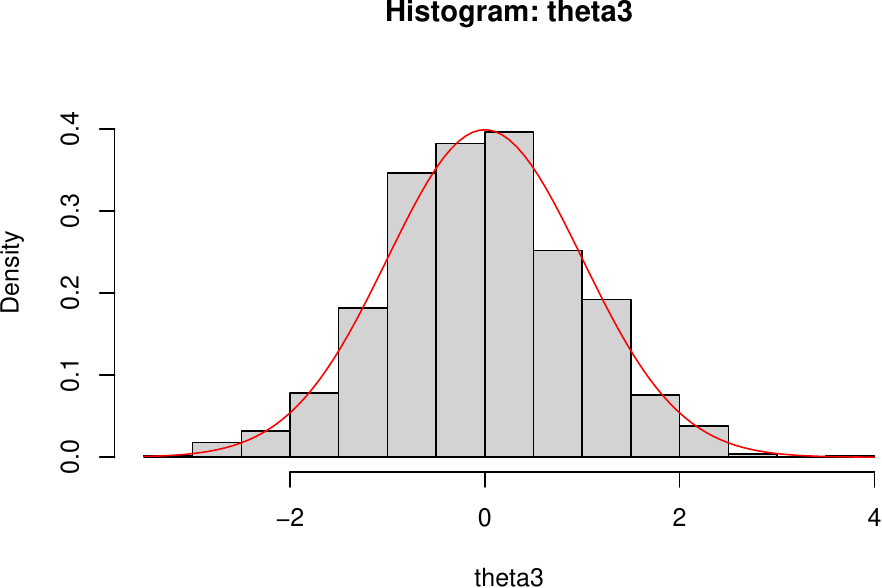}
\end{center}
\end{minipage}

\end{tabular}
\caption{Histograms of $u_{1,n}(\lambda)$, $u_{2,n}(\lambda)$, and $u_{3,n}(\lambda)$ corresponding to the H\"{o}lder-based estimator ($q=0.01n$, $n=5000$, $\lambda=0.7$).}
\label{holplot33}
\end{figure}


\begin{table}[ht]
\begin{center}
\caption{Frequency that the 95\% confidence interval corresponding to the density-power estimator contains the true value in each $\lambda$ ($q=0.01n$, $n=5000$).}
\begin{tabular}{r | r r r r r r r r r r} \hline
& \multicolumn{9}{c}{} \\[-3mm]
$\lambda$ & 0.1 & 0.2 & 0.3 & 0.4 & 0.5 & 0.6 & 0.7 & 0.8 & 0.9 & 1 \\[1mm] \hline
& \multicolumn{9}{c}{} \\[-3mm]
$\theta_{1}$ & 0.91 & 0.92 & 0.93 & 0.93 & 0.93 & 0.94 & 0.94 & 0.94 & 0.94 & 0.94 \\[1mm]
$\theta_{2}$ & 0.91 & 0.92 & 0.92 & 0.93 & 0.94 & 0.94 & 0.94 & 0.93 & 0.93 & 0.94 \\[1mm]
$\theta_{3}$ & 0.91 & 0.93 & 0.93 & 0.94 & 0.94 & 0.94 & 0.95 & 0.95 & 0.95 & 0.95 \\[1mm] \hline
\end{tabular}
\label{ciratio32}
\end{center}
\end{table}


\begin{table}[ht]
\begin{center}
\caption{Frequency that the 95\% confidence interval corresponding to the H\"{o}lder-based estimator contains the true value in each $\lambda$ ($q=0.01n$, $n=5000$).}
\begin{tabular}{r | r r r r r r r r r r} \hline
& \multicolumn{9}{c}{} \\[-3mm]
$\lambda$ & 0.1 & 0.2 & 0.3 & 0.4 & 0.5 & 0.6 & 0.7 & 0.8 & 0.9 & 1 \\[1mm] \hline
& \multicolumn{9}{c}{} \\[-3mm]
$\theta_{1}$ & 0.94 & 0.95 & 0.95 & 0.95 & 0.94 & 0.94 & 0.94 & 0.94 & 0.94 & 0.93 \\[1mm]
$\theta_{2}$ & 0.94 & 0.94 & 0.95 & 0.94 & 0.95 & 0.94 & 0.94 & 0.94 & 0.94 & 0.94 \\[1mm]
$\theta_{3}$ & 0.94 & 0.95 & 0.94 & 0.95 & 0.95 & 0.95 & 0.95 & 0.95 & 0.94 & 0.94 \\[1mm] \hline
\end{tabular}
\label{holciratio32}
\end{center}
\end{table}

\subsection{Results for Section \ref{se:simu6}} \label{suppl_2}
In Section \ref{se:simu6}, we propose the procedure for clustering into jump and continuous parts using one sample data set and the clustering simulations are conducted for the density-power GQMLE with $\lambda=0.1,0.2,0.5$, and $0.9$.
Since only the results for $\lambda=0.2$ are presented in Section \ref{se:simu6}, we provide below the results for $\lambda=0.1, 0.5$, and $0.9$.

Figures \ref{cluster3} and \ref{cluster5} show the clustering results for $\lambda=0.1, 0.5$, and $0.9$ using the data in Section \ref{se:simu12}.
The results in Figure \ref{cluster3} for $\lambda=0.1,0.5$, and $0.9$ are the same as those for $\lambda = 0.2$.
By comparing the data with the points included in $\mathfrak{D}_{n}$, the proportion of correctly identified noise points is $29/39$ for $\lambda=0.2$ and $28/39$ for $\lambda=0.5$ and $0.9$.

Figures \ref{cluster4} and \ref{cluster6} show the clustering results for $\lambda=0.1, 0.5$, and $0.9$ using the data in Section \ref{se:simu22} (i).
In Figure \ref{cluster4}, the number of elements in $\mathfrak{D}_{n}$ increases abruptly at $K=6$ for $\lambda=0.1$ and at $K=7$ for $\lambda=0.5$ and $0.9$.
In the $K$-means corresponding to Figure \ref{cluster4}, the number of elements in $\mathfrak{D}_n$ is $38$ for all $\lambda$, which is the same as that for $\lambda=0.2$.

\begin{figure}[t]
\begin{tabular}{c}

\begin{minipage}{0.32 \hsize}
\begin{center}
\includegraphics[scale=0.25]{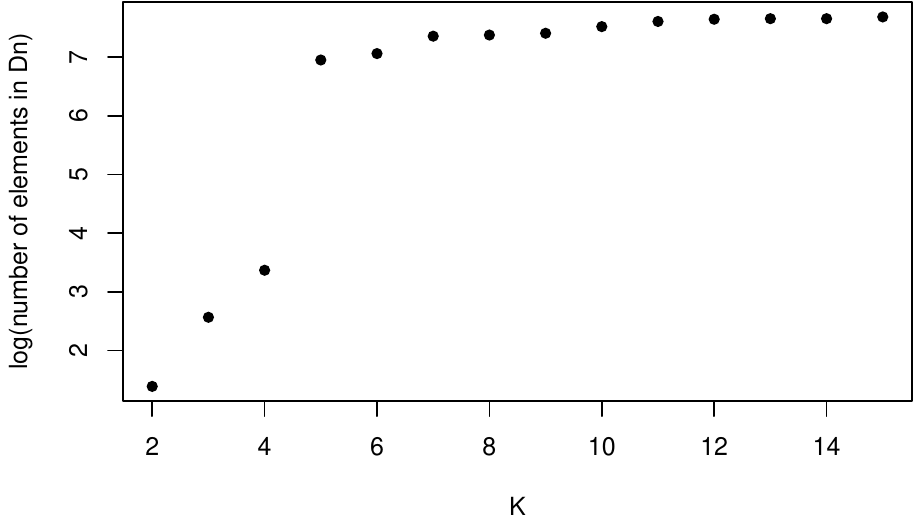}
\end{center}
\end{minipage}

\begin{minipage}{0.32 \hsize}
\begin{center}
\includegraphics[scale=0.25]{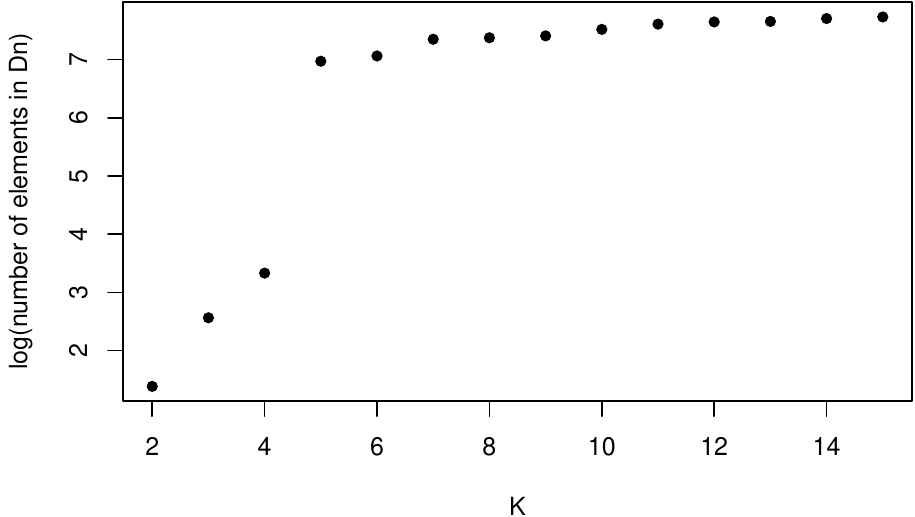}
\end{center}
\end{minipage}

\begin{minipage}{0.32 \hsize}
\begin{center}
\includegraphics[scale=0.25]{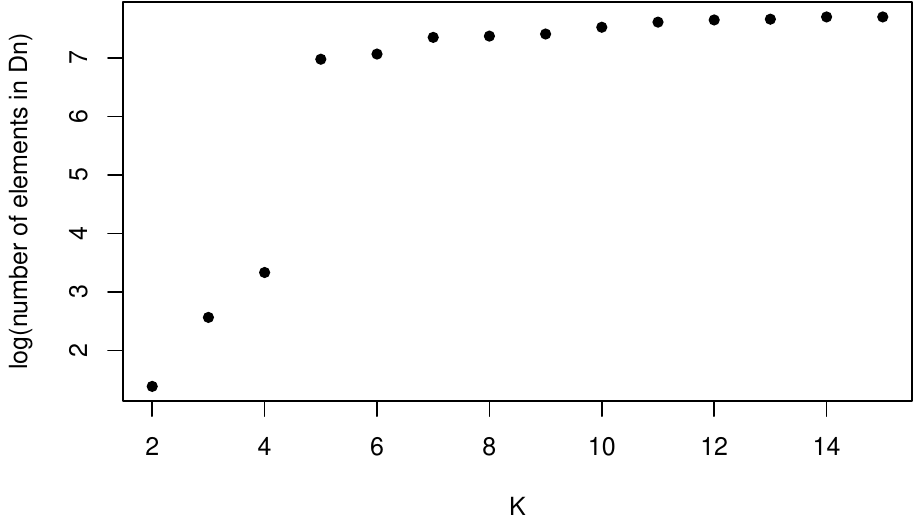}
\end{center}
\end{minipage}

\end{tabular}
\caption{Logarithmic of the number of elements in $\mathfrak{D}_{n}$ in $K$-means for each $K$ using the data in Section \ref{se:simu12} (left: $\lambda=0.1$, center: $\lambda=0.5$, right: $\lambda=0.9$).}
\label{cluster3}
\end{figure}

\begin{figure}[t]
\begin{tabular}{c}

\begin{minipage}{0.32 \hsize}
\begin{center}
\includegraphics[scale=0.25]{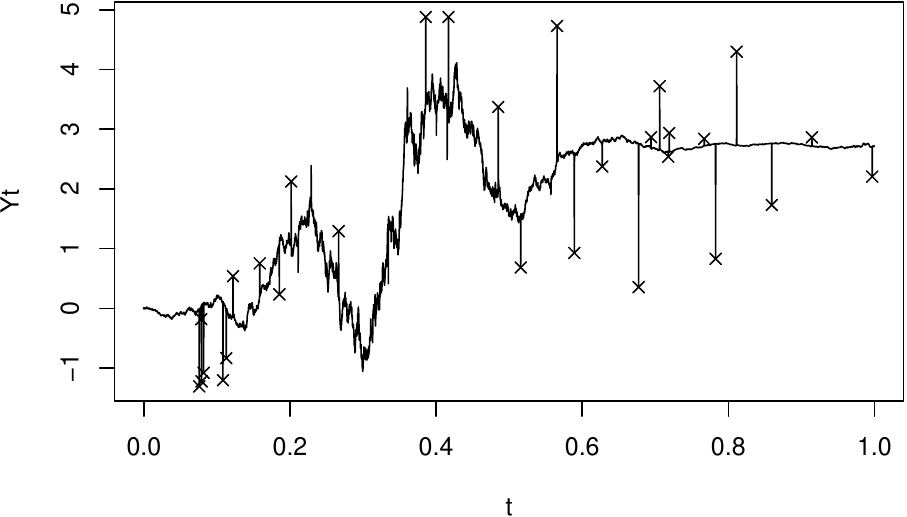}
\end{center}
\end{minipage}

\begin{minipage}{0.32 \hsize}
\begin{center}
\includegraphics[scale=0.25]{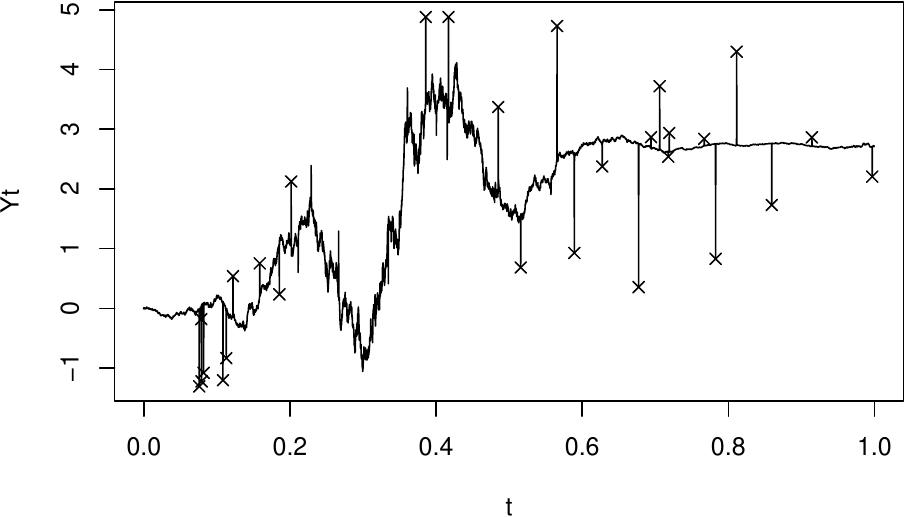}
\end{center}
\end{minipage}

\begin{minipage}{0.32 \hsize}
\begin{center}
\includegraphics[scale=0.25]{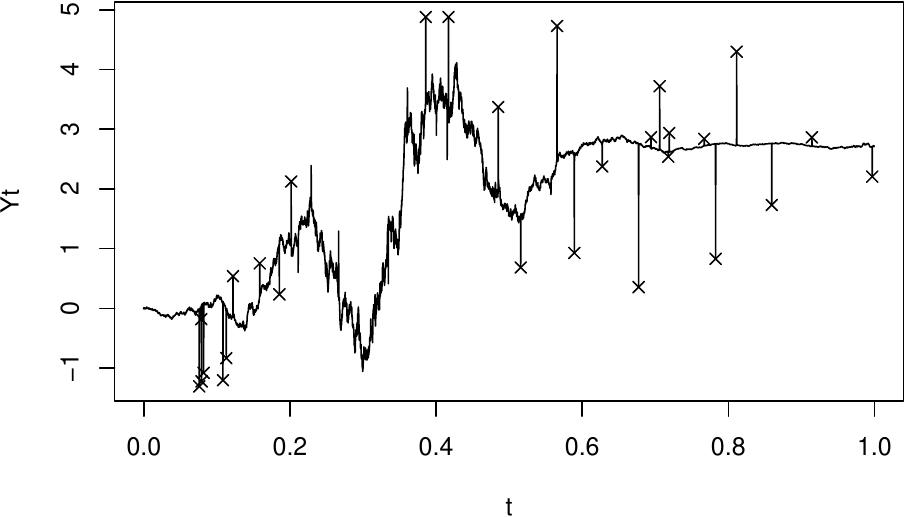}
\end{center}
\end{minipage}

\end{tabular}
\caption{The path of data from Section \ref{se:simu12} used for clustering with the results of $K$-means (left: $\lambda=0.1$, center: $\lambda=0.5$, right: $\lambda=0.9$). The cross mark points mean elements of $\mathfrak{D}_{n}$.}
\label{cluster5}
\end{figure}


\begin{figure}[t]
\begin{tabular}{c}

\begin{minipage}{0.32 \hsize}
\begin{center}
\includegraphics[scale=0.25]{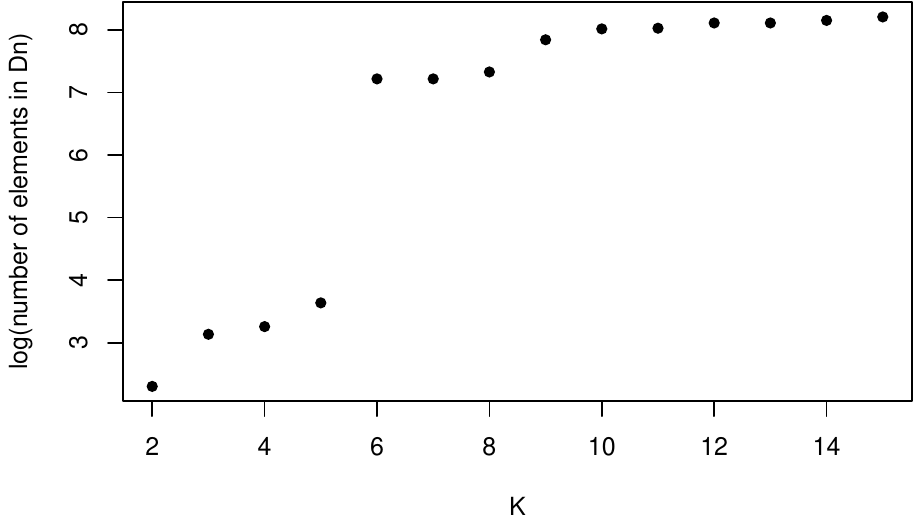}
\end{center}
\end{minipage}

\begin{minipage}{0.32 \hsize}
\begin{center}
\includegraphics[scale=0.25]{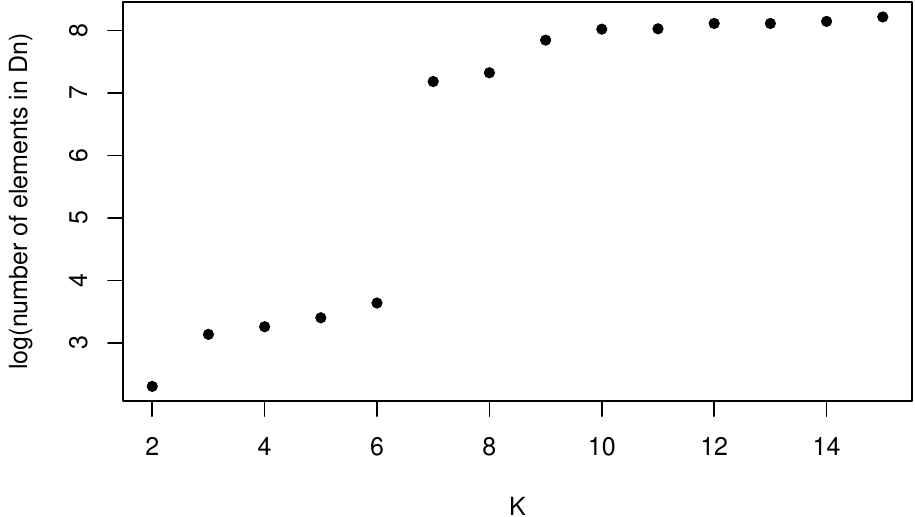}
\end{center}
\end{minipage}

\begin{minipage}{0.32 \hsize}
\begin{center}
\includegraphics[scale=0.25]{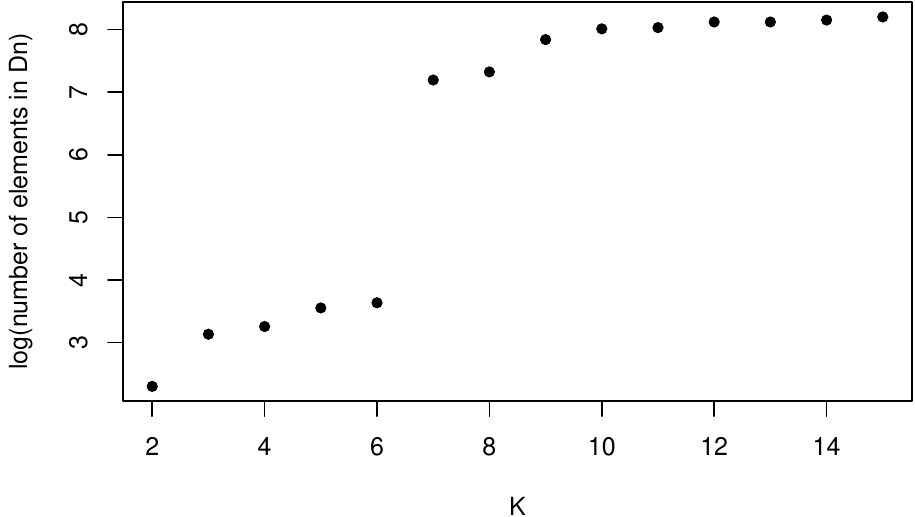}
\end{center}
\end{minipage}

\end{tabular}
\caption{Logarithmic of the number of elements in $\mathfrak{D}_{n}$ in $K$-means for each $K$ using the data in Section \ref{se:simu22} (i)}
\label{cluster4}
\end{figure}

\begin{figure}[t]
\begin{tabular}{c}

\begin{minipage}{0.32 \hsize}
\begin{center}
\includegraphics[scale=0.25]{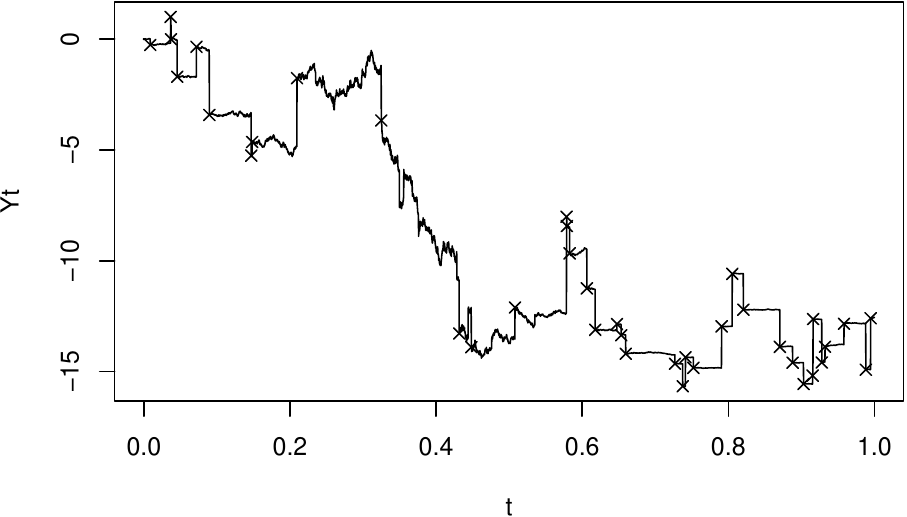}
\end{center}
\end{minipage}

\begin{minipage}{0.32 \hsize}
\begin{center}
\includegraphics[scale=0.25]{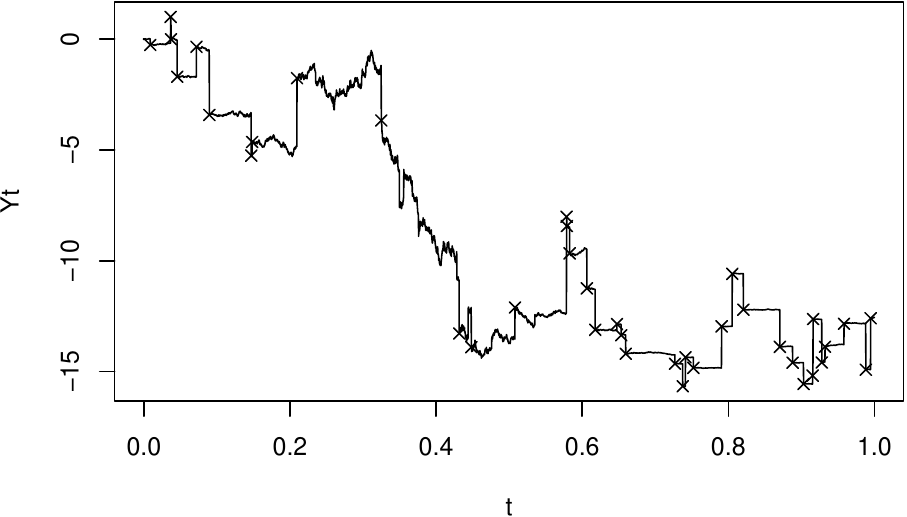}
\end{center}
\end{minipage}

\begin{minipage}{0.32 \hsize}
\begin{center}
\includegraphics[scale=0.25]{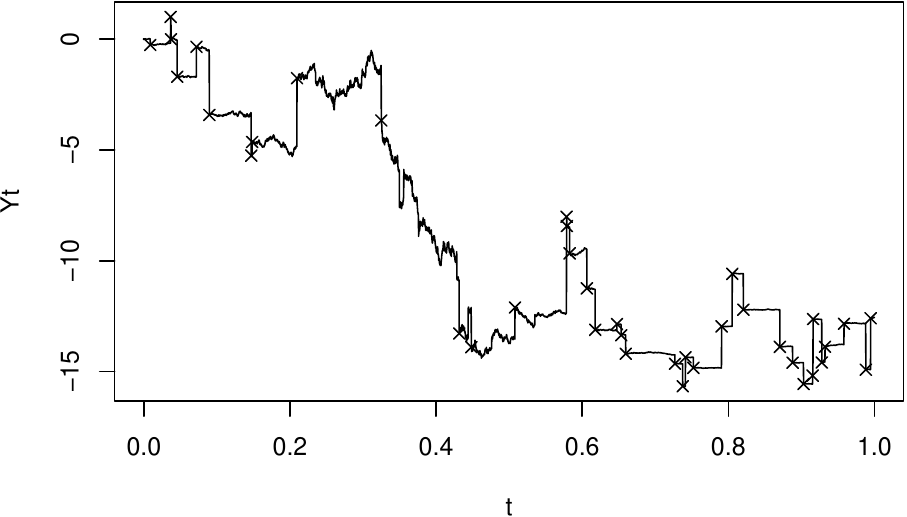}
\end{center}
\end{minipage}

\end{tabular}
\caption{The path of data from Section \ref{se:simu22} (i) used for clustering with the results of $K$-means (left: $\lambda=0.1$, center: $\lambda=0.5$, right: $\lambda=0.9$). The cross mark points mean elements of $\mathfrak{D}_{n}$.}
\label{cluster6}
\end{figure}

\bigskip

\noindent
\textbf{Acknowledgements.} 
The authors are grateful to the two reviewers, whose comments led to substantial improvement of the presentation. We also thank Mr Yoshiyuki Miyamori for his helpful comments on an earlier version.
This work was partially supported by JST CREST Grant Number JPMJCR2115 and JSPS KAKENHI Grant Number 23K22410 (HM), Japan.

\bigskip 
\bibliographystyle{abbrv} 
\bibliography{SE_bibs}

\end{document}